%% file: Sos_submodularity.tex
\documentclass[11pt]{article}

\usepackage[letterpaper,margin=1in]{geometry}
\usepackage{amsmath,amssymb,amsthm}
\usepackage{mathtools}
\usepackage{bm}
\usepackage{graphicx}
\usepackage{tikz}
\usepackage{xcolor}
\usepackage{multirow}
\usepackage{nicefrac}
\usepackage[shortlabels]{enumitem}
\usepackage{float}

\usepackage{natbib}
\bibpunct[, ]{(}{)}{,}{a}{}{,}%
\def\newblock{\ }%

\usepackage{hyperref}
\hypersetup{
   colorlinks   = true,
   urlcolor     = blue,
   linkcolor    = blue,
   citecolor    = blue,
   naturalnames = true,
}

\floatstyle{ruled}
\newfloat{algorithm}{tbp}{loa}
\floatname{algorithm}{Algorithm}

\newcommand{\revision}[1]{{#1}}
\newcommand{\minorrevision}[1]{{#1}}

\theoremstyle{plain}
\newtheorem{theorem}{Theorem}
\newtheorem{lemma}{Lemma}
\newtheorem{proposition}{Proposition}
\newtheorem{corollary}{Corollary}

\theoremstyle{definition}
\newtheorem{definition}{Definition}
\newtheorem{remark}{Remark}
\newtheorem{example}{Example}

\makeatletter
\def\proof#1{\par\smallskip\noindent{\itshape #1}\enskip\ignorespaces}
\def\endproof{\par\smallskip}
\makeatother
\providecommand{\Halmos}{\hfill\ensuremath{\square}}

\DeclareMathOperator*{\argmin}{arg\,min}
\providecommand{\coloneq}{\mathrel{:=}}

\begin{document}

\title{Sum of Squares Submodularity}
\author{Anna Deza\thanks{Mitsubishi Electric Research Laboratories (MERL), \texttt{deza@merl.com}.}
\and Georgina Hall\thanks{Decision Sciences, INSEAD, \texttt{georgina.hall@insead.edu}.}}
\date{}
\maketitle

\begin{abstract}
We introduce the notion of \emph{$t$-sum of squares (sos) submodularity}, which is a hierarchy, indexed by $t$, of sufficient algebraic conditions for certifying submodularity of set functions. We show that, for fixed $t$, each level of the hierarchy can be verified via a semidefinite program of size polynomial in $n$, the size of the ground set of the set function. This is particularly relevant given existing hardness results around testing whether a set function is submodular \citep{crama1989recognition}. We derive several equivalent algebraic characterizations of $t$-sos submodularity and identify  submodularity-preserving operations that also preserve $t$-sos submodularity. We further present a complete classification of the cases for which submodularity and $t$-sos submodularity coincide, as well as examples of $t$-sos-submodular functions. We demonstrate the usefulness of $t$-sos submodularity through three applications: (i) a new convex approach to submodular regression, involving minimal manual tuning; (ii) a systematic procedure to derive lower bounds on the submodularity ratio in approximate submodular maximization, and (iii) improved difference-of-submodular decompositions for difference-of-submodular optimization. Overall, our work builds a new bridge between discrete optimization and real algebraic geometry by connecting sum of squares--based algebraic certificates to a fundamental discrete structure, submodularity.
\end{abstract}


\medskip


\input{1_Introduction}

\input{2_Preliminaries}

\input{3_Theory}

\input{4_Applications}

\input{6_Conclusion}

\section{Code and Data Disclosure}\label{sec:Code and Data Disclosure}
The code and data to support the numerical experiments in this paper can be found \href{https://github.com/SOSSubmodularity/SOSSubmodularity_Revision}{here}.

\section*{Acknowledgments}
We are very grateful to Pablo Parrilo for his helpful feedback. We also gratefully acknowledge the constructive comments of the Department Editor of the Optimization Area at \emph{Operations Research}, Samuel Burer, an anonymous Associate Editor, and two anonymous reviewers that have greatly improved the paper.

\renewcommand{\refname}{References}

\input{references-main.tex}
\newpage

\appendix

\input{7_Appendix}

\renewcommand{\refname}{E-Companion References}

\input{references-ec.tex}
\end{document}

%% file: 1_Introduction.tex
\section{Introduction} \label{sec:intro}

Submodularity is a key structural property of set functions, encoding a notion of \emph{diminishing returns} \citep{edmonds2003submodular}. Indeed, a set function is submodular if adding an element to a smaller set increases the function value more than adding the same element to a larger set. As diminishing-return phenomena appear widely in real-world systems, submodular functions feature in a variety of areas, including in economics \citep{topkis1998supermodularity}, machine learning \citep{bilmes2022submodularity}, information theory \citep{fujishige2005submodular}, electrical networks \citep{narayanan1997submodular}, graph theory \citep{frank1993submodular}, and operations research \citep{queyranne1995scheduling, atamturk2020submodularity}. In addition to being ubiquitous, submodular functions also possess attractive computational properties, which have contributed to their significance in optimization theory. Unconstrained minimization of a submodular function, for example, can be done in polynomial time, even though its ground set is of exponential size \citep{schrijver2000combinatorial}. Furthermore, maximization of a non-decreasing submodular function under cardinality constraints, while NP-hard, can be done via the greedy algorithm with a well-known approximation guarantee of $\left(1-\frac{1}{e}\right)$ \citep{nemhauser1978analysis}. Due to its ubiquity and its appealing optimization properties, submodularity is sometimes referred to as ``discrete convexity''.

Unfortunately, in an echo of existing results for testing convexity \citep{ahmadi2013np}, \minorrevision{testing whether a set function of degree larger than or equal to $4$ is submodular is computationally intractable~\citep{crama1989recognition}; see Sections \ref{subsec:set.funcs.extensions} and \ref{subsec:np.hard.submod} for the definition of the degree of a set function and the exact hardness statement.} It thus becomes relevant to define sufficient conditions for submodularity, which are tractable to check. This paper introduces a hierarchy, indexed by $t$, of such conditions, referred to as \emph{$t$-sum of squares (sos) submodularity}. We provide a comprehensive theoretical study of $t$-sos submodularity and showcase its usefulness through three different applications. In more detail, our theoretical contributions are the following:
\begin{itemize}
\item In Section \ref{subsec:def.t.sos}, we define $t$-sos-submodularity as a hierarchy, indexed by $t$, of sufficient algebraic conditions for submodularity. We show that any $t$-sos-submodular function is submodular. We also show that testing whether a set function, defined over a set of size $n$, is $t$-sos-submodular amounts to solving a semidefinite program of size polynomial in $n$, for fixed $t$. We further show that any submodular function is $t$-sos-submodular, for large enough $t$.
\item In Section~\ref{subsec:equiv.charac.sos.submod}, we provide several equivalent characterizations of $t$-sos submodularity, corresponding to algebraic counterparts of standard definitions of submodularity. 
\item In Section~\ref{subsec:ops.preserving.sos.submod}, we discuss submodularity-preserving operations that also preserve $t$-sos submodularity. 
\item In Section \ref{subsec:gap.sos.submod.submod}, we present a complete classification, based on degree, of the cases where $t$-sos submodularity and submodularity coincide. In line with the above computational complexity result, we show that when the degree of the set function is less than 4, there is a small $t$ for which submodularity and $t$-sos submodularity coincide. In contrast, we provide an example of a degree-4 submodular set function which is only t-sos-submodular for very large $t.$ We further give examples of canonical submodular functions which are $t$-sos-submodular, for known $t$.
\end{itemize}
We also investigate application areas of $t$-sos submodularity:
\begin{itemize}
\item In Section \ref{subsec:regression}, we propose a new approach to \emph{submodular regression}, which involves fitting $t$-sos-submodular functions to data. We do this by solving a semidefinite program, whose only hyperparameters are the degree $k$ of the polynomial to fit and $t$. This contrasts with existing approaches which require extensive manual tuning and/or heuristics to do the fitting. We further provide an empirical comparison of our approach against two methods in the literature and show that we outperform both approaches in terms of generalization error on noisy data.
\item In Section \ref{subsec:approx.submod}, we consider the problem of maximizing, under cardinality constraints, approximately submodular functions. In this setting, the greedy algorithm provides an approximation guarantee of $\left( 1-\frac{1}{e^{\gamma^*}}\right)$, where $\gamma^*$ is the \emph{submodularity ratio} of the approximately submodular function \citep{das2018approximate,bian2017guarantees}. Unfortunately, $\gamma^*$ is NP-hard to compute. Current approaches to find lower bounds on $\gamma^*$ are ad hoc and rely on problem-specific insights. We propose a systematic approach to derive lower bounds on $\gamma^*$ using the framework of $t$-sos submodularity, and illustrate the power of this approach on an example from \citep{bian2017guarantees}.
\item In Section \ref{subsec:diff.submod}, we focus on \emph{difference of submodular (ds) optimization}. These are unconstrained optimization problems where the objective function is a set function, given as the difference of two submodular functions. Solving  ds programs can be done through the submodular-supermodular procedure (SSP) \citep{narasimhan2005submodular}. As any set function has an infinite number of ds decompositions, the question we address here is whether a specific decomposition can improve the performance of the SSP. We propose one such decomposition here and show that we are able to improve numerically on the quality of the solution found by the SSP by using this decomposition as opposed to a standard one.
\end{itemize}
We further provide, in Section~\ref{sec:prelims}, a brief review of some concepts that are central to the paper.

%% file: 2_Preliminaries.tex
\section{Preliminaries} \label{sec:prelims}

To make this paper self-contained, we review some topics relating to submodularity, set functions, and sum of squares polynomials, which are used throughout the paper. \revision{Readers familiar with the sum of squares literature can skip Section \ref{subsec:ideals.varieties.sos}.}

\subsection{Submodular functions} \label{subsec:submodular}

Let $\Omega$ be a set of cardinality $n$ and let $2^{\Omega}$ be its power set, that is, the set of all subsets of $\Omega$. A set function is a function $f: 2^{\Omega} \rightarrow \mathbb{R}.$ An example of a set function is the \emph{cut function} of a graph $G=(V,E)$ for $|V|=n$, which maps a subset $S$ of nodes to the number of edges between $S$ and $V\backslash S$. Submodularity is a property of set functions, with many equivalent definitions, which we give now.

\begin{definition}\cite[Proposition 2.1]{nemhauser1978analysis}. \label{def:f.submodular}
A set function $f:2^{\Omega} \rightarrow \mathbb{R}$ is \emph{submodular} if any one of the seven equivalent conditions below hold:
\begin{enumerate}[$(i)$]
\item $f(S)+f(T) \geq f(S\cup T)+f(S\cap T), ~\forall S,T \subseteq \Omega.$
\item $f(S\cup \{i\}) -f(S) \geq f(T \cup \{i\})-f(T),~\forall S,T \subseteq \Omega$ such that $S \subseteq T$ and $i\notin T$.
\item $f(S \cup \{i\})-f(S) \geq f(S \cup \{i,j\})-f(S \cup \{j\}),~\forall S \in \Omega,~ i,j \in \Omega \backslash S,~i \neq j$.
\item $f(T) \leq f(S)+\sum_{i \in T \backslash S} (f(S\cup \{i\})-f(S))-\sum_{i \in S \backslash  T} (f(S\cup T)-f(S\cup T \backslash  \{i\})),~\forall S,T \subseteq \Omega.$
\item $f(T) \leq f(S)+\sum_{i \in T \backslash S} (f(S \cup \{i\})-f(S)),~\forall S \subseteq T \subseteq \Omega$.
\item $f(T) \leq f(S)-\sum_{i \in S \backslash T} (f(S)-f(S \backslash \{i\})), ~\forall T \subseteq S  \subseteq \Omega$.
\item $f(T) \leq f(S) -\sum_{i \in S \backslash T} (f(S)-f(S\backslash i))+\sum_{i \in T \backslash S} (f(S\cap T \cup \{i\})-f(S \cap T)),~\forall S,T \subseteq  \Omega.$
\end{enumerate}
\end{definition}

The cut function is a submodular function (see Section \ref{subsec:set.funcs.extensions}). The functions for which the inequalities in the definitions hold with equality are \emph{modular functions}. Obviously, any modular function is submodular, and it can be shown that any modular function is of the form $\sum_{i \in \Omega} c_i$, for $\{c_i\}_{i \in \Omega} \in \mathbb{R}$.

Of all the definitions given above, $(i)$ is typically used as the ``canonical'' definition, despite $(ii)$ being the definition that encodes the diminishing returns property, as discussed in the introduction. Definition $(iii)$ is widely known as well. Definitions $(v)$ and $(vi)$ often appear in proofs in the submodularity literature (e.g., in the correlation gap literature \citep{rubinstein2017combinatorial}; see also the discussion after Proposition 2.6 in \citep{bach2013learning} and Section~\ref{subsec:approx.submod}) while definitions $(iv)$ and $(vii)$ are much less used.

\subsection{Multilinear extensions of set functions} \label{subsec:set.funcs.extensions}

There is a one-to-one correspondence between set functions over $\Omega$ and multilinear polynomials in $n$ variables, through the \emph{multilinear extension} of a set function. Before defining this polynomial, we introduce some relevant notation. Let $x\mathrel{\mathop{:}}=(x_1,\ldots,x_n)$ and denote by $\mathbb{R}[x]$ the set of polynomials in variables $x$ with coefficients in $\mathbb{R}.$ For any $p \in \mathbb{R}[x],$ we use $\deg(p)$ to denote the degree of $p$. A polynomial in $\mathbb{R}[x]$ is said to be multilinear if all of its monomials contain only variables that are to the power of at most 1. For example, the polynomial $x_1x_2+x_2-1$ is multilinear, whereas $x_1^2+x_1x_2$ is not. By definition, the degree of any multilinear polynomial in $\mathbb{R}[x]$ is less than or equal to $n$. Furthermore, the derivative of a multilinear polynomial with respect to  coordinate $i$ does not depend on $x_i$. Throughout, we use $1_S$ for the indicator function of a set $S.$ We now define the multilinear extension of a set function.
\begin{definition} \label{def:mle}
The multilinear extension (MLE) $F$ of a set function $f:2^{\Omega} \rightarrow \mathbb{R}$  is the unique multilinear polynomial in $\mathbb{R}[x]$ satisfying $F(1_S)=f(S)$ for all $S\subseteq \Omega.$
\end{definition}

In the remainder, we maintain the convention of using lowercase letters for set functions (e.g., $f,g$) and uppercase letters (e.g., $F,G$) for their MLEs. Existence of the MLE is straightforward by noting that the polynomial 
$$F(x)=\sum_{T \subseteq \Omega} f(T) \prod_{i \in T} x_i \prod_{i \notin T} (1-x_i)$$
is multilinear. One can equivalently write (see, e.g., \citep{grabisch2000equivalent})
$$F(x)=\sum_{T \subseteq \Omega} a(T) \prod_{ i \in T} x_i, \text{ where } a(T)=\sum_{S \subseteq T} (-1)^{t-s} f(S).$$
\revision{Here, $t$ is the cardinality of $T$ and $s$ that of $S$.} A proof of uniqueness can be found in \citep{owen1972multilinear}. This result enables us to define the \emph{degree}, $\deg(f)$, of a set function $f$, which is none other than the degree of its multilinear extension. As an example, one can check that the MLE of the cut function introduced above is $$F(x_1,\ldots,x_n)=\sum_{1 \leq i <j \leq n }A_{ij} (x_i-2x_ix_j+x_j),$$ where $A_{ij}$ is entry $(i,j)$ of the adjacency matrix $A$ of $G$. Therefore, the cut function has degree~$2$.

The equivalence between set functions and multilinear polynomials implies that Definition \ref{def:f.submodular}, which involves a set function $f$, has a counterpart involving its MLE, $F$. For $x,y \in \mathbb{R}^n$, we denote by $x \circ y$ the component-wise multiplication of $x$ and $y$, that is, $x \circ y=(x_1y_1,\ldots,x_ny_n).$

\begin{proposition} \label{prop:submod.mle}
Let $f:2^{\Omega} \rightarrow \mathbb{R}$ be a set function and let $F$ be its multilinear extension. We have that $f$ is submodular if and only if any one of the following seven equivalent conditions holds:
\begin{enumerate}[(i')]
\item $F(x)+F(y) \geq F(x+y-x\circ y)+F(x\circ y), \forall x,y \in \{0,1\}^n$.
\item $\frac{\partial F(x)}{\partial x_i} \geq \frac{\partial F(y)}{\partial x_i}, ~\forall x,y \in \{0,1\}^n$ such that $x \leq y$,$~\forall i=1,\ldots,n$.
\item $\frac{\partial^2 F(x)}{\partial x_i \partial x_j} \leq 0, \forall x \in \{0,1\}^n, \forall i,j=1,\ldots,n.$
\item $F(x)-F(y)+\sum_{i=1}^n y_i(1-x_i) \cdot \frac{\partial F(x)}{\partial x_i}-\sum_{i=1}^n x_i(1-y_i) \cdot \frac{\partial F(x+y-x\circ y)}{\partial x_i} \geq 0,~\forall x,y \in \{0,1\}^n$.
\item $F(x)-F(y)+\sum_{i=1}^n (y_i-x_i) \cdot \frac{\partial F(x)}{\partial x_i} 
\geq 0,~\forall x,y \in \{0,1\}^n \text{ such that } x \leq y.$
\item $F(y)-F(x)-\sum_{i=1}^n (y_i-x_i)\cdot \frac{\partial F(y)}{\partial x_i} \geq 0, ~\forall x,y \in \{0,1\}^n \text{ such that } x \leq y.$
\item $F(y)-F(x)-\sum_{i=1}^n y_i(1-x_i) \frac{\partial F(y)}{\partial x_i}+\sum_{i=1}^n x_i(1-y_i)\frac{\partial F(x \circ y)}{\partial x_i} \geq 0, ~\forall x,y \in \{0,1\}^n.$
\end{enumerate}
\end{proposition}

The proof of this proposition is straightforward: it involves showing that $(i')$ (resp. $(ii')$-$(vii')$) is equivalent to $(i)$ (resp. $(ii)$-$(vii)$), by translating set function operations into algebraic ones via the indicator function, and making use of a property of the  derivative of a multilinear polynomial given in Appendix \ref{appendix:background.material}. It is ommitted for brevity.

Conditions $(ii')-(iii')$ and, by extension, $(ii)-(iii)$ are referred to as the first and second order characterizations of submodularity, due to the presence of the first and second-order derivatives. They are folklore in the submodularity literature \citep{topkis1978minimizing}. A zero-th order characterization of submodularity exists as well, though it does not quite correspond to $(i')$. It is given by 
$$F(x)+F(y) \geq F(\max\{x,y\})+F(\min \{x,y\}), \forall x,y \in \{0,1\}^n.$$
We replace this condition by $(i')$ (which seems to be new, though quite elementary) in order to operate within the set of polynomials. Indeed, $F(\max \{x,y\})+F(\min \{x,y\})$ is not a polynomial, and using the zero-th order characterization from the literature would prevent us from making use of the sum of squares techniques central to this paper. Conditions $(iv')-(vii')$ correspond to conditions $(iv)-(vii)$ and, to the best of our knowledge, have not appeared in the literature as is.

\begin{remark}
All conditions in Proposition \ref{prop:submod.mle} are required to hold over $\{0,1\}^n$ or a subset of $\{0,1\}^n$, in line with their counterparts in $(i)$-$(vii).$ Interestingly, replacing the sets $\{0,1\}^n$ by the intervals $[0,1]^n$ does not lead to stronger conditions. In fact, it can be shown using \cite[Theorem 13.8]{crama2011boolean} that both sets of conditions are equivalent when $F$ is multilinear. However, their translations into algebraic certificates are not. In the $\{0,1\}^n$ setting, which we elect to work in, we provide sufficient conditions for submodularity of \emph{set functions}. If we were to replace $\{0,1\}^n$ by $[0,1]^n$, we would be providing sufficient conditions for continuous notions of submodularity, such as DR-submodularity \citep{bach2019submodular,bian2020continuous}. Among other advantages, working in $\{0,1\}^n$ enables us to leverage sum of squares techniques on quotient rings (see Section \ref{subsec:ideals.varieties.sos}), which are not applicable in the continuous setting and which lead to a simpler and cheaper hierarchy of sufficient conditions for submodularity.  
\end{remark}

\subsection{Hardness of testing submodularity} \label{subsec:np.hard.submod}

In this section, we review known complexity results relating to the decision problem $SUBMOD_d:$ given an integer $d$ and a set function $f$ of degree $d$, decide if $f$ is submodular. We place ourselves in the standard Turing model of computation where the input to every problem instance must be defined by a finite number of bits \citep{sipser1996introduction}, and emphasize that \revision{$f$ is given to us via its multilinear extension, that is, the input to the problem are the (rational) coefficients of the multilinear polynomial $F$ in the standard multilinear monomial basis.} Other complexity models are studied in \citep{seshadhri2014submodularity}. 

\begin{proposition}[\cite{gallo1989supermodular,crama1989recognition,billionnet1985maximizing}]  \label{prop:np.hardness}
The decision problem $SUBMOD_d$ is co-NP-complete for any fixed $d\geq 4$. It can be solved in polynomial time for $d=1,2,3.$
\end{proposition}
Both \citep{gallo1989supermodular} and \citep{crama1989recognition} show co-NP-completeness of $SUBMOD_d$: the former uses a reduction from unconstrained 0-1 quadratic minimization (which makes the problem strongly NP-hard), whereas the latter uses a reduction from partition (which only gives weak NP-hardness). The fact that $SUBMOD_d$ can be solved in polynomial time for $d=1,2,3$ is shown in \citep{billionnet1985maximizing}, and also as a special case of our results (see Section \ref{subsec:gap.sos.submod.submod}).

\subsection{Ideals, varieties, and sum of squares polynomials} \label{subsec:ideals.varieties.sos}

We introduce a couple of concepts from algebraic geometry that are of use in the rest of the paper.

\begin{definition}
The \emph{ideal} generated by polynomials $p_1,\ldots,p_r \in \mathbb{R}[x]$ and denoted by $\langle p_1,\ldots,p_r \rangle$ is the set $I[x]=\{\sum_{i=1}^r h_i p_i~|~h_i \in \mathbb{R}[x]\}$.
\end{definition}
For convenience, the variables that are involved in $p_1,h_1,\ldots,p_r,h_r$ appear explicitly through the notation $I[x]$. We define here one ideal, which appears throughout the paper:
\begin{equation} \label{eq:def.ideals}
I_2[x] \mathrel{\mathop{:}}= \langle x_1^2-x_1,\ldots,x_n^2-x_n \rangle.
\end{equation}

\begin{definition}
The \emph{real variety} of an ideal $I[x]$ is the set $$\mathcal{V}_{\mathbb{R}}(I[x])=\{x \in \mathbb{R}^n~|~p_i(x)=0,\forall i=1,\ldots,r\}.$$
\end{definition}
For example, $\mathcal{V}_{\mathbb{R}}(I_2[x])=\{0,1\}^n$. For two polynomials $p,q \in \mathbb{R}[x]$ and an ideal $I[x]$, if $p-q \in I[x]$, then $p(s)=q(s)$ for all $s \in \mathcal{V}_{\mathbb{R}}(I[x])$. If $p-q \in I[x]$, then $p$ and $q$ are said to be \emph{congruent mod $I[x]$}, written as $p \equiv q \mod{I[x]}.$ 
We are now ready to define the concept of $t$-sos mod $I[x]$ which is central to this paper.
\begin{definition} \label{def:t-sos}
Let $p \in \mathbb{R}[x]$ and $I[x]$ be a given ideal. We say that $p$ is $t$-sum of squares mod $I[x]$ or $t$-sos mod $I[x]$, if $$p \equiv \sum_{r=1}^m q_r^2 \mod{I[x]}$$
where $q_r \in \mathbb{R}[x], \deg(q_r) \leq t, ~\forall r=1,\ldots,m.$ 
\end{definition}
It is quite straightforward to see that if $p$ is $t$-sos mod $I[x]$ for some $t$, then $p(s) \geq 0, \forall s \in \mathcal{V}_{\mathbb{R}}(I).$ Indeed, as it a sum of squares, $\sum_{r} q_r^2(x) \geq 0, \forall x \in \mathbb{R}^n$ and $p(s)=\sum_r q_r^2(s),\forall s \in \mathcal{V}_{\mathbb{R}}(I[x]).$ 

Under certain conditions on $I[x]$, checking whether $p$ is $t$-sos mod $I[x]$ can be done using semidefinite programming \cite[Chapter 3.3.5]{blekherman2012semidefinite}. In the case where $I[x]=I_2[x]$, these conditions are met: the semidefinite program obtained corresponds to searching for a positive semidefinite (psd) matrix $Q$ such that 
$$p(x) \equiv z_t(x)^TQ z_t(x) \mod{I_2[x]},$$
where $z_t(x)=[1,x_1,\ldots,x_n,x_1x_2,\ldots]$ is the standard basis of square-free monomials of degree up to $t$. We make use of other ideals in Section \ref{subsec:equiv.charac.sos.submod}, for which these conditions also hold.

%% file: 3_Theory.tex
\section{Theory of $t$-sos-submodular polynomials} \label{sec:theory}

In this section, we define $t$-sos submodularity and show some of its theoretical properties.

\subsection{Definition of $t$-sos submodularity}\label{subsec:def.t.sos}

\revision{Proposition \ref{prop:np.hardness} motivates the introduction of an algebraic sufficient condition for submodularity that can be checked efficiently. We introduce such a condition in this section, based on sum-of-squares certificates, which we call $t$-sos submodularity. In principle, any of the equivalent characterizations of submodularity given in Proposition \ref{prop:submod.mle} could serve as the foundation for such a definition by replacing each nonnegativity inequality with its $t$-sos analog. Our definition is based on condition $(iii')$, a choice that is justified in Section \ref{subsec:equiv.charac.sos.submod}, where we show that the sos conditions arising from $(i')$–$(vii')$ are equivalent for appropriate values of $t$, and that the formulation based on $(iii')$ yields the smallest semidefinite programming representation. } 

Recall the ideal $I_2[x]$ defined in \eqref{eq:def.ideals}. For $i=1,\ldots,n$, we introduce the notation $x_{-i}$, for $x \in \mathbb{R}^n$ to mean the vector $x$ from which the $i^{th}$ component has been dropped, resulting in a vector in $\mathbb{R}^{n-1}$. We likewise let $x_{-i,j} \in \mathbb{R}^{n-2}$ for $i,j=1,\ldots,n,~i\neq j$, be the vector $x$ from which the $i^{th}$ and $j^{th}$ components have been removed. In the next definition, we work with $I_2[x_{-i,j}]$, which is generated by the polynomials $x_1^2-x_1,\ldots,x_n^2-x_n$ from which $x_i^2-x_i$ and $x_j^2-x_j$ have been removed.

\begin{definition} \label{def:t-sos-submod}
Let $f:2^{\Omega} \rightarrow \mathbb{R}$ be a set function, let $F$ be its multilinear extension, and let $t \in \mathbb{N}$. We say that $f$ is $t$-sos submodular if 
$-\frac{\partial^2 F(x)}{\partial x_i \partial x_j} \text{ is $t$-sos modulo $I_2[x_{-i,j}]$}, \forall i,j \in \{1,\ldots,n\}, i\neq j$, i.e.,
$$-\frac{\partial^2 F(x)}{\partial x_i \partial x_j} \equiv \sum_{r} q_r^2(x_{-i,j}) \mod{I_2[x_{-i,j}]}, \forall i,j \in \{1,\ldots,n\}, i \neq j,$$
where $q_r \in \mathbb{R}[x_{-i,j}]$ and $\deg(q_r)\leq t.$
\end{definition}
For $i,j\in \{1,\ldots,n\}$, $i\neq j$, the definition involves only variables $x_{-i,j}$ as $x_i$ and $x_j$ do not appear in $\partial^2 F(x)/\partial x_i \partial x_j$ as it is a multilinear function. We now show that being $t$-sos submodular is a sufficient condition for being submodular and can be checked efficiently.
\begin{proposition} \label{prop:sufficiency.sos.submod}
Let $t \in \mathbb{N}.$ Any $t$-sos-submodular set function is submodular.
\end{proposition}

\begin{proposition} \label{prop:tractability.sos.submod}
Let $t \in \mathbb{N}.$ Checking whether a set function is $t$-sos-submodular amounts to solving a semidefinite program (SDP) with $\frac{n(n-1)}{2}$ semidefinite constraints, each of size $\sum_{k=0}^t \binom{n-2}{k}.$
\end{proposition}
\revision{Proofs of Propositions~\ref{prop:sufficiency.sos.submod} and 
\ref{prop:tractability.sos.submod} are provided in Appendix~\ref{appendix:def.t.sos}. Testing whether a given function $f$ is $t$-sos for some $t$ can be easily done using JuMP \citep{dunning2017jump} and the SumOfSquares.jl library \citep{legat2020sumofsquares} in Julia as we see now.

\begin{example} \label{ex:test.t.sos}
Consider a degree-$5$ set function $f$ defined on a ground set of size $n=5$, whose multilinear extension $F$ is given by
$$
F(x_1,x_2,x_3,x_4,x_5)=3x_1x_2x_3x_4x_5-4x_3x_1x_2-9x_4x_1x_2-12x_1x_2x_4x_5-4x_1x_2x_3x_5-4x_1x_2x_3x_4+2.
$$
Checking whether $F$ is $t$-sos-submodular for a given $t$ amounts to checking whether $-\frac{\partial^2 F(x)}{\partial x_i x_j}$ is $t$-sos modulo $I_2[x_{-i,j}]$ for all $i,j=1,\ldots,n, i\neq j.$ An implementation of this test for $t=2$ and $t=3$ is given \href{https://mybinder.org/v2/gh/SOSSubmodularity/SOSSubmodularity_Revision/5c878bf0f0c92b63d781a20d07aabe9e26fcf9d1?filepath=ExampleEC1/ExampleEC1.ipynb}{here}. It can be run from a browser via the open-source solver SCS \citep{odowd2021scs}. We find that $F$ is not $2$-sos-submodular (the corresponding problem is infeasible) but it is $3$-sos-submodular, and thus submodular.
\end{example}
Note that the SumOfSquares.jl library converts the sum of squares program described above into an SDP under the hood and then solves this SDP. It is also possible to derive this SDP explicitly. We refer the reader to the proof of Proposition~\ref{prop:tractability.sos.submod} for the general SDP formulation of the problem of testing $t$-sos submodularity of a set function and to Example~\ref{example:SDP} for examples of the SDPs obtained when carrying out the tests from Example~\ref{ex:test.t.sos}.
}

As can be seen from Definition \ref{def:t-sos-submod}, being able to check whether a set function $f$ is $t$-sos submodular assumes that the MLE $F$ of $f$ is given. Further note that for $t,t' \in \mathbb{N}$ such that $t \leq t'$, the set of $t$-sos-submodular functions is included in the set of $t'$-sos-submodular functions. This follows immediately from the fact that any $t$-sos polynomial mod $I[x]$ is a $t'$-sos-polynomial mod $I[x]$ when $t \leq t'$. The next result specifies the range of $t$ to consider for $t$-sos submodularity. 

\begin{proposition} \label{prop:bounds.on.t}
Let $t \in \mathbb{N}$ and let $f:2^\Omega \rightarrow \mathbb{R}$ be a set function of degree $d$. If $t < \lceil \frac{d-2}{2} \rceil$, then $f$ cannot be $t$-sos submodular. When $t=\lceil \frac{n+d-5}{2} \rceil$, the set of $t$-sos-submodular functions is exactly the set of submodular functions of degree $d$.
\end{proposition}
A proof of this result is in Appendix \ref{appendix:def.t.sos}. Thus, in practice, if one wishes to show that a set function $f:2^{\Omega}\rightarrow \mathbb{R}$ of degree $d$ is $t$-sos submodular, one should only consider $ \lceil \frac{d-2}{2} \rceil \leq t \leq \lceil \frac{n+d-5}{2} \rceil$. 

\begin{remark} \label{remark:convexity.submodularity}
As mentioned in the introduction, submodularity has many interesting connections to convexity. In a parallel to Proposition \ref{prop:submod.mle}, convexity also has zero, first and second-order characterizations. Furthermore, as touched upon in the introduction, testing whether a quartic polynomial is convex is an NP-hard problem \citep{ahmadi2013np}, in an echo of Proposition \ref{prop:np.hardness}.  As a consequence, a sufficient, but tractable, condition for convexity based on sum of squares has been introduced in the literature: \emph{sos convexity} \citep{helton2010semidefinite}. Like $t$-sos submodularity, it is based on the second-order characterization of convexity, for similar reasons to the ones we discuss next.
\end{remark}

\begin{remark}
\revision{Proposition \ref{prop:tractability.sos.submod} states that checking whether a set function over a ground set of size $n$ is $t$-sos-submodular amounts to solving a semidefinite program  with order $O(n^2)$ semidefinite constraints, each one of size roughly $(n-2)^t$. To help the reader understand for which values of $(n,t)$ such a problem can be tackled with modern-day solvers, we provide, in Table \ref{tab:solving.times}, the time it takes to solve, for varying $n$ and $t$, the following problem:
\begin{equation} \label{eq:test.times}
\begin{aligned}
&\min_{F \in \mathbb{R}[x_1,\ldots,x_n], ~\deg(F)\leq d, ~\text{F multilinear}} c^T \hat{F}\\
&\text{s.t. } F \text{ is $t$-sos-submodular}, ||\hat{F}||_2^2 \leq 1,
\end{aligned}
\end{equation}
where $\hat{F}$ is a vector containing the coefficients of $F$ and $c$ is a random vector generated following a standard normal distribution. This problem is bounded due to the $2$-norm constraint on the coefficients of $F$ and always has a feasible solution (e.g., $F=0$): it is thus well-defined. Furthermore, its solving time reflects the time it takes to optimize over or test membership to the set of $t$-sos-submodular functions for fixed $n$. We use the solver SCS \citep{odowd2021scs} on an M4 MacBook Pro with 16GB of RAM to solve these problems. Note that these solving times can be further improved by leveraging, e.g., sparsity or symmetry of the underlying SDP formulation; see \citep{majumdar2020recent}.}


\begin{table}[h]
    \centering
    \begin{tabular}{|c|ccc|ccc|cc|}\hline
        $n$ & \multicolumn{3}{c|}{10} & \multicolumn{3}{c|}{15} & \multicolumn{2}{c|}{20} \\ \hline
        $t$ & 1 & 2 & 3 & 1 & 2 & 3 & 1 & 2 \\ \hline
        Runtime (s) & 0.12 & 7.08 & 471 & 0.81 & 90.2 & 32900 & 2.77 & 534 \\ \hline
    \end{tabular}
    \caption{Solving times for \eqref{eq:test.times} for varying values of $n$ and $t$. }
    \label{tab:solving.times}
\end{table}
    
\end{remark}

\subsection{Equivalent characterizations of $t$-sos submodularity} \label{subsec:equiv.charac.sos.submod}

\revision{In Section \ref{subsec:def.t.sos}, we define $t$-sos submodularity using the second-order characterization of submodularity as a starting point. As mentioned there, we could instead have based the definition on any of the characterizations from Proposition~\ref{prop:submod.mle}, which would have led to different sos-based sufficient conditions for submodularity. In this section, we argue that the definition of $t$-sos submodularity, as given in Definition~\ref{def:t-sos-submod}, is the most appropriate one to work with. To argue this, we show that all of the sos-based sufficient conditions arising from Proposition~\ref{prop:submod.mle} are in fact equivalent for appropriate values of $t$; see Theorems \ref{thm:equiv.sos.submod.ii}-\ref{thm:equiv.sos.submod.vii} and Table \ref{tab:summary.equiv.charac} for a summary. We then show that the second-order formulation yields the smallest semidefinite programming representation (Proposition \ref{prop:size}), thereby justifying its use. This section aims to serve a broader purpose, however, than simply justifying the definition of $t$-sos submodularity. In the submodularity literature, the different characterizations in Definition~\ref{def:f.submodular} often prove useful in distinct contexts, and a similar phenomenon arises in the sos setting. For example, in Section~\ref{subsec:approx.submod}, a corner case related to Corollary~\ref{cor:t.sos.approx.submod} is more naturally understood via Theorem~\ref{thm:equiv.sos.submod.v}, which characterizes $t$-sos submodularity through Proposition~\ref{prop:submod.mle} $(v')$, than directly through Definition~\ref{def:t-sos-submod}. Providing the full set of equivalent algebraic characterizations of $t$-sos submodularity therefore ensures that future work can draw upon whichever formulation is most appropriate for a given application.

}

We now proceed with the statement of the results in this section. This requires the introduction of two more ideals, on top of $I_2[x]$ given in \eqref{eq:def.ideals}. These rely on additional variables, the multivariate vectors $y\mathrel{\mathop{:}}=(y_1,\ldots,y_n)$ and $s\mathrel{\mathop{:}}=(s_1,\ldots,s_n).$ We define:
\begin{equation}
\begin{aligned} \label{eq:def.ideals.new}
I_0[x,y,s]&\mathrel{\mathop{:}}=\langle \{x_i^2-x_i,~ x_iy_i-s_i,~x_is_i-s_i, ~y_i^2-y_i,~y_is_i-s_i,s_i^2-s_i \}_{i=1,\ldots,n} \rangle \\
I_1[x,y] &\mathrel{\mathop{:}}=\langle \{x_i^2-x_i, ~x_iy_i-x_i,~ y_i^2-y_i\}_{i=1,\ldots,n}\rangle.
\end{aligned}
\end{equation}
One can easily show that:
$$ \mathcal{V}_{\mathbb{R}}(I_0[x,y,s])=\{(x,y,x \circ y) \in \{0,1\}^{3n} \} \text{ and } \mathcal{V}_{\mathbb{R}}(I_1[x,y])=\{(x,y)\in \{0,1\}^{2n}~|~x \leq y\}.$$
We are now ready to state the main results of this section.

\begin{theorem} \label{thm:equiv.sos.submod.ii}
Let $f:2^{\Omega} \rightarrow \mathbb{R}$ be a set function and let $F$ be its MLE. Then, $f$ is $t$-sos submodular if and only if $\frac{\partial F(x)}{\partial x_i}-\frac{\partial F(y)}{\partial x_i}$ is $(t+1)$-sos modulo $I_1[x_{-i},y_{-i}],~\forall i\in \{1,\ldots,n\}$.
\end{theorem}
This result is an algebraic counterpart of the equivalence between $(ii')$ and $(iii')$ in Proposition~\ref{prop:submod.mle}.

\begin{theorem} \label{thm:equiv.sos.submod.i}
Let $f:2^{\Omega} \rightarrow \mathbb{R}$ be a set function and let $F$ be its MLE. Then, $f$ is $t$-sos submodular if and only if $G(x,y,s) \mathrel{\mathop{:}}=F(x)+F(y)-F(x+y-s)-F(s)$ is $(t+2)$-sos modulo $I_0[x,y,s]$.
\end{theorem}
This result is an algebraic counterpart of the equivalence between $(i')$ and $(iii')$ in Proposition~\ref{prop:submod.mle}.

\begin{theorem} \label{thm:equiv.sos.submod.v}
Let $f:2^{\Omega} \rightarrow \mathbb{R}$ be a set function and let $F$ be its MLE. Then, $f$ is $t$-sos submodular if and only if $H_1(x,y)\mathrel{\mathop{:}}=F(x)-F(y)+\sum_{i=1}^n (y_i-x_i) \cdot \frac{\partial F(x)}{\partial x_i}$ is $(t+2)$-sos modulo $I_1[x,y].$
\end{theorem}
This result is an algebraic counterpart of the equivalence between $(iii')$ and $(v')$ in Proposition~\ref{prop:submod.mle}. 

\begin{theorem} \label{thm:equiv.sos.submod.vi}
Let $f:2^{\Omega} \rightarrow \mathbb{R}$ be a set function and let $F$ be its MLE. Then, $f$ is $t$-sos submodular if and only if $H_2(x,y)\mathrel{\mathop{:}}=F(y)-F(x)-\sum_{i=1}^n (y_i-x_i) \cdot \frac{\partial F(y)}{\partial x_i}$ is $(t+2)$-sos modulo $I_1[x,y].$
\end{theorem}
This result is an algebraic counterpart of the equivalence between $(iii')$ and $(vi')$ in Proposition~\ref{prop:submod.mle}. 

\begin{theorem} \label{thm:equiv.sos.submod.iv}
Let $f:2^{\Omega} \rightarrow \mathbb{R}$ be a set function and let $F$ be its MLE. Then, $f$ is $t$-sos submodular if and only if $G_1(x,y,s)\mathrel{\mathop{:}}=F(x)-F(y)+\sum_{i=1}^n (y_i-s_i) \cdot \frac{\partial F(x)}{\partial x_i}-\sum_{i=1}^n (x_i-s_i)\cdot \frac{\partial F(x+y-s)}{\partial x_i}$ is $(t+2)$-sos modulo $I_0[x,y,s].$
\end{theorem}
This result is an algebraic counterpart of the equivalence between $(iii')$ and $(iv')$ in Proposition~\ref{prop:submod.mle}. 

\begin{theorem} \label{thm:equiv.sos.submod.vii}
Let $f:2^{\Omega} \rightarrow \mathbb{R}$ be a set function and let $F$ be its MLE. Then, $f$ is $t$-sos submodular if and only if $G_2(x,y,s)\mathrel{\mathop{:}}=F(y)-F(x)-\sum_{i=1}^n (y_i-s_i) \frac{\partial F(y)}{\partial x_i}+\sum_{i=1}^n (x_i-s_i)\frac{\partial F(s)}{\partial x_i} $ is $(t+2)$-sos modulo $I_0[x,y,s].$
\end{theorem}
This result is an algebraic counterpart of the equivalence between $(iii')$ and $(vii')$ in Proposition~\ref{prop:submod.mle}. The proof of this result, together with all other proofs for this section, can be found in Appendix~\ref{appendix:equiv.charac.sos.submod}.

\begin{remark}
For the characterizations given in Theorems \ref{thm:equiv.sos.submod.i}, \ref{thm:equiv.sos.submod.iv}, and \ref{thm:equiv.sos.submod.vii}, it might seem more natural, at first glance, to work modulo the ideal:
$$\tilde{I}_0[x,y]\mathrel{\mathop{:}}=\langle x_1^2-x_1,\ldots,x_n^2-x_n, y_1^2-y_1,\ldots,y_n^2-y_n\rangle.$$
However, it can be shown that with this ideal, statements analogous to theorems \ref{thm:equiv.sos.submod.i}, \ref{thm:equiv.sos.submod.iv}, and \ref{thm:equiv.sos.submod.vii} would not be true. One can show, e.g., that the set of functions for which $G_1(x,y,x\circ y)$ is $(t+3)$-sos modulo $\tilde{I}_0[x,y]$ is a strict subset of the set of $t$-sos-submodular functions; but the set of functions for which $G_1(x,y,x\circ y)$ is $(t+4)$-sos modulo $\tilde{I}_0[x,y]$ is a strict superset of the set of $t$-sos-submodular functions. 
\end{remark}

\begin{proposition} \label{prop:size}
Let $t \in \mathbb{N}$. Checking whether a set function $f:2^{\Omega} \rightarrow \mathbb{R}$ is $t$-sos submodular using Theorem \ref{thm:equiv.sos.submod.ii} (resp. Theorems \ref{thm:equiv.sos.submod.i}, \ref{thm:equiv.sos.submod.v}, \ref{thm:equiv.sos.submod.vi}, \ref{thm:equiv.sos.submod.iv}, \ref{thm:equiv.sos.submod.vii}) involves solving a semidefinite program with $n$ (resp. 1) semidefinite constraints of size $\sum_{k=0}^{t+1} \binom{n-1}{k}\cdot 2^k$ (resp. $\sum_{k=0}^{t+2} \binom{n}{k}3^k$, $\sum_{k=0}^{t+2} \binom{n}{k}2^k$, $\sum_{k=0}^{t+2} \binom{n}{k}2^k$,
$\sum_{k=0}^{t+2} \binom{n}{k}3^k$, $\sum_{k=0}^{t+2} \binom{n}{k}3^k$).
\end{proposition}

\begin{table}[h]
    \centering
    \begin{tabular}{|c|c|c|c|c|c|c|}
    \hline
    Function & Degree & Ideal & \# of constr. & Size of constr. & Ref. theorems & Equiv. charac. \\
    \hline \hline
    $-\frac{\partial^2 F(x)}{\partial x_i \partial x_j}$ & t & $I_2[x_{-i,j}]$ & $n(n-1)/2$ & $\sum_{k=0}^t \binom{n-2}{k}$ & Def \ref{def:t-sos-submod} \& Prop \ref{prop:tractability.sos.submod} & $(iii')$ \\
    \hline
    $G(x,y,s)$ & $t+2$ & $I_0[x,y,s]$ & $1$ & $\sum_{k=0}^{t+2} \binom{n}{k} 3^k$ & Th. \ref{thm:equiv.sos.submod.i} \& Prop \ref{prop:size} & $(i')$ \\
     \hline
    $\frac{\partial F(x)}{\partial x_i} - \frac{\partial F(y}{\partial x_i}$ & $t+1$ & $I_1[x_{-i},y_{-i}]$ & $n$ & $\sum_{k=0}^{t+1} \binom{n-1}{k} 2^k$ & Th. \ref{thm:equiv.sos.submod.ii} \& Prop \ref{prop:size} & $(ii')$ \\
     \hline
    $G_1(x,y,s)$ & $t+2$ & $I_0[x,y,s]$ & $1$ & $\sum_{k=0}^{t+2} \binom{n}{k} 3^k$ & Th. \ref{thm:equiv.sos.submod.iv} \& Prop \ref{prop:size} & $(iv')$ \\
     \hline
    $H_1(x,y)$ & $t+2$ & $I_1[x,y]$ & $1$ & $\sum_{k=0}^{t+2} \binom{n}{k} 2^k$ & Th. \ref{thm:equiv.sos.submod.v} \& Prop \ref{prop:size} & $(v')$ \\
     \hline
    $H_2(x,y)$ & $t+2$ & $I_1[x,y]$ & $1$ & $\sum_{k=0}^{t+2} \binom{n}{k} 2^k$ & Th. \ref{thm:equiv.sos.submod.vi} \& Prop \ref{prop:size} & $(vi')$ \\
     \hline
    $G_2(x,y,s)$ & $t+2$ & $I_0[x,y,s]$ & $1$ & $\sum_{k=0}^{t+2} \binom{n}{k} 3^k$ & Th. \ref{thm:equiv.sos.submod.vii} \& Prop \ref{prop:size} & $(vii')$ \\
    \hline
    \end{tabular}
    \vspace{3mm}
    \caption{Comparison of the equivalent characterizations of $t$-sos submodularity;\\ see Theorems \ref{thm:equiv.sos.submod.ii}---\ref{thm:equiv.sos.submod.vii} for definitions of the functions appearing in the table. 
    %
    }
    \label{tab:summary.equiv.charac}
\end{table}

We summarize the results given in this section in Table \ref{tab:summary.equiv.charac}. From this table, assuming that $t \ll n$ and noting that the driver of complexity of an SDP is the size of the SDP constraint, rather than the number of SDP constraints, it becomes clear that our current definition of $t$-sos submodularity is the cheapest one to implement with an SDP constraint of size approx. $(n-2)^t$ (in contrast to, e.g., the equivalent characterizations coming from $(i')$, whose SDP constraint is of size $(3n)^{t+2}$, or from $(ii')$, whose SDP constraint is of size $(2n)^{t+1}$). This is why we define $t$-sos submodularity using the second-order characterization. Further note that $I_2[x_{-i,j}]$ is perhaps one of the most standard ideals to work with, in contrast to $I_1[x,y]$ or $I_0[x,y,s]$, which are a little more involved.

\begin{remark}
The results presented in the first three lines of Table \ref{tab:summary.equiv.charac} are the submodular analogs of \citep[Theorem 3.1]{ahmadi2013complete} for convexity. Our proofs however 
are constructive, whereas some of theirs are not. 
\end{remark}

\begin{example}
To illustrate the results in Table \ref{tab:summary.equiv.charac}, we consider the simple example of the cut function $f$ over the complete graph with 2 nodes. In this setting, the MLE of $f$ is $F(x_1,x_2)=x_1+x_2-2x_1x_2.$ It is not difficult to see that $f$ is submodular. Indeed, Proposition \ref{prop:submod.mle} $(iii')$ holds as
$$\frac{\partial F(x)}{\partial x_1 \partial x_2}=-2 \leq 0, \forall x \in \{0,1\}^2.$$
Furthermore, $f$ is $0$-sos submodular, as $-\frac{\partial F(x)}{\partial x_1 \partial x_2} \equiv (\sqrt{2})^2$ mod $I_2[x_{-1,2}]$. We then have that:
$$\frac{\partial F(x)}{\partial x_i}-\frac{\partial F(y)}{\partial x_i}=-2x_{k}+2y_k\equiv 2(y_k-x_k)^2 \mod{I_1[x_{-i},y_{-i}]}, \text{ for } (i,k) \in \{(1,2),(2,1)\}.$$
Thus, $\frac{\partial F(x)}{\partial x_i}-\frac{\partial F(y)}{\partial x_i}$ is $1$-sos mod $I_1[x_{-i},y_{-i}]$ for $i\in \{1,2\}$ as in Theorem \ref{thm:equiv.sos.submod.ii}. We also have:
\begin{align*}
G(x,y,s)&\equiv 2(x_1-s_1)^2(y_2-s_2)^2+2(x_2-s_2)^2(y_1-s_1)^2 \mod{I_0[x,y,s]}\\
H_1(x,y)&\equiv 2(y_1-x_1)^2 \cdot (y_2-x_2)^2 \mod{I_1[x,y]} \\
H_2(x,y)&\equiv 2(y_1-x_1)^2 \cdot (y_2-x_2)^2 \mod{I_1[x,y]}\\
G_1(x,y,s)&\equiv 2(x_1-s_1)^2\cdot (x_2-s_2)^2+2(y_2-s_2)^2\cdot (y_1-s_1)^2 \mod{I_0[x,y,s]}\\
G_2(x,y,s)&\equiv 2(x_1-s_1)^2\cdot (x_2-s_2)^2+2(y_2-s_2)^2\cdot (y_1-s_1)^2 \mod{I_0[x,y,s]}
\end{align*}
illustrating Theorems \ref{thm:equiv.sos.submod.i}---\ref{thm:equiv.sos.submod.vii}.
\end{example}

\subsection{Operations that preserve $t$-sos submodularity} \label{subsec:ops.preserving.sos.submod}

There are many well-known operations that preserve submodularity \citep[Appendix~B]{bach2013learning}. Do these operations also preserve $t$-sos submodularity? We answer this question now.

\begin{proposition} \label{prop:ops.maintain.sos.submod}
The following operations maintain $t$-sos submodularity:
\begin{itemize}
\item (Nonnegative weighted sums). If $f_1,\ldots,f_n$ are $t$-sos submodular and $w_1,\ldots,w_n$ are nonnegative weights, then $w_1f_1+\ldots+w_nf_n$ is $t$-sos submodular.
\item (Complementation). If $f$ is $t$-sos submodular, then the set function $g(S)=f(\Omega \backslash S)$ is $t$-sos submodular. 
\item (Restriction). If $f$ is $t$-sos submodular and $A \subseteq \Omega$ is fixed, then the set function $g(S)=f(S\cap A)$ is $t$-sos submodular. 
\item (Contraction.) If $f$ is $t$-sos submodular and $A \subseteq \Omega$ is fixed, then the set function $g(S)=f(S\cup A)-f(A)$ is $t$-sos submodular.
\end{itemize}
\end{proposition}
The proof of this result can be found in Appendix \ref{appendix:ops.preserving.sos.submod}. It makes use of the fact that the set of $t$-sos polynomials is a cone and that the composition of a $t$-sos polynomial with an affine mapping remains $t$-sos. In contrast, the following submodularity-preserving operations do \emph{not} preserve $t$-sos submodularity, in a way that we formalize now.

\begin{proposition} \label{prop:ops.do.not.maintain.sos.submod}
Consider the following three operations:
\begin{itemize}
\item (Partial minimization). Let $f:2^{\Omega_1} \times 2^{\Omega_2} \rightarrow \mathbb{R}$ be a set function, where $\Omega_1, \Omega_2$ are two sets such that $\Omega_1 \cap \Omega_2=\emptyset.$ Define $g:2^{\Omega_1}\rightarrow \mathbb{R}$ to be the set function $g(S)=\min_{T \subseteq \Omega_2} f(S\cup T)$ for any $S \subseteq \Omega_1.$
\item (Convolution). Let $f:2^{\Omega} \rightarrow \mathbb{R}$ be a set function and let $u:2^{\Omega} \rightarrow \mathbb{R}$ be a modular function. Define $g:2^{\Omega}\rightarrow \mathbb{R}$ to be the set function $g(S)=\min_{T\subseteq S} f(T)+u(S\backslash T)$ for any $S\subseteq \Omega.$
\item  (Monotonization). Let $f:2^{\Omega} \rightarrow \mathbb{R}$ be a nonnegative set function. Define $g:2^{\Omega}\rightarrow \mathbb{R}$ to be the set function $g(S)=\min_{T\supseteq S} f(T)$ for any $S\subseteq \Omega.$
\end{itemize}
\revision{In each case, there exists a function $f$ that is $0$-sos-submodular for which the corresponding $g$ is $(n-2)$-sos-submodular, but not $t'$-sos-submodular for any $t'<n-2.$}
\end{proposition}
The proof of this result can be found in Appendix \ref{appendix:ops.preserving.sos.submod}. In simpler terms, it shows that using any one of these operations can create a wide gap in the complexity of certifying submodularity: while $f$ can be shown to be submodular using the lowest $t$ (namely $t=0$), $g$ requires the highest possible $t$ (namely $t=\lceil \frac{n+d-5}{2} \rceil=n-2$ when $d=n$) to certify submodularity. The proof of this result involves showing that these operations do not preserve the degree of the set functions involved: when $f$ is of degree $2$, $g$ can be of degree $n$. This phenomenon mimics certain convexity-preserving properties which do not translate to sos-convexity, for example, partial minimization or pointwise maxima. The issue here is that these operations, when applied to polynomials, do not return polynomials.

\subsection{A study of the gap between submodularity and $t$-sos submodularity} \label{subsec:gap.sos.submod.submod}

Our first result in this section shows that when the degree $d$ of a submodular set function $f$ is low (i.e., $d<4$), then we can certify submodularity of $f$ via $t$-sos submodularity for small $t$. When $d=4$, however, this stops being the case (see Proposition \ref{prop:CE.deg.4}). This is in line with the hardness result in Proposition~\ref{prop:np.hardness}. We then provide examples of functions that are $t$-sos submodular with known $t.$

\begin{proposition} \label{prop:low.degrees}
The following results hold:
\begin{itemize}
\item Any degree-0 or 1 set function is $0$-sos-submodular.
\item Any degree-2 submodular set function is $0$-sos-submodular.
\item Any degree-3 submodular set function is $1$-sos-submodular. 
\end{itemize}
\end{proposition}
In particular, any quadratic function $F(x)=x^TQx+b^Tx+c$ where $Q$ has nonpositive off-diagonal terms is $0$-sos-submodular. The proof of this proposition (Appendix \ref{appendix:gap.sos.submod.submod}) further shows that in the degree-3 setting, one can reduce the semidefinite program under consideration to a linear program.

\begin{proposition} \label{prop:CE.deg.4}
Consider the degree-4 set function $f$ in $n$ variables whose MLE $F$ is given by
$$F(x)=\begin{cases} &\left(-2\sum_{1 \leq i<j\leq n-2}x_ix_j-\frac{n^2-4n+3}{4}\right) x_{n-1}x_{n}+(n-3)\sum_{i=1}^{n-2} x_i \cdot \left( x_{n-1}x_{n}-x_{n-1}-x_{n} \right),\text{ $n$ odd}\\
&\left(-2\sum_{1 \leq i<j\leq n-2}x_ix_j-\frac{(n-2)(n-4)}{4}\right) x_{n-1}x_{n}+(n-4)\sum_{i=1}^{n-2} x_i \cdot \left( x_{n-1}x_{n}-x_{n-1}-x_{n} \right),\text{ $n$ even} \end{cases}.$$
We have that $f$ is submodular (in fact, $\lceil\frac{n-2}{2} \rceil$-sos-submodular when $n$ is odd and $\lceil\frac{n-1}{2} \rceil$-sos-submodular when $n$ is even) but not $t$-sos-submodular for any $t \leq \left(\lceil\frac{n-2}{2}\rceil-1\right)$
\end{proposition}

The proofs of these propositions can be found in Appendix \ref{appendix:gap.sos.submod.submod}. Proposition \ref{prop:low.degrees} confirms that checking submodularity of a set function of degree $\leq 3$ can be done efficiently. Proposition~\ref{prop:CE.deg.4} shows the existence of a degree-4 submodular polynomial which requires, in light of Proposition~\ref{prop:bounds.on.t}, the highest possible $t$ to certify submodularity via $t$-sos submodularity. As $t=O(n)$, the SDP which needs to be solved in this case is of size exponential in $n$, as expected given Proposition \ref{prop:np.hardness}.

We now provide example of functions that are $t$-sos submodular. The first example is a generalization of cut functions to hypergraphs. Recall that a hypergraph is a generalization of a graph where edges can connect more than 2 nodes. In a similar way to the graph setting, if $\mathcal{H}=(V,E)$ is a hypergraph with nonnegative hyperedge weights $w_e$, we define the (weighted) cut function of $\mathcal{H}$ as the function that maps $S\subseteq V$ to the sum of the weights of the hyperedges that contain one node from $S$ and one node from $V\backslash S.$

\begin{proposition} \label{prop:hypercut}
Let $\mathcal{H}=(V,E)$ be a hypergraph with nonnegative hyperedge weights $w_e.$ Let $\max_{e \in E}|e|=t$, i.e., $t$ is the edge size. The (weighted) cut function of $\mathcal{H}$ is a $2 \left \lfloor t/2-1\right \rfloor$-sos-submodular function.
\end{proposition}
When the edge size is $2$ (i.e., when $\mathcal{H}$ is a graph) we recover the result that the cut function is $0$-sos-submodular. Another canonical example of a submodular function are coverage functions. We show here that such functions are $(t-2)$-sos-submodular for a problem-specific $t$.

\begin{proposition} \label{prop:coverage}
Let $m,n \in \mathbb{N}$ and let $A_1,\ldots,A_n \subseteq \{1,\ldots,m\}$. Consider the coverage function $f(S)=|\cup_{i \in S} A_i|.$ Let $t=\max_{\ell=1,\ldots,m} |\{k~|~ \ell \in A_k\}|$, that is, $t$ is the maximum number of sets any element of $\{1,\ldots,m\}$ can belong to. We have that $f$ is $(t-2)$-sos-submodular.
\end{proposition}
The next proposition is the counterpart, for $t$-sos submodularity, of the fact that $f(S)=\phi(|S|)$ is a submodular function when $\phi$ is concave.
\begin{proposition} \label{prop:concave}
Let $\Omega=\{1,\ldots,n\}$ and let $\phi:\mathbb{R}\rightarrow \mathbb{R}$ be a concave polynomial of degree $2d$ with $d \leq n-1,$ then $f:2^{\Omega} \rightarrow \mathbb{R}$ defined by $f(S)=\phi(|S|)$ for $S\subseteq \Omega$ is $d$-sos-submodular.
\end{proposition}
This proposition can be readily extended to $f(S)=\phi(m(S))$, where $m$ is a nonnegative modular function. One can similarly show that if $\phi_{ij}:\mathbb{R} \rightarrow \mathbb{R}$ is a convex \emph{function} for all $1\leq i<j \leq n,$ then $f(S)=\sum_{1 \leq i<j\leq n} \phi_{ij}(1_{i \in S}-1_{j \in S})$ is $0$-sos-submodular.
All of these proofs can be found in Appendix~\ref{appendix:gap.sos.submod.submod}.

%% file: 4_Applications.tex
\section{Applications} \label{sec:applications}

In this section, we propose three different applications of sos submodularity: in Section \ref{subsec:regression}, we consider \emph{submodular regression}; in Section \ref{subsec:approx.submod},  \emph{approximate submodular maximization}; and in Section \ref{subsec:diff.submod}, optimization problems involving \emph{difference of submodular functions}. \revision{Information relating to implementation details can be found in Appendix \ref{appendix:numeric} and all code can be found \href{https://github.com/SOSSubmodularity/SOSSubmodularity_Revision}{here}.}


\subsection{Submodular regression}\label{subsec:regression}

Given $m$ data points $(x_i,y_i) \in \{0,1\}^{n} \times \mathbb{R},$ we wish to fit a submodular function $f:2^{\Omega} \rightarrow \mathbb{R}$ to these data points so as to minimize some loss function. This is shape-constrained regression, where the shape constraint of the function is submodularity, a very natural prior along with monotonicity and convexity; see \citep{stobbe2012learning,balcan2011learning} for some application examples. Of course, given the hardness result in Section \ref{subsec:np.hard.submod}, this is a hard problem when $\deg(f)\geq 4$. We thus propose to fit instead degree-$k$ sos-submodular functions to the data. This problem is tractable and gives rise to a submodular regressor, as encapsulated in the following proposition. 

\begin{proposition}
Let $(x_i,y_i)_{i=1,\ldots,m}\in \{0,1\}^n \times \mathbb{R}$, $k \in \mathbb{N}$. For $\lceil \frac{k-2}{2} \rceil \leq t \leq \lceil\frac{n+k-5}{2} \rceil$, the program
\begin{align}\label{eq:t.sos.regression}
&\min_{F \in \mathbb{R}[x], \deg(F)\leq k} \sum_{i=1}^n (y_i-F(x_i))^2\\
&\text{s.t. } F \text{ is $t$-sos submodular}\notag
\end{align}
\revision{is a semidefinite program to solve with $\frac{n(n-1)}{2}$ semidefinite constraints of size $\sum_{t'=0}^t \binom{n-2}{t'}$.} The\footnote{We place ourselves in a regime where $m$ is large enough so that the problem is guaranteed to have a unique solution.} optimal solution $F^*_{k,t}$ is guaranteed to be submodular.
\end{proposition}
The proof of this proposition is omitted as it follows directly from  Propositions \ref{prop:sufficiency.sos.submod} and \ref{prop:tractability.sos.submod}. To perform submodular regression via this approach, the user need only specify the degree $k$ of the regressor and the degree $t$ of the sum of squares constraint. Both of these degrees can be viewed as regularization hyperparameters and can be set using, e.g., cross-validation. Obtaining $F^*_{k,t}$ is tractable as it involves solving a convex optimization problem, and $F^*_{k,t}$ is guaranteed to be submodular. Furthermore, evaluation of $F^*_{k,t}$ on unseen data can be done very quickly, via point evaluations of a polynomial. Finally, fitting low-degree polynomials, as opposed to higher-degree polynomials, to the data may be of value in some applications. For example, Shapley values can be computed in polynomial time in this setting \citep{grabisch2000equivalent}.

Existing literature on submodular regression can be grouped into two categories. The first category---which is not particularly relevant to us---is theoretical and focuses on understanding how hard it is to learn a submodular function from a polynomial number of samples, under different learning models \citep{goemans2009approximating,balcan2011learning,balcan2018submodular,feldman2013representation,feldman2016optimal}. The second category is practical and aims to develop algorithms for fitting submodular functions to data, similarly to what is done here. There are two main lines of work here. One involves fitting mixtures of submodular ``shells'' to data; see, e.g., \citep{sipos2012large,lin2012learning,tschiatschek2014learning}. This requires significant application knowledge as the user needs to hand-pick the shells to consider, and has consequently mostly been deployed for summarization applications. The other line of work involves fitting neural networks constrained to be submodular to the data \citep{dolhansky2016deep,de2022neural}. As \cite{de2022neural} mention, the approach proposed by \cite{dolhansky2016deep} still requires a consequent amount of user input to choose the activation functions. In contrast, FlexSubNet, the approach proposed by \citep{de2022neural}, requires setting a handful of hyperparameters only, though it is a higher number than what is needed for our method. Furthermore, fitting FlexSubNet to data is a nonconvex problem, unlike our approach, which leads to some variability in the fitting, depending on, e.g., the initialization of the algorithm used to do so.



\begin{example}
We illustrate the effectiveness of $t$-sos-submodular regression using the synthetic data generation procedure outlined in \citet{de2022neural}. We consider a submodular function $$F(S) = \log\left(\sum_{i \in S} 1^\top z_i\right)$$ where $z_i \sim \text{unif}[0,1]^{10},\ i = 1, \ldots, n$. We set $n=15$, and sample 2000 pairs $(S, F(S))$ to form a dataset, which we split 50\%:25\%:25\% split for training, validation, and testing respectively. Gaussian noise is added to the training labels, with standard deviation equal to the empirical standard deviation of the noiseless labels scaled by $\sigma$. To show the effects of noise, we vary $\sigma \in \{0, 0.05, 0.1, 0.15\}$. We then fit four different functions to this data: $(i)$ a $t$-sos-submodular degree-$k$ regressor for varying $(k,t)$, i.e., a solution to \eqref{eq:t.sos.regression}; $(ii)$ a degree-$k$ polynomial regressor for varying $k$, i.e., a solution to \eqref{eq:t.sos.regression} without the sos submodularity constraints; $(iii)$ FlexSubNet from \citep{de2022neural}, described above; $(iv)$ the degree-$k$ polynomial regressor given in \citep{stobbe2012learning} for varying $k$. This last regressor is somewhat of an intermediate between the regressors in $(i)$ and $(ii)$. It is obtained by fitting a degree-$k$ polynomial regressor to data with constraints corresponding to necessary conditions for submodularity. It is shown in the paper that when $k=2$ or $k=3,$ these conditions are also sufficient, and so this approach coincides with the sos approach for those degrees---though they differ when $k \geq 4.$ It is quite clear, from this list, that our sos approach to submodular regression does not have a natural comparator. Indeed, the regressors obtained in $(ii)$ and $(iv)$ are not submodular in general, though they can be obtained by solving convex programs, whereas the regressor obtained in $(iii)$ is submodular, but is obtained by solving a nonconvex program. This is in contrast to convex regression where sum of squares approaches have a very natural comparator: the convex least squares regressor \citep{boyd2004convex,curmei2025shape}, which is guaranteed to be convex and can be obtained by solving a quadratic program. Our rationale in choosing $(ii)-(iv)$ as comparators thus requires some explanation: we elect to compare our approach against $(ii)$ and $(iv)$ as they all involve polynomial regression; we further choose $(iii)$, as it is the current state-of-the-art for submodular regression. Full details regarding the implementation and additional experiments which vary the functional form of $F$ can be found in Appendix \ref{appendix:t.sos.regression}.

\minorrevision{We present the results of our simulation in Figure \ref{fig:regression}. We plot there the test error (i.e., the average test root mean squared error or RMSE) for each method, averaged over five randomly generated problem instances, as a function of the noise level $\sigma$. We remark that the $t$-sos-submodular approach for degree $k=3$ and $k=4$ leads to significantly lower test error than other methods overall, with $k=4$ providing the best fit. Within the $k=4$ setting, varying $t$ does not lead to large performance differences for the $t$-sos-submodular approach (in fact, increasing $t$ leads to a slight deterioration), suggesting that small values of $t$ already capture most of the structure present in the data.

The good performance of the $t$-sos-submodular approach can be interpreted through a bias--variance lens. The true data-generating function is submodular, and the $t$-sos-submodular constraint enforces this structure on the fitted regressor. This restriction may introduce some bias, since we fit a low-degree polynomial to the data rather than the true logarithmic function. However, it may also reduce variance by restricting the set of admissible fitted functions: the $t$-sos-submodular approach cannot fit the training data using functions that violate submodularity. When this reduction in variance outweighs the increase in bias, the resulting test error decreases; this is consistent with what we observe in Figure \ref{fig:regression}.

This interpretation also helps explain the comparison between the three polynomial-based regressors, that is, regressors $(i), (ii),$ and $(iv).$  Since all three approaches fit polynomial regressors of the same degree, their approximation biases should be comparable for a fixed value of $k$, but their variances may differ substantially. The $t$-sos-submodular regressor is expected to have the lowest variance, since it restricts the fitted polynomial to be submodular; the necessary-condition-based regressor in \citep{stobbe2012learning} allows a larger class of degree-$k$ polynomials and hence may have larger variance; and this effect will likely be strongest for the unconstrained polynomial regressor, as it allows the largest class. This seems to explain the performances we observe in Figure~\ref{fig:regression} for degree $k=4$. At this degree, the necessary conditions for submodularity no longer coincide with submodularity itself, and hence the three polynomial-based approaches impose genuinely different levels of structure. It is also a degree at which overfitting can begin to appear. The observed ordering of the test errors is consistent with the bias--variance interpretation: the $t$-sos-submodular approach performs best, while the necessary-condition-based regressor behaves more similarly to unconstrained polynomial regression.

The comparison with FlexSubNet is less direct, since FlexSubNet also enforces submodularity but does so through a non-polynomial parametrization. In Figure \ref{fig:regression}, its performance does not deteriorate as noise is added to the data, which illustrates the value of submodularity as a structural constraint. However, its performance is poor on this instance. It fares better on some other functional forms of the data (see Appendix \ref{appendix:t.sos.regression}) but, overall, it performs inconsistently and is always dominated by the $t$-sos-submodular approach. We surmise that this may be due to the large number of hyperparameters of FlexSubNet, which make it a highly expressive class of functions and may lead to overfitting on medium-sized data samples. This is particularly interesting as FlexSubNet leverages additional structural properties of $F$ for fitting (e.g., monotonicity), unlike the $t$-sos-submodular approach.}

\begin{figure}[]
    \centering    \includegraphics[width=0.7\linewidth]{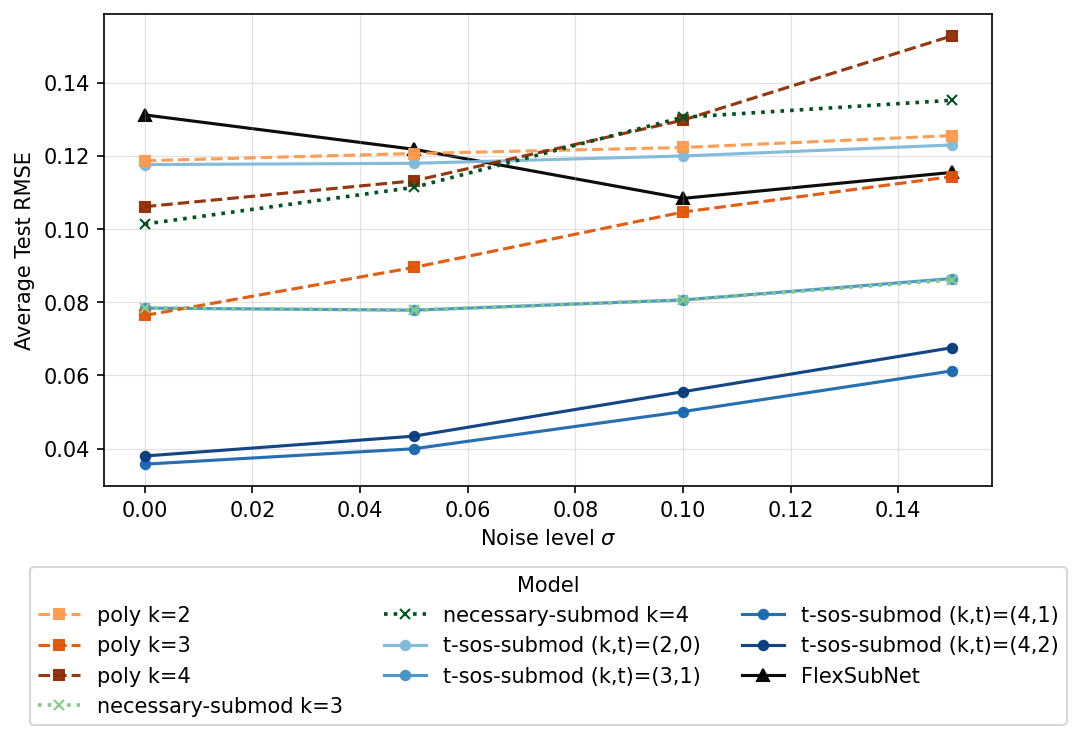}
    \caption{Comparison of root mean squared test error for various methods, fitted to noisy samples of submodular log functions.}
    \label{fig:regression}
\end{figure}
\end{example}

\subsection{Approximate submodular maximization} \label{subsec:approx.submod}

As mentioned in the introduction, it is well known that maximization of a non-decreasing submodular function subject to a cardinality constraint is NP-hard, but that a $(1-e^{-1})$-approximation can be obtained using the greedy algorithm \citep{nemhauser1978analysis}. Interestingly, these guarantees remain somewhat valid for \emph{approximately submodular} non-decreasing functions as shown by \cite{das2018approximate} and \cite{bian2017guarantees}. We make this statement more precise now. We consider a nonnegative monotone set function $f$, that is a function $f$ for which $f(S) \leq f(T)$ for any $S\subseteq T$ and for which $f(S) \geq 0,\forall S.$ For functions such as these, one can define the notion of \emph{submodularity ratio} \citep{das2018approximate}.

\begin{definition}[Definition 2, \cite{das2018approximate}]
The submodularity ratio of a nonnegative monotone function $f$ with respect to a set $U$ and a parameter $k \geq 1$ is:
$$\gamma_{U,k}(f)=\min_{\{L \subseteq U, S~|~|S| \leq k, S\cap L=\emptyset\}} \frac{\sum_{x\in S} (f(L\cup \{x\})-f(L))}{f(L\cup S)-f(L)}.$$
\end{definition}

It can be shown that $f$ is submodular if and only if $\gamma_{U,k} \geq 1$ for all $U$ and $k$ \citep[Proposition 3]{das2018approximate}. If we let $\gamma^*(f)=\min_{U\subseteq \Omega,k} \gamma_{U,k \geq 1}(f)$, then it can also be shown that the greedy algorithm recovers a $(1-e^{-\gamma^*(f)})$-approximation. Computation of $\gamma^*(f)$ is discussed at length in \citep[Remark 4]{das2018approximate}. In particular, it is known that $\gamma^*(f)$ is NP-hard to compute and it is acknowledged that ``typically, rather than computing the submodularity ratio on a given instance, one would use problem-specific insights to derive a priori lower bounds on the submodularity ratio in terms of quantities that are easier to compute exactly or approximately.'' We propose here a tractable and systematic way of obtaining good-quality lower bounds on $\gamma^*(f)$ using $t$-sos submodularity. This approach has commonalities with other computer-assisted worst-case analyses of optimization algorithms relying on semidefinite programming; see, e.g. \citep{goujaud2024pepit} and related work. We describe the approach in more detail.

\begin{proposition} \label{prop:approx.submod}
Let $f:2^{\Omega} \rightarrow \mathbb{R}$ be a nonnegative monotone set function and let $F$ be its multilinear extension. Let $\gamma$ be the largest scalar such that
$$\sum_{i=1}^n (y_i-x_i)\frac{\partial F(x)}{\partial x_i}+\gamma(F(x)-F(y)) \geq 0, \forall x,y \in \{0,1\}^n, x\leq y.$$
Then $\gamma=\gamma^*(f).$
\end{proposition}
The proof of this proposition can be found in Appendix \ref{appendix:proofs.approx.submod}. When $\gamma=1$, we recover the definition of submodularity as given in Proposition \ref{prop:submod.mle} $(v').$ By using similar proof techniques to Propositions \ref{prop:sufficiency.sos.submod}, \ref{prop:tractability.sos.submod} and \ref{prop:bounds.on.t}, the following corollary is immediate.

\begin{corollary} \label{cor:t.sos.approx.submod}
    Let $f:2^{\Omega} \rightarrow \mathbb{R}$ be a nonnegative monotone set function of degree $d$ and let $F$ be its multilinear extension. Define:
    \begin{equation} \label{eq:approx.submod.sos}
    \begin{aligned}
    \gamma^t_{sos} &\mathrel{\mathop{:}}= \max_{\gamma} \gamma\\
    &\text{s.t. } \sum_{i=1}^n (y_i-x_i)\frac{\partial F(x)}{\partial x_i}+\gamma(F(x)-F(y)) \text{ is $t$-sos} \mod{I_1[x,y]}.
    \end{aligned}
    \end{equation}
Then $\gamma^*(f) \geq \gamma^{t'}_{sos} \geq \gamma^t_{sos}, \forall t' \geq t \geq 0$ and $\gamma^*(f)=\gamma^t_{sos}$ for $t=\lceil \frac{n+d-1}{2}\rceil.$ Furthermore, computing $\gamma^t_{sos}$ for fixed $t$ amounts to solving a semidefinite program.
\end{corollary}

The sum of squares condition in \eqref{eq:approx.submod.sos} involves a polynomial that is equal to $H_1(x,y)$ when $\gamma=1$ (see Theorem \ref{thm:equiv.sos.submod.v}). One cannot, however, derive analogs of the algebraic equivalences from Section~\ref{subsec:equiv.charac.sos.submod} when $\gamma \neq 1$. This is apparent when considering the proof of Theorem \ref{thm:equiv.sos.submod.v}. In particular, one cannot recover a characterization of approximate submodularity which relies solely on second-order derivatives. We now illustrate this approach on an example given in \citep{bian2017guarantees}.

\begin{example}
We consider here the determinantal function in \citep[Section 4.2]{bian2017guarantees}. Given a positive definite covariance matrix $\Sigma \in \mathbb{R}^{n \times n}$ and a parameter $\sigma>0$, this function can be written:
$$f(S)=\det(I+\sigma^{-2}\Sigma_S)$$
where $I$ is the $n \times n$ identity matrix and $\Sigma_S$ is a submatrix of $\Sigma$ whose rows and columns are indexed by $S.$ As explained in \citep{bian2017guarantees}, $f$ is not submodular in general; in fact, it is supermodular. Letting the eigenvalues of $I+\sigma^{-2}\Sigma$ be $\lambda_1 \geq \ldots \lambda_n >1$, they show that
\begin{align} \label{eq:bianetal}
\gamma^*(f) \geq \frac{n(\lambda_n-1)}{\left( \prod_{j=1}^n \lambda_j \right)-1}= \mathrel{\mathop{:}} \gamma_{spectral}(f).
\end{align}
As this leads to a loose bound in general, we use sos techniques to improve on this bound.

The first thing to note is that Corollary \ref{cor:t.sos.approx.submod} cannot be applied as is to this example. This is because $\deg(f)=n$ for a randomly generated $\Sigma.$ Indeed, one can show using, e.g., Leibniz's formula, that the multilinear extension of $f$ is given by
$$F(x)=\sum_{S\subseteq \Omega}a(S) \prod_{i \in S}x_i,\text{ where } a(S)=\frac{1}{\sigma^{2|S|}}\det(\Sigma_S) >0$$
with the convention that $a(\emptyset)=1.$ As $\deg(f)=n$, simply writing out \eqref{eq:approx.submod.sos} is exponential in the dimension. As a consequence, we work instead with the degree-$k$ truncation $F_k$ of $F$, that is
$$F_k(x)=\sum_{S \subseteq \Omega, |S| \leq k} a(S) \prod_{i \in S} x_i.$$
The question is then how to adapt \eqref{eq:approx.submod.sos} when we replace $F$ by $F_k$, without losing any guarantees. It turns out that the following statement holds.

\begin{proposition} \label{prop:truncation}
Let $f:2^{\Omega} \rightarrow \mathbb{R}$ be a nonnegative monotone set function and let $F$ be its multilinear extension with degree-$k$ truncation $F_k$. Further define $T_k=F-F_k$ and let $$m=\min_{x \in [0,1]^n} \min_{i=1,\ldots,n} \frac{\partial{T_k(x)}}{\partial x_i} \leq 0 \text{ and } M=\max_{x \in [0,1]^n} ||\nabla T_k(x)||_{\infty} \geq 0.$$ Then, setting
\begin{equation} \label{eq:approx.submod.trunc}
\begin{aligned}
\gamma_{trunc, sos}^{t,k}\mathrel{\mathop{:}}= &\max_{\gamma} \gamma\\
&\text{s.t. } \sum_{i=1}^n (y_i-x_i) \left( \frac{\partial F_k(x)}{\partial x_i}+m-M\gamma \right)+\gamma(F_k(x)-F_k(y)) \text{ is $t$-sos }\mod{I_1[x,y]},
\end{aligned}
\end{equation}
we have $\gamma^*(f) \geq \gamma_{trunc,sos}^{t,k}$ for any $k\leq n$ and $\lceil \frac{k}{2}\rceil \leq t  \leq  \lceil \frac{n+k-1}{2}\rceil.$
\end{proposition}
Of course, $\gamma_{trunc,sos}^{t,\deg(f)}=\gamma^t$ for any appropriate $t$, as $F_k=F$ and $T_k=0$ when $k=\deg(f)$. The proof of Proposition \ref{prop:truncation} can be found in Appendix~\ref{appendix:proofs.approx.submod}.

In our setting, we have $T_k(x)=\sum_{S\in \Omega, |S|>k} a(S) \prod_{i \in S} x_i,$ where $a(S) \geq 0.$ Thus, $m=0.$ To compute $M$, we use the notation $x^{i=1}$ (resp. $x^{i=0}$) to mean the vector $x$ where entry $i$ is set to $1$ (resp. $0$). For fixed $i \in \{1,\ldots,n\}$, we have:
\begin{align*}
\left|\frac{\partial T_k(x)}{\partial x_i} \right| =T_k(x^{i=1})-T_k(x^{i=0}) &\leq T_k(x^{i=1}) \leq T_k(1,\ldots,1)\\
&=F(1,\ldots,1)-F_k(1,\ldots,1)=\det(I+\sigma^{-2}\Sigma)-F_k(1,\ldots,1),
\end{align*}
which can be computed numerically as we choose $k$ small enough that we can write out $F_k.$

Placing ourselves in the same experimental setting as that in \cite[Section 5.3]{bian2017guarantees}, we let $n=10$, set $\sigma=2,$ and generate random covariance matrices $\Sigma \in \mathbb{R}^{n \times n}$ with uniformly distributed eigenvalues in $[0,1]$. We take in our case $k \in \{1,2,3,4\}$ (with corresponding $t \in \{1,2,2,3\}$). All results are averaged over 10 instances. We plot in Figure \ref{fig:approx.submod} the true value of $\gamma^*$, computed in a brute-force manner, the lower bound $\gamma_{spectral}$ as given in \eqref{eq:bianetal}, and $\gamma_{trunc,sos}^{t,k}$ for different values of $k$. Even with a drastic truncation at $k=1,$ our method provides an improved upper bound on $\gamma^*$, as compared to $\gamma_{spectral}.$ As expected the quality of this bound increases as $k$ increases, nearly matching the true value of $\gamma^*$ when $k=4.$ Of course, as $k$ increases, the computation time also increases, leading to the classic trade-off of computation versus accuracy; see Appendix \ref{appendix:comp.submod} for details. Note that the case $k=3$, which gives a tight bound on $\gamma^*$ already, can be solved in 15 secs.

\begin{figure}
    \centering    \includegraphics[width=0.6\linewidth]{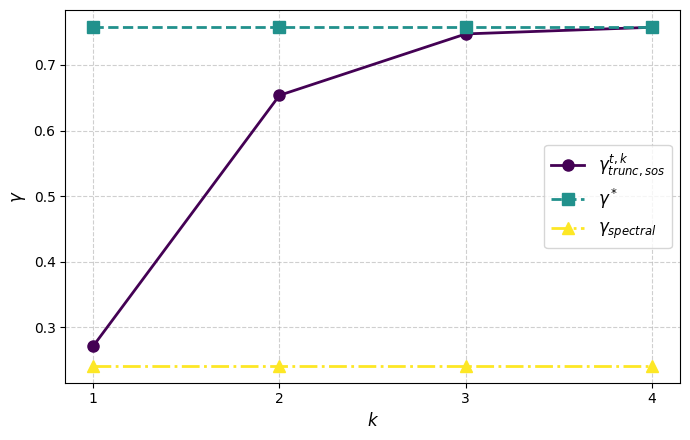}
    \caption{Comparison of $\gamma^*$, $\gamma_{spectral}$ and $\gamma_{trunc,sos}^{t,k}$ for different values of $k$.}
    \label{fig:approx.submod}
\end{figure}
\end{example}

We remark that our methods could likewise be used to obtain bounds on parameters appearing in other definitions of approximate submodularity, such as the ones defined in \citep{horel2016maximization,krause2008near}.

\subsection{Difference of submodular optimization} \label{subsec:diff.submod}

We consider now difference of submodular (ds) optimization. These are problems of the form:
\begin{equation}\label{ds.problem}
    \min_{S \subseteq \Omega} f(S) \coloneq g(S) - h(S),
\end{equation}
where $f$ is a set function, and $g$ and $h$ are submodular set functions. Problems of these type are very general, as any set function $f$ can be decomposed into a difference of submodular functions, although such a decomposition may not be known. Indeed, the general decomposition methods in the literature require exponential time \citep{narasimhan2005submodular}. In many settings, however, decompositions are naturally available or trivial to compute. For example, training discriminatively structured graphical models amounts to optimizing the difference between two mutual information functions, which are both submodular under the naïve Bayes assumption \citep{narasimhan2005submodular}. Another important class of set functions with straightforward decompositions are ones where the MLE $F$ of $f$ is given. If $F = \sum_{T \subseteq \Omega}a(T)\prod_{i\in T} x_i$, then decomposing $F$ into $G-H$, where $G$ and $H$ are submodular can be done by writing:
\begin{equation}
    G(x) = \sum_{T \subseteq T^-}a(T)\prod_{i\in T}x_i \text{ and }H(x) = -\sum_{T \subseteq T^+}a(T) \prod_{i \in T}x_i, \label{eq:trivial.decomp}
\end{equation}
where $T^- = \{T \subseteq \Omega : a(T) < 0 \}$ and $T^+ = \{T \subseteq \Omega : a(T) > 0\}$. Examples of applications where this approach is used are  Markov random field inference problems \citep{jegelka2011submodularity}.

To solve ds programs such as  \eqref{ds.problem}, \citet{narasimhan2005submodular} propose the Submodular-Supermodular Procedure (SSP), which we describe in detail now. Note that other approaches to solving ds programs have also been considered \citep{iyer2012algorithms,el2023difference}, though we do not focus on them in this section. As a first step in describing SSP, we introduce the notion of \textit{subgradients} for submodular functions. For a permutation $\pi$ of $\Omega$, we define $S_i^{\pi}=\{\pi(1),\pi(2),\ldots,\pi(i)\}$ as the set of the first $i$ elements of $\pi.$



\begin{definition}[see, e.g., \citet{iyer2012algorithms}] \label{def:subgradient}
Let $T \subseteq \Omega$ and $\pi$ be a permutation of $\Omega$ such that $S_{|T|}^\pi = T$. Let $h$ be a submodular function. The modular function $\underbar{h}_\pi(S) = \sum_{i \in S} w_\pi(i)$, where
    $$w_\pi(\pi(i)) = \begin{cases}
        h(S_1^\pi) &\text{if } i = 1\\
        h(S_i^\pi) - h(S_{i-1}^\pi)& \text{else,}
    \end{cases}$$
    is a subgradient of $h$ at $T$, that is, it satisfies
    $\underbar{h}_\pi(S) \leq h(S), ~\forall S \subseteq \Omega \text{ and } \underbar{h}_\pi(S^\pi_i) = h(S^\pi_i),~\forall i \in 1, \ldots n.$
\end{definition}

Importantly, for any submodular $h$ and any set $T,$ the subgradients of $h$ at $T$ are non-unique. This is in contrast to, e.g., differentiable convex functions. The SSP is then an iterative procedure: it starts with the set $S_0=\emptyset.$ At iterate $S_t$, it derives a subgradient $\underbar{h}_{\pi}(S)$  of $h$ at $S_t$ by selecting a random permutation $\pi$ such that $S^\pi_{|S_t|}=S_t$. To obtain the next iterate, it solves the submodular minimization problem $\min_{S \subseteq \Omega} (g(S)-\underbar{h}_{\pi}(S)),$ which can be done in polynomial time (see the introduction) and sets $S^{t+1}$ to be a minimizer of the problem. This procedure is repeated until a local minimum of $f$ is attained. (Recall that, for set functions, a set is a local minimum if deleting or adding any item to the set does not result in an improved objective.) Algorithm~\ref{alg:subsup} summarizes the SSP. This procedure has commonalities with the convex-concave procedure (CCP) for difference of convex programs \citep{yuille2003concave}. However, due to the non-uniqueness of the subgradients, the SSP does not have a unique set of iterates when run from the same initial iterate, in contrast to the CCP.

\begin{algorithm}[t]
\caption{Submodular-supermodular procedure \citep{narasimhan2005submodular}}\label{alg:subsup}
\textbf{Input:} Submodular $g$ and $h$ such that $f=g-h$ \\
\textbf{Output:} Local minimum of $f$
\begin{enumerate}[leftmargin=2.4em,label=\arabic*:,itemsep=2pt,topsep=4pt]
\item $S_0 \gets \emptyset,\ t \gets 0$
\item \textbf{while} $S_t$ is not a local minimum \textbf{do}
\item \quad $\pi \gets$ random permutation of $\Omega$ such that $S^{\pi}_{|S_t|}=S_t$
\item \quad $S_{t+1} \gets \arg\min_{S\subseteq \Omega} \left(g- \underbar{h}_{\pi}\right)(S)$
\item \quad $t \gets t+1$
\item \textbf{end while}
\item \textbf{return} $S_t$
\end{enumerate}
\end{algorithm}
The SSP is predicated on a ds decomposition of $f$ being known. Given that any set function $f$ has infinitely many ds decompositions, it is natural to wonder whether specific decompositions would lead to a better performance of the SSP algorithm. A relevant concept to answer this question is that of  \textit{irreducible} decompositions \citep{brandenburg2024decomposition}. An irreducible decomposition $(g,h)$ of $f$ are submodular functions $g$ and $h$ such that $f=g-h$ and such that there exists no non-modular submodular function $p$ such that $g-p$ and $h-p$ are submodular as well. \cite{brandenburg2024decomposition} introduce irreducible decompositions, and the related \textit{minimal} decompositions, for computational purposes, as they lead to Lovasz extensions with the smallest number of pieces, enabling more efficient ds function storage and evaluation. It turns out that an irreducible decomposition also provides a better next-iterate performance of the SSP, compared to any decomposition that can be reduced to it, as we show next.

\begin{proposition}\label{prop:irreducible}
    Let $f:2^{\Omega}\rightarrow \mathbb{R}$ be a set function and let $(g,h)$ be an irreducible decomposition of $f$.  Let $(g',h')$ be a decomposition that can be reduced to $(g,h)$. Then for any permutation $\pi$ of $\Omega$, $(g - \underbar{h}_\pi)(X) \leq (g^\prime - \underbar{h}^\prime_\pi)(X), ~ \forall X \subseteq \Omega$.
    This property no longer holds when comparing two different permutations $\pi_1$, $\pi_2$ of $\Omega$.
\end{proposition}
The proof of this proposition and all results in this section can be found in Appendix \ref{appendix:proofs.diff.submod}.

\begin{remark}
Irreducible decompositions are reminiscent of undominated difference of convex decompositions \citep{bomze2004undominated,ahmadi2018dc}. Undominated decompositions are similarly attractive in difference of convex optimization, as a single iteration of CCP using an undominated decomposition will result in a better iterate than any decomposition which it dominates. Irreducible decompositions in the ds setting have the same property, with a significant caveat: the property only holds if \emph{the same permutation $\pi$ is used to derive the subgradient}. 
\end{remark}


As a consequence of Proposition \ref{prop:irreducible}, it might be relevant, for a given set function, to apply SSP to one of its irreducible decompositions. We show next how such a decomposition can be obtained. 

\begin{proposition}\label{prop:opt.irreducible}
Let $f:2^{\Omega} \rightarrow \mathbb{R}$ be a set function of degree $d$ with MLE $F.$ Let $G^*,H^*$ be an optimal solution\footnote{The problem is always feasible as the trivial decomposition in \eqref{eq:trivial.decomp} is a solution to the problem.} to the following optimization problem:
\begin{equation}\label{model:curv.maximization}
\begin{aligned}
    \underset{G,H \in \mathbb{R}[x], \deg(G),\deg(H)\leq d}{\max}&\sum_{x \in \{0,1\}^n} \sum_{i < j} \frac{\partial^2 H(x)}{\partial x_i \partial x_j} \\
    \text{s.t. } &F = G-H, \\
    & G, H \;\text{submodular}.
\end{aligned}
\end{equation}
The decomposition $(G^*,H^*)$ is an irreducible decomposition of $F.$
\end{proposition}
The intuition behind the choice of objective in \eqref{model:curv.maximization} is the following: a good decomposition of $f$ is such that $h$ is as modular as possible, so that the linear under-approximation $\underbar{h}_{\pi}$ at each iteration is as close as possible to $h$. 
Trivially, $h$ is modular if and only if the second derivative of the MLE $H$ is zero. By submodularity, $\frac{\partial H(x)}{\partial x_i \partial x_j} \leq 0$ for any $i, j \in [n]$, making it natural to search for a decomposition which maximizes the total curvature of $H$ over the hypercube, which is what \eqref{model:curv.maximization} does. There is a problem with \eqref{model:curv.maximization} however in that it is NP-hard to solve.
\begin{proposition} \label{prop:np.hard.irreducible}
Given a set function $f:2^{\Omega}\rightarrow \mathbb{R}$ of degree $4$, and a rational number $k$, it is strongly NP-hard to decide whether there exists a feasible solution to \eqref{model:curv.maximization} with objective value greater or equal to $k$.
\end{proposition}
We thus propose to solve instead
\begin{align}
    \max_{G,H \in \mathbb{R}[x],\deg(G),\deg(H)\leq d} & ~ \sum_{x \in \{0,1\}^n} \sum_{i < j} \frac{\partial^2 H(x)}{\partial x_i \partial x_j} \nonumber\\
    \text{s.t. } & F = G-H, \label{model:curv.maximization.sos} \\
    & G, H \;t\text{-sos-submodular}\nonumber
\end{align}
which is guaranteed to be feasible for any $t \geq 2\lceil d/2\rceil -2$, since the trivial decomposition in \eqref{eq:trivial.decomp} is $(2\lceil d/2\rceil -2)$-sos-submodular. Note that solving \eqref{model:curv.maximization.sos} is a semidefinite program as the objective of \eqref{model:curv.maximization.sos} can be shown to be linear in the coefficients of $H$. 
It can further be shown that the decomposition $(G^*,H^*)$ obtained is $t$-sos-submodular-irreducible, i.e., there is no non-modular submodular $P \in \mathbb{R}[x]$ such that $(G^*-P,H^*-P)$ is a $t$-sos-submodular decomposition of $F.$

\begin{example}
We now provide numerical results to show the potential for improved local optima found by the SSP when using decompositions obtained by solving \eqref{model:curv.maximization}. We compare the value of solutions returned by the algorithm when using the trivial decomposition obtained by \eqref{eq:trivial.decomp}, and the $t$-sos-irreducible decomposition obtained by solving \eqref{model:curv.maximization}. The degree $d$ set functions we optimize are generated by randomly sampling MLE coefficients from $\mathrm{unif}(-10, 10)$. For each randomly generated set function, we run the SSP ten times and report the average relative gap of the local optima found by each decomposition. This is to account for the randomness in the procedure stemming from permutation choice at every iteration. The optimal solution to the pseudo-boolean optimization problem $\min_{x\in\{0,1\}^n} F(x)$ is found for every randomly generated $F$ using a mixed-integer linear program. We fix $d=4$, as this is the smallest degree for which testing submodularity is hard, and set $t=2$ as this is the smallest $t$ such that \eqref{model:curv.maximization} is guaranteed to be feasible. For $n \in \{10, 15, 20\}$, we generate 10 random set functions and report the average relative optimality gap obtained by both decomposition methods in Table~\ref{tab:ds.decomp}. The numerical results indicate that the choice of decomposition of $F$ is significant when applying ds optimization techniques to search for minima, as the $t$-sos-submodular-irreducible decompositions lead to consistently higher-quality local optima. As can be seen in Appendix \ref{appendix:diff.submod}, this is not only the case on average, but on each instance individually. The improved solution quality does come at the cost of computation time, echoing the computation versus quality trade-off we observed in the previous sections.
\begin{table}[h]
    \centering
    \begin{tabular}{|c|c|c|c|}
    \hline
     & $n=10$& $n=15$ & $n=20$ \\ \hline \hline
     Trivial Decomposition & 26.74 & 28.66 & 30.61 \\ \hline
     $2$-sos-submodular Decomposition & 14.78 & 15.41 & 13.52\\
    \hline
    \end{tabular}
    \vspace{3mm}
    \caption{Average relative optimality gap (\%) achieved when using the submodular-submodular procedure to optimize a decomposition provided by the trivial decomposition \eqref{eq:trivial.decomp} and the 2-sos-submodular-irreducible decomposition \eqref{model:curv.maximization} for randomly generated set functions.
    }
    \label{tab:ds.decomp}
\end{table}
\end{example}

%% file: 6_Conclusion.tex
\section{Conclusion}

In this paper, we introduce the new concept of \emph{$t$-sos submodularity}, which is a hierarchy, indexed by $t$, of sufficient algebraic conditions for submodularity that can be efficiently checked via semidefinite programming. We present several equivalent algebraic characterizations of $t$-sos submodularity and show that some submodularity-preserving operations also maintain $t$-sos submodularity. We further quantify the gap between $t$-sos-submodularity and submodularity. We then present three applications in submodular regression, approximately submodular maximization, and difference of submodular optimization of $t$-sos submodularity, which showcase the usefulness of this new concept. By introducing this concept, we connect semidefinite programming to submodularity, an open question in \citep[Chapter 13]{bach2013learning}, and more broadly, real algebraic geometry to discrete optimization.

A natural follow-up work to this paper is whether unconstrained minimization of $t$-sos-submodular functions can be done using semidefinite programming, more specifically, low levels of the Lasserre hierarchy. In the quadratic setting (a setting where the function is $0$-sos submodular, as shown in this paper), it is well-known that minimizing a submodular function can be solved \revision{ efficiently via linear programming \citep[Theorem 3.16]{punnen2022quadratic}}. In the cubic setting (a setting where the function is $1$-sos submodular), \cite{billionnet1985maximizing} show that minimizing a submodular function can be cast as finding the maximum weight stable set in a bipartite graph, \revision{which can also be solved efficiently via a linear program.} More recently, \cite{burer2025semidefinite} show that \emph{continuous} quadratic submodular minimization can be cast as a semidefinite program. It would be interesting to know whether these results generalize to higher degrees. \revision{Another natural direction for future work is to establish formal statistical guarantees for the generalization error of $t$-sos-submodular regression, in line with existing results in shape-constrained regression \citep{chatterjee2015risk}.}

%% file: 7_Appendix.tex
\section{Additional Material Related to Section \ref{sec:theory}} \label{appendix:theory}

\subsection{Background Material and Notation} \label{appendix:background.material}

The first two results we make use of are linked to multilinear polynomials and are straightforward to show.

\begin{lemma} \label{lem:deriv.multilin}
Let $F:\mathbb{R}^n \rightarrow \mathbb{R}$ be a multilinear polynomial. We have:
$$\frac{\partial F(x)}{\partial x_i}=F(x_1,\ldots,x_{i-1},1,x_{i+1},\ldots,x_n)-F(x_1,\ldots,x_{i-1},0,x_{i+1},\ldots,x_n), \forall i=1,\ldots,n.$$
\end{lemma}

\begin{lemma} \label{lem:extension.multilin}
Let $F:\mathbb{R}^n \rightarrow \mathbb{R}$ be a multilinear polynomial. We have:
$$F(x_1,\ldots,x_i,\ldots,x_n)=x_iF(x_1,\ldots,1,\ldots,x_n)+(1-x_i)F(x_1,\ldots,0,\ldots,x_n),~\forall i=1,\ldots,n.$$
\end{lemma}

Some additional background material in algebraic geometry is required for some proofs, \revision{which readers familiar with the material may skip.} We assume that some concepts, e.g., monomial orderings, radical ideals, and Gr\"obner bases, are known; see \cite[Section 2]{laurent2009sums} or \cite[Appendix A.2]{blekherman2012semidefinite} for definitions. Recall the definitions of the ideals given in \eqref{eq:def.ideals} and \eqref{eq:def.ideals.new}. 

\begin{proposition}
Assuming $s_n> \ldots >s_1>y_n>\ldots>y_1>x_n>\ldots>x_1$, the polynomials generating the ideals in \eqref{eq:def.ideals} and \eqref{eq:def.ideals.new} constitute reduced Gr\"obner bases of the corresponding ideals with respect to the graded lexicographic order. 
\end{proposition}
This is straightforward to show using, e.g., the definition of a reduced Gr\"obner basis. Given a polynomial $p \in \mathbb{R}[x]$ and an ideal with a reduced Gr\"obner basis, the \emph{normal form} of $p$ with respect to the Gr\"obner basis, denoted by $\overline{p}$, is the remainder of $p$ when performing multivariate division on $p$ by the Gr\"obner basis. This remainder is \emph{unique}. Normal forms have the following properties.

\begin{lemma}[see, e.g., \cite{permenter2012selecting}] \label{lem:props.normal.form}
Let $p,q \in \mathbb{R}[x]$ and let $I$ be an ideal with associated reduced Gr\"obner basis $G$.  Further, let $\overline{p}$ and $\overline{q}$ be the normal forms of $p$ and $q$ with respect to $G$. We have:
\begin{enumerate}
\item Ideal membership: $p \in I \Leftrightarrow \overline{p}=0$
\item Arithmetic identities: $\overline{p(x) \cdot q(x)}=\overline{\overline{p(x)}\cdot \overline{q(x)}}$ and $\overline{p(x)+q(x)}=\overline{p(x)}+\overline{q(x)}.$
\item Congruence modulo $I$:$p\equiv q \mod{I}$ if and only if $\overline{p}=\overline{q}$.
\end{enumerate}
\end{lemma}

\begin{remark} \label{remark:normal.form.ideal}
To obtain the normal forms of a given polynomial $p$ modulo any one of the three ideals under consideration, we simply use the equations that describe these ideals to sub out the higher degree variables/monomials (e.g., $x_i^2$ for $I_2[x]$) that appear in $p$ by the lower degree variables/monomials (e.g., $x_i$ for $I_2[x]$). 
\end{remark}

\begin{remark} \label{remark:normal.form.qr}
In the definition of $t$-sos modulo $I$ (see Definition \ref{def:t-sos}), if $G$ is a reduced Gr\"obner basis of $I$, one can assume without loss of generality that $q_r$ is in normal form with respect to $G$; see \citep[Section 3.3.5]{blekherman2012semidefinite}. In particular, for all three of our ideals, we can assume that $q_r$ is a multilinear polynomial in the appropriate variables.
\end{remark}

\begin{definition}\label{def:conv.not.proof}
Given $x=(x_1,\ldots,x_n)$ and $y=(y_1,\ldots,y_n)$, and for $i\in \{0,\ldots,n+1\},$ we let $[xy]^i$ to be the vector of variables of size $n$ whose $k^{th}$ entry is given by:
\begin{align*}
[xy]^i_k=
\begin{cases}
x_k &\text{ if } k=1,\ldots,i\\
y_k &\text{ if } k=i+1,\ldots,n.
\end{cases}
\end{align*}
Furthermore, for a vector of variables $x$ and $\alpha \in \mathbb{R}$ or a variable, we let $[x]^{k=\alpha }$ be the vector $x$ with $k^{th}$ entry set to $\alpha.$ We further let $[x]^{k=\alpha,j=\beta}$ be the vector $x$ with $k^{th}$ entry set to $\alpha$ and $j^{th}$ entry set to $\beta$, where $\alpha,\beta$ are real numbers or variables.
\end{definition}

\subsection{Proofs of Results in Section \ref{subsec:def.t.sos}.}\label{appendix:def.t.sos}

\revision{Readers familiar with the sum of squares literature may find these arguments standard and may wish to skip this section as a consequence.}

\proof{Proof of Proposition \ref{prop:sufficiency.sos.submod}.}
Let $t \in \mathbb{N}$ and let $f$ be $t$-sos-submodular with MLE $F$. Further, let $i,j \in \{1,\ldots,n\}, i\neq j$. We have: 
\begin{equation} \label{eq:t-sos}
-\frac{\partial^2 F(x)}{\partial x_i \partial x_j} \equiv \sum_{r} q_r^2(x_{-i,j}) \mod{I_2[x_{-i,j}]}
\end{equation}
for some $q_r \in \mathbb{R}[x_{-i,j}]$ such that $\deg(q_r) \leq t$. Recall from Section~\ref{subsec:ideals.varieties.sos} that $\mathcal{V}_{\mathbb{R}}(I_2[x_{-i,j}])=\{0,1\}^{n-2}$ and that \eqref{eq:t-sos} implies $$-\frac{\partial^2 F(x)}{\partial x_i \partial x_j}= \sum_{r} q_r^2(x_{-i,j}),~ \forall x_{-i,j} \in \mathcal{V}_{\mathbb{R}}(I_2[x_{-i,j}]).$$
As $q_r^2(x_{-i,j}) \geq 0,~ \forall x_{-i,j} \in \mathbb{R}^{n-2}$, we have $\sum_r q_r^2(x_{-i,j}) \geq 0, ~\forall x_{-i,j} \in \{0,1\}^{n-2}$, and $$- \partial^2 F(x_{-i,j})/ \partial x_i \partial x_j \geq 0,~\forall x_{-i,j} \in \{0,1\}^{n-2}.$$ The result follows from Proposition \ref{prop:submod.mle}, $(iii').$
\Halmos
\endproof

We use the notation $S^n$ to denote the set of symmetric matrices in $\mathbb{R}^{n \times n}.$ 

\proof{Proof of Proposition \ref{prop:tractability.sos.submod}.} Let $f:2^{\Omega} \rightarrow \mathbb{R}$ be a set function and let $F$ be its MLE. Following \citep[Section 3.3.5]{blekherman2012semidefinite}, checking whether $f$ is $t$-sos-submodular amounts to solving the following feasibility semidefinite program:
\begin{equation}\label{eq:SDP.2.sos.submod}
\begin{aligned}
&\min_{Q_{ij} \in S^{N}, i,j=1,\ldots,n, i\neq j} 0\\
&\text{s.t. } -\overline{\frac{\partial^2 F(x)}{\partial x_i x_j}} = \mbox{tr}\left(Q_{ij} \overline{z(x_{-i,j})z(x_{-i,j})^T}\right), \forall i,j=1,\ldots,n, i\neq j,\\
&Q_{ij}\succeq 0, \forall i,j=1,\ldots,n, i\neq j.
\end{aligned}
\end{equation}
Here, the overline refers to the normal form with respect to $I_2[x_{-i,j}]$ (see Appendix \ref{appendix:background.material}), $N=\sum_{k=0}^t \binom{n-2}{k}$, and $z(x_{-i,j})$ is the set of all multilinear monomials in $x_{-i,j}$ of degree less than or equal to $t$. As $x_{-i,j} \in \mathbb{R}^{n-2}$, this means that the size of $z(x_{-i,j})$ is $N$, which justifies the size of $Q$.
\Halmos
\endproof

\revision{To give the reader a clearer view of the SDP formulations of checking $t$-sos-submodularity for a given $t$, we derive the SDPs for testing $2$ and $3$-sos-submodularity on the function in Example \ref{ex:test.t.sos}. Code to solve these SDPs is provided  \href{https://mybinder.org/v2/gh/SOSSubmodularity/SOSSubmodularity_Revision/5c878bf0f0c92b63d781a20d07aabe9e26fcf9d1?filepath=ExampleEC1/ExampleEC1.ipynb}{here} which uses the open source solver SCS.


\begin{example}\label{example:SDP}
Recall the degree-$5$ set function $f$ defined on a ground set of size $n=5$ given in Example \ref{ex:test.t.sos}. Its multilinear extension $F$ is given by
$$
F(x_1,x_2,x_3,x_4,x_5)=3x_1x_2x_3x_4x_5-4x_3x_1x_2-9x_4x_1x_2-12x_1x_2x_4x_5-4x_1x_2x_3x_5-4x_1x_2x_3x_4+2.
$$
To determine whether $f$ is $2$-sos-submodular, we need to solve \eqref{eq:SDP.2.sos.submod} with $n=5,N=7,$ and $z(x_{-i,j})=[1,x_k,x_{\ell},x_m, x_kx_{\ell},x_{k}x_m,x_{\ell m}]^T$ for $\{k,\ell,m\}\in \{1,\ldots,5\} \backslash \{i,j\}$ and $i,j \in \{1,\ldots,5\}, i\neq j.$ This optimization problem can be reformulated as an explicit SDP as we see now. We set $i=1$ and $j=2$ to illustrate. Other values of $i$ and $j$ are dealt with identically. Let $Q_{12} \in S^7$ and denote by $Q_{12}^{k\ell}$ the $(k,\ell)$-entry of $Q_{12}$. The constraint 
$$-\overline{\frac{\partial^2 F(x)}{\partial x_1 x_2}} = \mbox{tr}\left(Q_{12} \overline{z(x_{-1,2})z(x_{-1,2})^T}\right)$$
requires us to derive the partial derivative of $F$ with respect to $x_1$ and $x_2:$
$$ -\frac{\partial^2 F}{\partial x_1\,\partial x_2}= -3x_3x_4x_5 + 4x_3x_4 + 4x_3x_5 + 4x_3 + 12x_4x_5 + 9x_4.$$
With $z(x_{-1,2})=[1,x_3,x_4,x_5,x_3x_4,x_3x_5,x_4x_5]^T,$ the constraint then becomes:
\begin{small}
\begin{align*}
&-3x_3x_4x_5 + 4x_3x_4 + 4x_3x_5 + 4x_3 + 12x_4x_5 + 9x_4\\
&=\mbox{tr}(Q_{12} \overline{z(x_{-1,2}) z(x_{-1,2})^T})\\
&= \mbox{tr} \left(Q_{12}\cdot  \begin{pmatrix}
1 & x_3 & x_4 & x_5 & x_3x_4 & x_3x_5 & x_4x_5 \\
x_3 & x_3 & x_3x_4 & x_3x_5 & x_3x_4 & x_3x_5 & x_3x_4x_5 \\
x_4 & x_3x_4 & x_4 & x_4x_5 & x_3x_4 & x_3x_4x_5 & x_4x_5 \\
x_5 & x_3x_5 & x_4x_5 & x_5 & x_3x_4x_5 & x_3x_5 & x_4x_5 \\
x_3x_4 & x_3x_4 & x_3x_4 & x_3x_4x_5 & x_3x_4 & x_3x_4x_5 & x_3x_4x_5 \\
x_3x_5 & x_3x_5 & x_3x_4x_5 & x_3x_5 & x_3x_4x_5 & x_3x_5 & x_3x_4x_5 \\
x_4x_5 & x_3x_4x_5 & x_4x_5 & x_4x_5 & x_3x_4x_5 & x_3x_4x_5 & x_4x_5
\end{pmatrix} \right)\\
&=Q_{12}^{11} + (2Q_{12}^{12}+Q_{12}^{22})x_3
     + (2Q_{12}^{13}+Q_{12}^{33})x_4
     + (2Q_{12}^{14}+Q_{12}^{44})x_5 + (2Q_{12}^{15}+2Q_{12}^{23}+2Q_{12}^{25}+2Q_{12}^{35}+Q_{12}^{55})x_3x_4 \\
&\quad + (2Q_{12}^{16}+2Q_{12}^{24}+2Q_{12}^{26}+2Q_{12}^{46}+Q_{12}^{66})x_3x_5  + (2Q_{12}^{17}+2Q_{12}^{34}+2Q_{12}^{37}+2Q_{12}^{47}+Q_{12}^{77})x_4x_5 \\[6pt]
&\quad + (2Q_{12}^{27}+2Q_{12}^{36}+2Q_{12}^{45}
        +2Q_{12}^{56}+2Q_{12}^{57}+2Q_{12}^{67})x_3x_4x_5 .
\end{align*}
\end{small}
By matching coefficients of the polynomial on the first line and the polynomial on the last three lines, we obtain the following linear equalities:
\begin{align*}
&Q_{12}^{11}=0, \quad 2Q_{12}^{12}+Q_{12}^{22}=4, \quad
2Q_{12}^{13}+Q_{12}^{33}=9, \quad
2Q_{12}^{14}+Q_{12}^{44}=0, \\
&2Q_{12}^{15}+2Q_{12}^{23}+2Q_{12}^{25}+2Q_{12}^{35}+Q_{12}^{55}=4, \quad 2Q_{12}^{16}+2Q_{12}^{24}+2Q_{12}^{26}+2Q_{12}^{46}+Q_{12}^{66}=4,\\
&2Q_{12}^{17}+2Q_{12}^{34}+2Q_{12}^{37}+2Q_{12}^{47}+Q_{12}^{77}=12, \quad  2Q_{12}^{27}+2Q_{12}^{36}+2Q_{12}^{45}
+2Q_{12}^{56}+2Q_{12}^{57}+2Q_{12}^{67}=-3,
\end{align*}
to be combined with the constraint $Q_{12}\succeq 0.$ We can then proceed in an identical fashion for all $(i,j)\in \{1,\ldots,n\}, i\neq j$. We thus obtain the following SDP reformulation of \eqref{eq:SDP.2.sos.submod}:
\begin{small}
\begin{equation} \label{eq:SDP.t=1}
\begin{aligned}
&\min_{Q_{12}, Q_{13}, Q_{14}, Q_{15}, Q_{23}, Q_{24}, Q_{25}, Q_{34}, Q_{35}, Q_{45}\in S^7} 0\\
\text{s.t. } &Q_{12}^{11}=0, \quad 2Q_{12}^{12}+Q_{12}^{22}=4, \quad
2Q_{12}^{13}+Q_{12}^{33}=9, \quad
2Q_{12}^{14}+Q_{12}^{44}=0, \\
&2Q_{12}^{15}+2Q_{12}^{23}+2Q_{12}^{25}+2Q_{12}^{35}+Q_{12}^{55}=4, \quad 2Q_{12}^{16}+2Q_{12}^{24}+2Q_{12}^{26}+2Q_{12}^{46}+Q_{12}^{66}=4,\\
&2Q_{12}^{17}+2Q_{12}^{34}+2Q_{12}^{37}+2Q_{12}^{47}+Q_{12}^{77}=12, \quad  2Q_{12}^{27}+2Q_{12}^{36}+2Q_{12}^{45}
+2Q_{12}^{56}+2Q_{12}^{57}+2Q_{12}^{67}=-3\\
&Q_{13}^{11}=0, \quad 2Q_{13}^{12}+Q_{13}^{22}=4, \quad 2Q_{13}^{13}+Q_{13}^{33}=0, \quad 2Q_{13}^{14}+Q_{13}^{44}=0, \\
&2Q_{13}^{15}+2Q_{13}^{23}+2Q_{13}^{25}+2Q_{13}^{35}+Q_{13}^{55}=4, \quad 2Q_{13}^{16}+2Q_{13}^{24}+2Q_{13}^{26}+2Q_{13}^{46}+Q_{13}^{66}=4,\\
&2Q_{13}^{17}+2Q_{13}^{34}+2Q_{13}^{37}+2Q_{13}^{47}+Q_{13}^{77}=0,\quad 2Q_{13}^{27}+2Q_{13}^{36}+2Q_{13}^{45}
+2Q_{13}^{56}+2Q_{13}^{57}+2Q_{13}^{67}=-3\\ 
& \vdots \\
& Q_{45}^{11}=0, \quad 2Q_{45}^{12}+Q_{45}^{22}=0, \quad 2Q_{45}^{13}+Q_{45}^{33}=0, \quad
2Q_{45}^{14}+Q_{45}^{44}=0, \\
&2Q_{45}^{15}+2Q_{45}^{23}+2Q_{45}^{25}+2Q_{45}^{35}+Q_{45}^{55}=12, \quad
2Q_{45}^{16}+2Q_{45}^{24}+2Q_{45}^{26}+2Q_{45}^{46}+Q_{45}^{66}=0, \\
&2Q_{45}^{17}+2Q_{45}^{34}+2Q_{45}^{37}+2Q_{45}^{47}+Q_{45}^{77}=0, \quad
2Q_{45}^{27}+2Q_{45}^{36}+2Q_{45}^{45}+2Q_{45}^{56}+2Q_{45}^{57}+2Q_{45}^{67}=-3,\\
&Q_{12}, Q_{13}, Q_{14}, Q_{15}, Q_{23}, Q_{24}, Q_{25}, Q_{34}, Q_{35}, Q_{45} \succeq 0.
\end{aligned}
\end{equation}
\end{small}
The SDP \eqref{eq:SDP.t=1} can be solved numerically; see \href{https://mybinder.org/v2/gh/SOSSubmodularity/SOSSubmodularity_Revision/5c878bf0f0c92b63d781a20d07aabe9e26fcf9d1?filepath=ExampleEC1/ExampleEC1.ipynb}{here} for an implementation. In an echo of Example~\ref{ex:test.t.sos}, we find that the problem is infeasible, implying that $f$ is not $2$-sos-submodular. 

To test whether $f$ is $3$-sos-submodular, we change $z(x_{-i,j})$ to allow for multilinear monomials of degree up to 3, that is, we let $z(x_{-i,j})=[1,x_k,x_{\ell},x_m, x_kx_{\ell},x_{k}x_m,x_{\ell m},x_{k \ell m}]^T$ for $\{k,\ell,m\}\in \{1,\ldots,5\} \backslash \{i,j\}$ and $i,j \in \{1,\ldots,5\}, i\neq j.$ We then solve \eqref{eq:SDP.2.sos.submod} once again with decision variables $\{Q_{ij}\}_{i,j}$ that are of size $8 \times 8.$ Proceeding in a similar way as above for its derivation, we obtain the slightly larger SDP:
\begin{small}
\begin{equation} \label{eq:SDP.t=2}
\begin{aligned}
&\min_{Q_{12}, Q_{13}, Q_{14}, Q_{15}, Q_{23}, Q_{24}, Q_{25}, Q_{34}, Q_{35}, Q_{45}\in S^8} 0\\
\text{s.t. } &Q_{12}^{11}=0, \quad
2Q_{12}^{12}+Q_{12}^{22}=4, \quad
2Q_{12}^{13}+Q_{12}^{33}=9, \quad
2Q_{12}^{14}+Q_{12}^{44}=0,\\
&2Q_{12}^{15}+2Q_{12}^{23}+2Q_{12}^{25}+2Q_{12}^{35}+Q_{12}^{55}=4, \quad
2Q_{12}^{16}+2Q_{12}^{24}+2Q_{12}^{26}+2Q_{12}^{46}+Q_{12}^{66}=4,\\
&2Q_{12}^{17}+2Q_{12}^{34}+2Q_{12}^{37}+2Q_{12}^{47}+Q_{12}^{77}=12,\\
&2Q_{12}^{18}+2Q_{12}^{27}+2Q_{12}^{28}+2Q_{12}^{36}+2Q_{12}^{38}+2Q_{12}^{45}+2Q_{12}^{48}
+2Q_{12}^{56}+2Q_{12}^{57}+2Q_{12}^{58}+2Q_{12}^{67}+2Q_{12}^{68}+2Q_{12}^{78}+Q_{12}^{88}=-3\\
& \vdots \\
&Q_{45}^{11}=0, \quad
2Q_{45}^{12}+Q_{45}^{22}=0, \quad
2Q_{45}^{13}+Q_{45}^{33}=0, \quad
2Q_{45}^{14}+Q_{45}^{44}=0,\\
&2Q_{45}^{15}+2Q_{45}^{23}+2Q_{45}^{25}+2Q_{45}^{35}+Q_{45}^{55}=12, \quad
2Q_{45}^{16}+2Q_{45}^{24}+2Q_{45}^{26}+2Q_{45}^{46}+Q_{45}^{66}=0,\\
&2Q_{45}^{17}+2Q_{45}^{34}+2Q_{45}^{37}+2Q_{45}^{47}+Q_{45}^{77}=0,\\
&2Q_{45}^{18}+2Q_{45}^{27}+2Q_{45}^{28}+2Q_{45}^{36}+2Q_{45}^{38}+2Q_{45}^{45}+2Q_{45}^{48}
+2Q_{45}^{56}+2Q_{45}^{57}+2Q_{45}^{58}+2Q_{45}^{67}+2Q_{45}^{68}+2Q_{45}^{78}+Q_{45}^{88}=-3\\
&Q_{12}, Q_{13}, Q_{14}, Q_{15}, Q_{23}, Q_{24}, Q_{25}, Q_{34}, Q_{35}, Q_{45} \succeq 0.
\end{aligned}
\end{equation}
\end{small}
The SDP \eqref{eq:SDP.t=2} is feasible, and a feasible solution can be obtained numerically; see \href{https://mybinder.org/v2/gh/SOSSubmodularity/SOSSubmodularity_Revision/5c878bf0f0c92b63d781a20d07aabe9e26fcf9d1?filepath=ExampleEC1/ExampleEC1.ipynb}{here}. This establishes that $f$ is $3$-sos-submodular, and hence submodular.
\end{example}

}

\revision{We now move onto the proof of Proposition \ref{prop:bounds.on.t}, which makes use of the following result which is an immediate corollary of \cite[Theorem 3.2]{sakaue2016exact}.

\begin{lemma}[\cite{sakaue2016exact}] \label{lem:sakaue}
Consider the binary polynomial optimization problem:
\begin{equation} \label{eq:bin.POP}
\begin{aligned}
&\min_z f(z)\\
&\text{s.t. } z\in \{-1,1\}^n,
\end{aligned}
\end{equation}
where $f$ is a polynomial of degree $d$ and in $n$ variables. The semidefinite programming-based relaxation of order $\lceil \frac{n+d-1}{2}\rceil$ in Lasserre's hierarchy is an exact reformulation of \eqref{eq:bin.POP}.
\end{lemma}

We can then combine Lemma \ref{lem:sakaue}, where we operate the affine change of variables $x_i=(z_i+1)/2$ for $i=1,\ldots,n$ to work on $\{0,1\}^n$ instead of $\{-1,1\}^n$, with known results in the sum of squares literature regarding the Lasserre hierarchy (see \cite[Section 6.2]{laurent2009sums}) to obtain the following corollary. 

\begin{corollary} \label{cor:sakaue}
Let $f$ be a degree-$d$, $n$-variate polynomial, nonnegative on $\{0,1\}^n$. For any $t \geq \lceil \frac{n+d-1}{2}\rceil,$ $f$ is $t$-sos modulo $I_2[x].$
\end{corollary}

}

\proof{Proof of Proposition \ref{prop:bounds.on.t}.} Let $f:2^\Omega \rightarrow \mathbb{R}$ be a set function of degree $d$ and let $F$ be its MLE. Note that $-\frac{\partial^2 F(x)}{\partial x_i \partial x_j}$ is of degree at most $d-2$ for any $i,j \in \{1,\ldots,n\}.$ To obtain monomials of degree $d-2$ when expanding a sum of squares polynomial, it must be the case that this sum of squares polynomial has degree at least $\lceil \frac{d-2}{2} \rceil.$

Now suppose that $f$ is of degree $d$ and is submodular. Thus, from Proposition \ref{prop:submod.mle}, it follows that $-\frac{\partial^2 F(x)}{\partial x_i \partial x_j} \geq 0, \forall x\in \{0,1\}^n$ and $i,j \in \{1,\ldots,n\}$. Fix $i,j \in \{1,\ldots,n\}$ and note that $-\frac{\partial^2 F(x)}{\partial x_i \partial x_j} $ is of degree $d-2$ and in $n-2$ variables. \revision{The fact that $-\frac{\partial^2 F(x)}{\partial x_i \partial x_j} $ is $t$-sos mod $I_2[x_{-i,j}]$ for $t$ at most $\lceil \frac{n+d-5}{2} \rceil$ follows Corollary \ref{cor:sakaue}.}
\Halmos
\endproof

\subsection{Proofs of Results in Section \ref{subsec:equiv.charac.sos.submod}.}\label{appendix:equiv.charac.sos.submod}

Recall the notation given in Definition \ref{def:conv.not.proof} in Section \ref{appendix:background.material}.

\begin{lemma} \label{lem:1.to.2}
Let $f$ be a set function and let $F$ be its MLE. We have:
\begin{align}
\frac{\partial F}{\partial x_i}(x)-\frac{\partial F}{\partial x_i}(y)&= \sum_{j=1, j \neq i}^n (y_j-x_j) \cdot \left( -\frac{\partial^2 F}{\partial x_i \partial x_j}\left([xy]^j\right)\right), \text{ for } i=1,\ldots,n \label{eq:for.ii}\\
F(x)-F(y)&=- \sum_{i=1}^n (y_i-x_i) \frac{\partial F}{\partial x_i}\left([xy]^i\right) \label{eq:for.v},\\
F(x)-F(y)&= -\sum_{i=1}^n (y_i-x_i) \frac{\partial F}{\partial x_i}\left([yx]^{i}\right) \label{eq:F.diff.flipped}
\end{align}
\end{lemma}

\proof{Proof.} We first show \eqref{eq:for.ii}.
Let $i \in \{1,\ldots,n\}$. Using telescopic sums, we have:
\begin{align*}
\frac{\partial F}{\partial x_i}(x)-\frac{\partial F}{\partial x_i}(y)=\frac{\partial F}{\partial x_i}([xy]^{n})-\frac{\partial F}{\partial x_i}([xy]^{0})=\sum_{j=1,j\neq i}^n \left( \frac{\partial F}{\partial x_i}([xy]^j)-\frac{\partial F}{\partial x_i}([xy]^{j-1}) \right).
\end{align*}
We observe that $[xy]^{j}$ and $[xy]^{j-1}$ only differ in entry $j$, that is, $[xy]^{j,j=1}=[xy]^{j-1,j=1}$ and $[xy]^{j,j=0}=[xy]^{j-1,j=0}$. Thus, from Lemmas \ref{lem:deriv.multilin} and \ref{lem:extension.multilin},
\begin{align*}
\frac{\partial F}{\partial x_i}([xy]^{j})-\frac{\partial F}{\partial x_i}([xy]^{j-1}) =&x_{j}\frac{\partial F}{\partial x_i}([xy]^{j, j=1})+(1-x_{j})\frac{\partial F}{\partial x_i}([xy]^{j,j=0})-y_{j}\frac{\partial F}{\partial x_i}([xy]^{j-1,j=1})\\
&-(1-y_{j})\frac{\partial F}{\partial x_i}([xy]^{j-1,j=0})\\
&=(y_j-x_j)\cdot \left( \frac{\partial F}{\partial x_i}([xy]^{j,j=0})- \frac{\partial F}{\partial x_i}([xy]^{j,j=1}) \right) \\
&=(y_{j}-x_{j}) \cdot \left( -\frac{\partial^2 F}{\partial x_i \partial x_{j}}([xy]^{j}) \right).
\end{align*}

We now show \eqref{eq:for.v}. Similarly to above, using telescopic sums, we have:
\begin{align*}
F(x)-F(y)=F([xy]^n)-F([xy]^0)=\sum_{j=1}^n \left( F([xy]^j)-F([xy]^{j-1})\right).
\end{align*}
As above, note that:
\begin{align*}
F([xy]^j)-F([xy]^{j-1})&=x_j F([xy]^{j,j=1})+(1-x_j) F([xy]^{j,j=0})\\
&-y_j F([xy]^{j-1,j=1})-(1-y_j)F([xy]^{j-1,j=0})\\
&=-(y_j-x_j)\frac{\partial F}{\partial x_j} ([xy]^j_{-j}),
\end{align*}
which gives the conclusion. Finally, we show \eqref{eq:F.diff.flipped}. We use telescopic sums to obtain:
\begin{align*}
F(x)-F(y)=F([yx]^0)-F([yx]^{n})=\sum_{j=1}^n \left( F([yx]^{j-1})-F([yx]^{j})\right).
\end{align*}
Similarly to before, we have:
\begin{align*}
F([yx]^{j-1})-F([yx]^{j})&=x_j F([yx]^{j-1,j=1})+(1-x_j)F([yx]^{j-1,j=0})\\
&-y_j F([yx]^{j,j=1})-(1-y_j)F([yx]^{j,j=0})\\
&=-(y_j-x_j) \frac{\partial F}{\partial x_j}([yx]^j),
\end{align*}
which leads to the conclusion.
\Halmos
\endproof

\proof{Proof of Theorem \ref{thm:equiv.sos.submod.ii}.}
Suppose that $f$ is $t$-sos-submodular. For $i,j \in \{1,\ldots,n\}, i\neq j,$ we have: $$-\frac{\partial^2 F(x)}{\partial x_i \partial x_j} \equiv \sum_{r} q_{r,ij}^2(x_{-i,j}) \mod{I_2[x_{-i,j}]}$$ where $q_{r,ij} \in \mathbb{R}[x_{-i,j}]$ and $\deg(q_{r,ij})\leq t$ for $r=1,\ldots,m.$ From Lemma \ref{lem:1.to.2}, for fixed $i=1,\ldots,n,$ we then have:
\begin{align*}
\frac{\partial F}{\partial x_i}(x)-\frac{\partial F}{\partial x_i}(y) &= \sum_{j=1, j\neq i}^n (y_j-x_j) \left( -\frac{\partial^2 F([xy]^j)}{\partial x_i \partial x_j} \right)\\
&\equiv \sum_{j=1, j\neq i}^n \sum_r (y_j-x_j)q_{r,ij}^2([xy]^j_{-i,j}) ~\text{ mod }I_2[[xy]^j_{-i,j}]. 
\end{align*}
Now, for any $j \in \{1,\ldots,n\}, j\neq i$,  note that $I_2[[xy]^j_{-i,j}] \subseteq I_1[x_{-i},y_{-i}]$ and 
$$(y_j-x_j)^2\equiv y_j-x_j \mod{I_1[x_{-i},y_{-i}]}.$$
Thus,
$$\frac{\partial F(x)}{\partial x_i}-\frac{\partial F(y)}{\partial x_i}=\sum_{r=1}^m \sum_{j=1, j\neq i}^n \left((y_j-x_j) q_{r,ij}([xy]^j_{-i,j}) \right)^2 \mod{I_1[x_{-i},y_{-i}]}.
$$
As $\deg(q_{r,ij}) \leq t$ for $q_{r,ij} \in \mathbb{R}[x_{-i,j}]$, $\deg\left( (y_j-x_j)q_{r,ij}([xy]^j_{-i,j})\right) \leq t+1$, which gives the result.

Now let $i \in \{1,\ldots,n\}$. Suppose that we have:
\begin{equation} \label{eq:def.partial}
\frac{\partial F(x)}{\partial x_i}-\frac{\partial F(y)}{ \partial x_i} \equiv \sum_r q_r^i(x_{-i},y_{-i})^2 \mod I_1[x_{-i},y_{-i}]
\end{equation} where $q_r^i \in \mathbb{R}[x_{-i},y_{-i}]$ is multilinear (see Remark \ref{remark:normal.form.qr}) and $\deg(q_r^i) \leq t+1$. Note that
\begin{align*}
\frac{\partial F(x)}{\partial x_i}-\frac{\partial F(x)}{ \partial x_i}=0 \equiv \sum_r q_r^i(x_{-i},x_{-i})^2 \mod I_1[x_{-i}].
\end{align*}
Thus, $q_r^i(x_{-i},x_{-i})\equiv 0 \mod{I_1[x_{-i}]}$ for any $r$. Now, let $j \in \{1,\ldots,n\}$ such that $j\neq i.$ We have:
\begin{align} \label{eq:relationship.F.partial}
- \frac{\partial^2 F(x)}{\partial x_i \partial x_j}=\frac{\partial F([x]^{j=0})}{\partial x_i}-\frac{\partial F([x]^{j=1})}{ \partial x_i}.
\end{align}
As $q_r^i$ in \eqref{eq:def.partial} is multilinear, we can write using Lemma \ref{lem:extension.multilin},
\begin{align*}
q_r^i(x_{-i},[x]_{-i}^{j=y_j}) &= x_j y_j q_r^i([x]_{-i}^{j=1},[x]_{-i}^{j=1})+x_j (1-y_j) q_r^i([x]_{-i}^{j=1},[x]_{-i}^{j=0})\\
&+(1-x_j) y_j q_r^i([x]_{-i}^{j=0},[x]_{-i}^{j=1})+(1-x_j)(1-y_j) q_r^i([x]_{-i}^{j=0},[x]_{-i}^{j=0})\\
&\equiv (y_j-x_j)q_r^i([x]_{-i}^{j=0},[x]_{-i}^{j=1}) \mod I_1[x_{-i},[x]^{j=y_j}_{-i}],
\end{align*}
where we have used the fact that $q_r^i(x_{-i},x_{-i})\equiv 0 \mod{I_1[x_{-i}]}$ (and $I_1[x_{-i}]\subseteq I_1[x_{-i},[x]^{j=y_j}_{-i}]$) and the fact that $x_jy_j\equiv x_j \mod I_1[x_{-i},[x]^{j=y_j}_{-i}]$. As $\deg(q_r^i)\leq t+1$ for $q_r^i \in \mathbb{R}[x_{-i},y_{-i}]$, it follows that $\deg(q_r^i) \leq t+1$ when $q_r^i \in \mathbb{R}[x_{-i},[x]_{-i}^{j=y_j}].$ As the normal forms of $q_r^i(x_{-i},[x]_{-i}^{j=y_j})$ and $q_r^i([x]_{-i}^{j=0},[x]_{-i}^{j=1})$ mod $I_1[x_{-i},[x]^{j=y_j}_{-i}]$ are themselves, we can use Lemma \ref{lem:props.normal.form} to conclude that $\deg(q_r^i([x]_{-i}^{j=0},[x]_{-i}^{j=1})) \leq t.$ Combining \eqref{eq:relationship.F.partial} and \eqref{eq:def.partial}, we then have that:
$$- \frac{\partial^2 F(x)}{\partial x_i \partial x_j}\equiv \sum_r q_r^i([x]_{-i}^{j=0},[x]_{-i}^{j=1})^2 \mod I_1[[x]_{-i}^{j=0},[x]_{-i}^{j=1}]$$
As $I_1[[x]_{-i}^{j=0},[x]_{-i}^{j=1}]=I_2[x_{-i,j}]$ and $\deg(q_r^i([x]_{-i}^{j=0},[x]_{-i}^{j=1})) \leq t$, $f$ is $t$-sos-submodular. \Halmos
\endproof

\proof{Proof of Theorem \ref{thm:equiv.sos.submod.i}.} Suppose that $f$ is $t$-sos-submodular. Using \eqref{eq:for.v} and \eqref{eq:F.diff.flipped} and the notation in Definition \eqref{def:conv.not.proof}, we have:
\begin{align}
G(x,y,s)&=F(x)+F(y)-F(x+y-s)-F(s) \notag \\
&= -\sum_{i=1}^n (y_i-s_i) \frac{\partial F}{\partial x_i}([x(x+y-s)]^i) -\sum_{i=1}^n (s_i-y_i) \frac{\partial F}{\partial x_i}([sy]^i) \notag \\
&=\sum_{i=1}^n (y_i-s_i) \left( \frac{\partial F}{\partial x_i}([sy]^i) - \frac{\partial F}{\partial x_i}([x(x+y-s)]^i) \right). \label{eq:G.expl}
\end{align}
Let $i \in \{1,\ldots,n\}.$ From Theorem \ref{thm:equiv.sos.submod.ii}, as $f$ is $t$-sos-submodular, we have
$$\frac{\partial F}{\partial x_i}(x) - \frac{\partial F}{\partial x_i}(y)  \equiv \sum_r q_{ri}^2(x_{-i},y_{-i}) \mod{I_1[x_{-i},y_{-i}]},$$
where $q_r(x_{-i},y_{-i}) \in \mathbb{R}[x_{-i},y_{-i}]$ and $\deg(q_{ri}(x_{-i},y_{-i}) )\leq t+1.$ This implies that
\begin{align*}
\frac{\partial F}{\partial x_i}([sy]^i) - \frac{\partial F}{\partial x_i}([x(x+y-s)]^i)  \equiv &\sum_r q_{ri}^2([sy]^i_{-i},[x(x+y-s)]^i_{-i}) \mod{I_1[[sy]^i_{-i},[x(x+y-s)]^i_{-i}]}.
\end{align*}
Note that $\deg(q_{ri}([sy]^i_{-i},[x(x+y-s)]^i_{-i})) \leq t+1.$ Substituting this in \eqref{eq:G.expl}, we get:
$$G(x,y,s)\equiv \sum_{i=1}^n \sum_r (y_i-s_i)q_{ri}^2([sy]^i_{-i},[x(x+y-s)]^i_{-i}) \mod{I_1[[sy]^i_{-i},[x(x+y-s)]^i_{-i}]}.$$
It is straightforward to show that $I_1[[sy]^i_{-i},[x(x+y-s)]^i_{-i}] \subseteq I_0[x,y,s].$
Furthermore, $y_i-s_i\equiv (y_i-s_i)^2 \mod{I_0[x,y,s]}$. Renaming $q_{ri}([sy]_{-i}^i,[x(x+y-s)]_{-i}^i)$ to be $\bar{q}_{ri}(x_{-i},y_{-i},s_{-i})$, we get that 
$$G(x,y,s) \equiv \sum_{i=1}^n \left((y_i-s_i) \cdot \bar{q}_{ri}(x_{-i},y_{-i},s_{-i})\right)^2 \mod{I_{0}[x,y,s]}.$$ 
As $\deg(\bar{q}_{ri}(x_{-i},y_{-i},s_{-i})) \leq t+1$, the result follows.

For the converse, we have that:
\begin{align} \label{eq:G.eq}
G(x,y,s)\equiv \sum_{r} q_r(x,y,s)^2 \mod{I_0[x,y,s]}
\end{align}
where $q_r \in \mathbb{R}[x,y,s]$ and $\deg(q_r) \leq t+2$. As mentioned in Remark \ref{remark:normal.form.qr}, we can take $q_r$ to be in normal form with respect to $I_0[x,y,s];$ in particular, we can assume that it is multilinear. 

It is not hard to see that $G(x,y,x) =0 $ for any $x,y$ and so $q_r(x,y,x)\equiv 0$ mod $I_0[x,y,x]$ for any $r$. Let $i \in \{1,\ldots,n\}$. Using this and the fact that $q_r$ is multilinear, we write, for fixed $r$:
\begin{align*}
q_r(x,y,[x]^{i=s_i})&=(1-x_i)y_is_i q_r([x]^{i=0}, [y]^{i=1}, [x]^{i=1})+x_iy_i(1-s_i)q_r([x]^{i=1}, [y]^{i=1}, [x]^{i=0})\\
&+(1-x_i)(1-y_i)s_i q_r([x]^{i=0}, [y]^{i=0}, [x]^{i=1})+x_i(1-y_i)(1-s_i) q_r([x]^{i=1}, [y]^{i=0}, [x]^{i=0})\\
&\equiv (x_i-s_i) q_r([x]^{i=1}, [y]^{i=0}, [x]^{i=0}) \mod{I_0[x,y,s]}.
\end{align*}
We have that $\deg(q_r(x,y,[x]^{i=s_i})) \leq t+2$ as a polynomial in $\mathbb{R}[x,y,s].$ Following a similar argument as the previous proof regarding normal forms, it follows that $\deg(q_r([x]^{i=1}, [y]^{i=0}, [x]^{i=0})) \leq t+1$. 

Now, from Lemma \ref{lem:deriv.multilin}, note that 
\begin{align*} 
\frac{\partial F(x)}{\partial x_i} -\frac{\partial F(y)}{\partial x_i}=F([x]^{i=1})+F([y]^{i=0})-F([y]^{i=1})-F([x]^{i=0})
=G([x]^{i=1},[y]^{i=0}, [x]^{i=0}).
\end{align*}
Plugging \eqref{eq:G.eq} into this, we get:
\begin{align*}
\frac{\partial F(x)}{\partial x_i} -\frac{\partial F(y)}{\partial x_i}&\equiv G([x]^{i=1},[y]^{i=0}, [x]^{i=0})\\
&\equiv \sum_r q_r^2([x]^{i=1},[y]^{i=0}, [x]^{i=0}) \mod{I_0[[x]^{i=1},[y]^{i=0}, [x]^{i=0}]}
\end{align*}
It can easily be shown that $I_0[[x]^{i=1},[y]^{i=0}, [x]^{i=0} \subseteq I_1[x_{-i},y_{-i}]$. Thus, 
\begin{align*}
\frac{\partial F(x)}{\partial x_i} -\frac{\partial F(y)}{\partial x_i} \equiv \sum_r q_r^2([x]^{i=1},[y]^{i=0}, [x]^{i=0}) \mod{I_1[x_{-i},y_{-i}]}
\end{align*}
As $\deg(q_r([x]^{i=1}, [y]^{i=0}, [x]^{i=0})) \leq t+1$, the result follows. 
\Halmos
\endproof

\proof{Proof of Theorem \ref{thm:equiv.sos.submod.v}.}
Suppose that $f$ is $t$-sos-submodular. Using Lemma \ref{lem:1.to.2}, we get that 
$$F(x)-F(y)+\sum_{i=1}^n (y_i-x_i)\frac{\partial F}{\partial x_i}(x)=\sum_{i=1}^n (y_i-x_i) \left( \frac{\partial F}{\partial x_i}(x)- \frac{\partial F}{\partial x_i}([xy]^i)\right).$$
From Theorem \ref{thm:equiv.sos.submod.ii}, as $f$ is $t$-sos-submodular, we have, for any $i \in \{1,\ldots,n\},$
$$\frac{\partial F}{\partial x_i}(x)- \frac{\partial F}{\partial x_i}([xy]^i) \equiv \sum_r q_{ri}^2(x_{-i},[xy]^i_{-i}) \mod I_1[x_{-i}, [xy]_{-i}^i],$$
where $\deg(q_{ri}(x_{-i},[xy]^i_{-i})) \leq t+1.$
Renaming $q_{ri}(x_{-i},[xy]^i_{-i})$ to $\bar{q}_{ri}(x_{-i},y_{-i})$ and noting that $\deg(\bar{q}_{ri}) \leq t+1$ and $I_1[x_i,[xy]_{-i}^i] \subseteq I_1[x_{-i},y_{-i}]$, we get that, for any $i \in \{1,\ldots,n\},$
$$\frac{\partial F}{\partial x_i}(x)- \frac{\partial F}{\partial x_i}([xy]^i) \equiv \sum_r \bar{q}_{ri}^2(x_{-i},y_{-i}) \mod I_1[x_{-i}, y_{-i}].$$
As $I_1[x_{-i}, y_{-i}] \subseteq I_1[x,y]$ and $y_i-x_i\equiv (y_i-x_i)^2$ mod $I_1[x,y]$ for any $i$, it follows that 
$$F(x)-F(y)+\sum_{i=1}^n (y_i-x_i)\frac{\partial F}{\partial x_i}(x)\equiv \sum_{i=1}^n \sum_r \left( (y_i-x_i)\cdot \bar{q}_{ri}(x_{-i},y_{-i})\right)^2 \mod{I_1[x,y]}$$
and we obtain the result.

For the converse, suppose that 
\begin{align} \label{eq:def.nemhauser}
H_1(x,y) \equiv \sum_r q_r^2(x,y) \mod I_1[x,y]
\end{align}
for $q_r \in \mathbb{R}[x,y]$ and $\deg(q_r)\leq t+2.$ As in Remark \ref{remark:normal.form.qr}, we can assume that $q_r$ is in normal form with respect to $I_1[x,y].$ In particular, $q_r$ is multilinear. It is straightforward to see that $H(x,x)=0$, thus $q_r(x,x)\equiv 0 \mod{I_1[x,y]}$ for any $r$. Likewise, $H([x]^{i=0},[x]^{i=1})=0$, which implies that $q_r([x]^{i=0},[x]^{i=1})\equiv 0 \mod{I_1[x,y]}.$ Now, fix $i,j \in \{1,\ldots,n\}$ with $i\neq j.$ Recalling the notation in Definition \ref{def:conv.not.proof} and multilinearity of $q_r$ and using the equivalences above, we have:
\begin{align*}
q_r(x,[x]^{i=y_i,j=y_j})&\equiv (1-x_i)(1-x_j)y_iy_j q_r([x]^{i=0,j=0},[x]^{i=1,j=1}) \mod I_1[x,[x]^{i=y_i,j=y_j}]\\
&\equiv (y_i-x_i)(y_j-x_j)q_r([x]^{i=0,j=0},[x]^{i=1,j=1}) \mod I_1[x,[x]^{i=y_i,j=y_j}].
\end{align*}
From this, we obtain, as done in previous proofs, that $\deg(q_r([x]^{i=0,j=0},[x]^{i=1,j=1}))\leq t.$ Now, note that:
$$-\frac{\partial^2 F(x)}{\partial x_i \partial x_j}=H_1([x]^{i=0,j=0}, [x]^{i=1,j=1}) \equiv \sum_r  q_r^2([x]^{i=0,j=0}, [x]^{i=1,j=1}) \mod I_1[[x]^{i=0,j=0}, [x]^{i=1,j=1}],$$
where we have used Lemma \ref{lem:deriv.multilin} and the equivalence from \eqref{eq:def.nemhauser}. As $I_1[[x]^{i=0,j=0}, [x]^{i=1,j=1}] \subseteq I_2[x_{-i,j}]$, and we have $\deg(q_r([x]^{i=0,j=0},[x]^{i=1,j=1}))\leq t$, it follows that $f$ is $t$-sos-submodular.
\Halmos
\endproof

\proof{Proof of Theorem \ref{thm:equiv.sos.submod.vi}.}
Suppose that $f$ is $t$-sos-submodular. Similarly to the proof of Theorem \ref{thm:equiv.sos.submod.v}, we can show that:
$$F(y)-F(x)-\sum_{i=1}^n (y_i-x_i)\frac{\partial F(y)}{\partial x_i}=\sum_{i=1}^n (y_i-x_i) \left( \frac{\partial F([yx]^i)}{\partial x_i}-\frac{\partial F(y)}{\partial x_i}\right).$$
The proof of the implication is then identical to that of Theorem \ref{thm:equiv.sos.submod.v}.

For the converse, note that, once again, $$-\frac{\partial^2 F(x)}{\partial x_i \partial x_j}=H_2([x]^{i=0,j=0},[x]^{i=1,j=1}).$$ 
The proof of the converse is then also identical to that of Theorem \ref{thm:equiv.sos.submod.v}.
\Halmos
\endproof

\proof{Proof of Theorem \ref{thm:equiv.sos.submod.iv}.} Suppose that $f$ is $t$-sos-submodular and let $i \in \{1,\ldots,n\}$. From Theorem~\ref{thm:equiv.sos.submod.ii}, we have that:
\begin{equation}\label{eq:def.partial.2}
\frac{\partial F(x)}{\partial x_i}-\frac{\partial F(y)}{ \partial x_i} \equiv \sum_r q_r^i(x_{-i},y_{-i})^2 \mod I_1[x_{-i},y_{-i}]
\end{equation} where $q_r^i \in \mathbb{R}[x_{-i},y_{-i}]$ and $\deg(q_r^i) \leq t+1$.
Using Lemma \ref{lem:1.to.2}, we have that:
\begin{align} \label{eqref:eq.diff1}
F(x+y-s)-F(x)=\sum_{i=1}^n (y_i-s_i) \frac{\partial{F}([(x+y-s)x]^i)}{\partial x_i}.
\end{align}
We also have:
\begin{align} \label{eqref:eq.diff2}
F(x+y-s)-F(y)=\sum_{i=1}^n (x_i-s_i) \frac{\partial F([(x+-s)y]^i}{\partial x_i}.
\end{align}
Now, using \eqref{eq:def.partial.2}, we get:
\begin{align*}
 \frac{\partial{F}([(x+y-s)x]^i)}{\partial x_i} \equiv \frac{\partial F(x)}{\partial x_i}- \sum_r q_r^i(x_{-i},[(x+y-s)x]_{-i}^i)^2 \mod{I_1[x_{-i},[(x+y-s)x]_{-i}^i]},
\end{align*}
where $\deg(q_r^i(x_{-i},[(x+y-s)x]_{-i}^i)) \leq t+1.$
Likewise,
\begin{align*}
 \frac{\partial F([(x+y-s)y]^i)}{\partial x_i}\equiv  \frac{\partial{F}(x+y-s)}{\partial x_i} + \sum_r &q_r^i([(x+y-s)y]^i_{-i},(x+y-s)_{-i})^2\\
 &\mod{I_1[[(x+y-s)y]^i_{-i},(x+y-s)_{-i}]},
\end{align*}
where $\deg(q_r^i([(x+y-s)y]^i_{-i},(x+y-s)_{-i})) \leq t+1.$
Plugging these equivalences back into \eqref{eqref:eq.diff1} and \eqref{eqref:eq.diff2} and noting that:
$$ I_1[x_{-i},[(x+y-s)x]_{-i}^i] \subseteq I_0[x,y,s] \text{ and }I_1[[(x+y-s)y]^i_{-i},(x+y-s)_{-i}] \subseteq I_0[x,y,s],$$
we get the following equivalences, modulo $I_0[x,y,s],$
\begin{align*}
F(x+y-s)-F(x)&\equiv \sum_{i=1}^n (y_i-s_i) \left( \frac{\partial F(x)}{\partial x_i} -\sum_r q_r^i(x_{-i},[(x+y-s)x]_{-i}^i)^2 \right)\\
F(x+y-s)-F(y)&\equiv \sum_{i=1}^n (x_i-s_i) \left( \frac{\partial{F}(x+y-s)}{\partial x_i}+ \sum_r q_r^i([(x+y-s)y]^i_{-i},(x+y-s)_{-i})^2 \right)
\end{align*}
Subtracting the first one from the second one, we obtain:
\begin{align*}
G_1(x,y,s) \equiv &\sum_{i=1}^n (y_i-s_i)\sum_r q_r^i(x_{-i},[(x+y-s)x]_{-i}^i)^2 \\
&+\sum_{i=1}^n (x_i-s_i)\sum_r q_r^i([(x+y-s)y]^i_{-i},(x+y-s)_{-i})^2 \mod{I_0[x,y,s]}.
\end{align*}
Noting that $(y_i-s_i)\equiv (y_i-s_i)^2 \mod I_0[x,y,s]$ and $(x_i-s_i) \equiv (x_i-s_i)^2 \mod{I_0[x,y,s]}$, we get that, modulo $I_0[x,y,s],$
\begin{align*}
G_1(x,y,s) \equiv &\sum_{i=1}^n \sum_r \left( \left((y_i-s_i) q_r^i(x_{-i},[(x+y-s)x]_{-i}^i) \right)^2+ \left( (x_i-s_i) q_r^i([(x+y-s)y]^i_{-i},(x+y-s)_{-i}) \right)^2 \right). 
\end{align*}
The result follows as the degree of each term in the square is less than or equal to $t+2$. 

For the converse, we have
$$G_1(x,y,s)\equiv \sum_r q_r^2(x,y,s) \mod I_0[x,y,s]$$
for some polynomials $q_r \in \mathbb{R}[x,y,s]$ with $\deg(q_r)\leq t+2.$ We remark that $G_1(x,y,x)=H_1(x,y)$. Furthermore, $I_0[x,y,x]=I_1[x,y].$ Thus, setting $s=x$ and renaming $q_r(x,y,x)$ to $\bar{q}_r(x,y)$ in the above equation gives us:
$$H_1(x,y)= G_1(x,y,x)\equiv \sum_r \bar{q}_r^2(x,y) \mod I_1[x,y],$$
where $\deg(\bar{q}_r) \leq t+2$. This gives us the result from Theorem \ref{thm:equiv.sos.submod.v}.
\Halmos
\endproof

\proof{Proof of Theorem \ref{thm:equiv.sos.submod.vii}.}
For the implication, using Lemma \ref{lem:1.to.2}, 
\begin{align*}
    F(x)-F(s)=\sum_{i=1}^n (x_i-s_i) \frac{\partial F([xs]^i)}{\partial x_i}   \text{ and }   F(y)-F(s)=\sum_{i=1}^n (y_i-s_i)\frac{\partial F([ys]^i)}{\partial x_i}.
\end{align*}
Then as $f$ is $t$-sos-submodular, from Theorem \ref{thm:equiv.sos.submod.ii}, \eqref{eq:def.partial.2} holds, which enables us to write
\begin{align*}
\frac{\partial F([xs]^i)}{\partial x_i} &\equiv \frac{\partial F(s)}{\partial x_i}-\sum_r q_{ri}(s,[xs]^i)^2 \mod I_1[s_{-i},[xs]^i_{-i}]\\
\frac{\partial F([ys]^i)}{\partial x_i} &\equiv \frac{\partial F(y)}{\partial x_i}+\sum_r q_{ri}(s,[ys]^i)^2 \mod I_1[s_{-i},[ys]^i_{-i}].
\end{align*}
for any $i\in \{1,\ldots,n\}.$ The conclusion follows in an identical way to the proof of Theorem \ref{thm:equiv.sos.submod.iv}.

For the converse, set, once again $s=x$ and note that $G_2(x,y,x)=H_2(x,y).$ The conclusion follows in an identical way to the proof of Theorem \ref{thm:equiv.sos.submod.iv}.
\Halmos
\endproof

\proof{Proof of Proposition \ref{prop:size}.} It is easy to see that Theorem \ref{thm:equiv.sos.submod.ii} involves $n$ sos constraints as opposed to Theorems \ref{thm:equiv.sos.submod.i}--\ref{thm:equiv.sos.submod.vii} which involve just one. To establish the size of the semidefinite programs under consideration, we consider the number of monomials in $(x,y,s)$ (resp. $(x,y)$) of degree at most $t$ corresponding to the Gr\"obner basis for $I_0[x,y,s]$ (resp. $I_1[x,y]$). 

For the case of $I_1[x,y],$ consider any monomial of degree $k \leq t$. It contains at most one of either $x_i$ or $y_i$ for fixed $i \in \{1,\ldots,n\}.$ Indeed, $x_iy_i \equiv x_i \mod I_1[x,y]$ and thus any appearance of $x_i$ and $y_i$ simultaneously will reduce to $x_i.$ As a consequence, one needs to pick $k$ indexes out of the $n$ possible and choose either $x_i$ or $y_i$ for each index. This corresponds to 2 choices. Thus, there are $\binom{n}{k} 2^k$ possible monomials of degree $k.$ As we are interested in the number of monomials of degree less than or equal to $t$, we get $\sum_{k=0}^t \binom{n}{k} 2^k.$ This explains the size of the semidefinite programs involved in Theorems \ref{thm:equiv.sos.submod.ii}, \ref{thm:equiv.sos.submod.v}, and \ref{thm:equiv.sos.submod.vi}.

For the case of $I_0[x,y,s],$ consider once again any monomial of degree $k \leq t.$ As before, it contains at most one of either $x_i,y_i$ or $s_i$ for fixed $i \in \{1,\ldots,n\}.$ Indeed, modulo $I_0[x,y,s]$, $x_iy_i\equiv s_i$, $x_is_i\equiv s_i$, and $y_is_i \equiv s_i.$ As a consequence, one needs to pick $k$ indexes out of $n$ and choose either $x_i, y_i$ or $s_i$ for each index. This corresponds to $3$ choices. Thus, there are $\binom{n}{k} 3^k$ possible monomials of degree $k$ and a total of $\sum_{k=0}^t \binom{n}{k} 3^k$ monomials of degree at most $t$. This explains the size of the semidefinite programs involved in Theorems \ref{thm:equiv.sos.submod.i}, \ref{thm:equiv.sos.submod.iv}, and \ref{thm:equiv.sos.submod.vii}.
\Halmos
\endproof

\subsection{Proofs of Results in Section \ref{subsec:ops.preserving.sos.submod}.} \label{appendix:ops.preserving.sos.submod}

\proof{Proof of Proposition \ref{prop:ops.maintain.sos.submod}.}
Nonnegative weighted sums is immediate from the fact that the set of $t$-sos polynomials are a cone and that differentiation is a linear operation. 
In the remainder, we let $F$ be the MLE of $f$ and $G$ be the MLE of $g$.

For complementation, we have $G(x)=F(\textbf{1}-x)$ for all $x \in \{0,1\}^n,$ where $\textbf{1}$ is the vector of size $n$ of all ones. Let $i,j \in \{1,\ldots,n\}$, we have:
$$-\frac{\partial G_(x)}{\partial x_i \partial x_j}=-\frac{\partial F(\textbf{1}-x)}{\partial x_i \partial x_j}.$$
As a $t$-sos polynomial remains $t$-sos when composed with an affine mapping, the result holds.

For restriction, setting $a=1_A$, we have $G(x)=F(a \circ x)$. Thus, for  $i,j \in \{1,\ldots,n\}$, we have:
$$-\frac{\partial G_(x)}{\partial x_i \partial x_j}=- a_ia_j\frac{\partial F(a\circ x)}{\partial x_i \partial x_j}.$$
Note that $a_ia_j \geq 0$ as $a_i,a_j \in \{0,1\}$. The result holds for the same reason as complementation. 

For contraction, setting $a=1_A$, we have $G(x)=F(a +x-a\circ x)-F(a)$.  Thus, for  $i,j \in \{1,\ldots,n\}$, we have:
$$-\frac{\partial G_(x)}{\partial x_i \partial x_j}=- (1-a_i)(1-a_j)\frac{\partial F(a\circ x)}{\partial x_i \partial x_j}.$$
Note that $(1-a_i)(1-a_j) \geq 0$ and the result follows similarly to restriction.
\Halmos
\endproof

The following lemmas are needed to show Proposition \ref{prop:ops.do.not.maintain.sos.submod}.

\begin{lemma} \label{prop:deg.n.budg.add}
Let $w_1=w_2=\ldots=w_n=2$ and let $B=2n-1$. Consider the set function $f(S)=\min \{B, \sum_{i \in S} w_i\}$. Its multilinear extension $F$ is given by $F(x)=\sum_{i=1}^n w_ix_i-\prod_{i=1}^n x_i$.
\end{lemma}

We use the notation from Section \ref{subsec:set.funcs.extensions} in the proof. Namely, we write $$F(x)=\sum_{T \subseteq \Omega} a(T) \prod_{i \in T} x_i, \text{ where }a(T)=\sum_{S\subseteq T}(-1)^{t-s} f(S), \forall T \subset \Omega.$$

\proof{Proof.}
Let $\tilde{B}=2n$ and let $\tilde{f}(S)=\min\{ \tilde{B}, \sum_{i \in S} w_i\}=\sum_{i \in S} w_i.$  Let $\tilde{F}$ be the MLE of $\tilde{f}$ with coefficients $\tilde{a}_T$ for $T \subseteq \Omega.$ It is quite straightforward to see that $\tilde{F}(x)=\sum_{i=1}^n w_i x_i$. 
Now, note that $f$ and $\tilde{f}$ coincide on all sets $S$ with $|S|\leq n-1$ as $\min\{B,\sum_{i \in S} w_i\}=\sum_{i \in S} w_i$ for any $|S|<n.$ However, $f(\Omega)=2n-1$ and $\tilde{f}(\Omega)=2n$. It follows that 
$$a(T)=\sum_{S \subseteq T} (-1)^{t-s} f(S)=\sum_{S \subseteq T} (-1)^{t-s} \tilde{f}(S)=\tilde{a}(T), \forall |T|<n.$$
For $T=\Omega$, we then have
\begin{align*}
a(\Omega)&=\sum_{T \subseteq \Omega} (-1)^{n-|T|} f(T)=\sum_{T \subseteq \Omega, |T| \leq n-1} (-1)^{n-|T|} f(T)+(-1)^0 f(\Omega)\\
&= \sum_{T \subseteq \{1,\ldots,n\}, |T| \leq n-1} (-1)^{n-|T|} \tilde{f}(T)+\tilde{f}(\Omega)- \tilde{f}(\Omega)+ f(\Omega)\\
&= \tilde{a}(\Omega)-2n+2n-1=-1,
\end{align*}
which gives us the result.
\Halmos
\endproof

\begin{lemma} \label{lem:convolution}
For $n>1$, let $f:2^{\Omega}\rightarrow \mathbb{R}$ be the set function with MLE:
$$F(x)=1+\frac{n-1}{n}\sum_{i=1}^n x_i-\frac{2}{n-1}\sum_{1 \leq i<j \leq n} x_ix_j. $$
We have that $F(0)=1, F(1)=0,$ and $F(x)>1$ for all $x \neq 0,1.$
\end{lemma}

\proof{Proof.}
It is straightforward to see that $F(0)=1$ and $F(1)=1+\frac{n-1}{n}\cdot n-\frac{2}{n-1}\cdot \frac{n(n-1)}{2}=0.$ Now, let $x\neq 0,1$  and suppose that $x$ has $k$ ones and $n-k$ zeros. We have that $\sum_{i=1}^n x_i=k$ and $\sum_{1 \leq i<j \leq n} x_ix_j=\binom{k}{2}$, thus:
$$F(x)=1+\frac{n-1}{n}\cdot k-\frac{2}{n-1}\frac{k(k-1)}{2}=1+k\left(\frac{n-1}{n}-\frac{k-1}{n-1} \right).$$
As $n-1 \geq k$, it can be shown that $\frac{n-1}{n}-\frac{k-1}{n-1} >0$, and so $F(x)>1$ for all $x \neq 0,1.$
\Halmos
\endproof

\begin{lemma} \label{lem:monotonization}
For $n>1$, let $f:2^{\Omega}\rightarrow \mathbb{R}$ be the set function with MLE:
$$F(x)=2\sum_{i=1}^n x_i-\frac{4n-2}{n(n-1)}\sum_{1 \leq i<j \leq n} x_ix_j. $$
We have that $F(0)=0, F(1)=1,$ and $F(x)>1$ for all $x \neq 0,1.$
\end{lemma}
\proof{Proof.}
It is straightforward to see that $F(0)=0$ and $F(1)=2n-\frac{4n-2}{n(n-1)}\cdot \frac{n(n-1)}{2}=1.$ Now let $x \neq 0,1$ and suppose that $x$ has $k$ ones as above. As before, we obtain
$$F(x)=2k-\frac{4n-2}{n(n-1)}\frac{k(k-1)}{2}.$$
We define $h(k)=2k-\frac{4n-2}{n(n-1)}\frac{k(k-1)}{2}$ for $1\leq k\leq n-1$ and show that $h(k)>1$ for all $k.$ This implies that $F(x)>1$ for all $x \neq 0,1.$ Note first that $h$ is concave in $k$ by, e.g., considering its second derivative. It follows that its minimum is at $k=1$ or $k=n-1.$ As $h(1)=2>1$ and $$h(n-1)=2(n-1)-\frac{4n-2}{n(n-1)}\cdot \frac{(n-1)(n-2)}{2}=3-\frac{2}{n}>1,$$
we obtain our result.
\Halmos
\endproof

\begin{lemma} \label{lem:example.n-2.sos}
Let $F(x)=x_1x_2\ldots x_n$ for all $x \in \{0,1\}^n.$ We have that $F$ is $n$-sos modulo $I_2[x]$ but not $(n-1)$-sos modulo $I_2[x].$
\end{lemma}
\proof{Proof.}
To see that $F$ is $n$-sos modulo $I_2[x]$, simply note that $F(x)\equiv (x_1x_2\ldots x_n)^2$ modulo $I_2[x].$ To show that $F$ is not $(n-1)$-sos mod $I_2[x],$ we proceed by contradiction. Suppose that
$F(x)\equiv \sum_r q_r(x)^2 \mod I_2[x]$
for some polynomials $q_r$ with degree strictly less than $n$. As $F(x)=0$ for any $x \neq 1,$ it follows that $q_r(x)=0$ for any $x \neq 1$ and all coefficients of $q_r$ involving monomials containing less than or equal to $n-1$ variables are equal to zero; see, e.g., \citep{grabisch2000equivalent}. As $q_r(x)$ does not contain the monomial $x_1x_2\ldots x_n$ (otherwise it would be of degree $n$), this implies that $q_r(x)=0$ for all $x$. But this contradicts the fact that $F(x)=1$ when $x=1.$
\Halmos
\endproof

\proof{Proof of Proposition \ref{prop:ops.do.not.maintain.sos.submod}.}
We start first with partial minimization. For $n\in \mathbb{N}, B=2n-1,$ and $w_1=\ldots=w_n=2,$ let $f:2^{n+1}\rightarrow \mathbb{R}$ be the set function with MLE:
$$F:(x,z) \in \{0,1\}^n\times \{0,1\} \mapsto Bz+\left(\sum_{i=1}^n w_ix_i\right) \cdot (1-z)$$
and let $g:2^n\rightarrow \mathbb{R}$ be the set function with MLE:
$G:x\in \{0,1\}^n \mapsto \min_{z \in \{0,1\}} G(x,z).$
It is straightforward to check that $f$ is $0$-sos-submodular by noting that 
$-\partial F(x,z)/\partial x_i\partial x_j=0$, for any $i\neq j$, and $-\partial F(x,z) /\partial x_i \partial z=w_i=\sqrt{2}^2$ for any $i.$ Now, note that $g(x)=\min\{B,\sum_i w_ix_i\}.$ Indeed if $B \leq \sum_i w_i x_i$ then $z=1$ and $G(x)=B$ and conversely, if $B\geq \sum_{i} w_i x_i$ then $z=0$ and $G(x)=\sum_i w_i x_i.$ Thus following Lemma \ref{prop:deg.n.budg.add}, we have that $G(x)=\sum_{i=1}^n w_ix_i-\prod_{i=1}^n x_i$ and $-\frac{\partial^2 G(x)}{\partial x_i \partial x_j}=\prod_{k \neq i,j} x_k$. The result follows from Lemma \ref{lem:example.n-2.sos}.

We now consider convolution. Let $f$ be as defined in Lemma \ref{lem:convolution} and take $u=0$. It is easy to see from Lemma \ref{lem:convolution} that $g(S)=\min_{T \subseteq S} f(T)$ is such that $g(S)=1$ for any $S\neq \Omega$ and $g(\Omega)=0.$ Thus, its MLE is equal to
$$G(x)=\sum_{S\neq \Omega} \left( \prod_{i \in S} x_i \prod_{i\notin S}(1-x_i)\right)=\sum_{S} \left( \prod_{i \in S} x_i \prod_{i\notin S}(1-x_i)\right)-\prod_{i=1}^n x_i=1-\prod_{i=1}^n x_i.$$ Taking second-order derivatives, we obtain: 
$$-\frac{\partial^2 G(x)}{\partial x_i \partial x_j}=\prod_{k=1,k\neq i,j}^n x_k.$$
The result follows from Lemma \ref{lem:example.n-2.sos}.

We finish with monotonization. Let $f$ be a set function as defined in Lemma \ref{lem:monotonization}. It is easy to see that $f$ is $0$-sos-submodular as its second-order derivatives are negative numbers. Now define $g(S)=\min_{T\supseteq S} f(T)$ for any $S\in \Omega.$ It is easy to see from Lemma \ref{lem:monotonization} that $g(\emptyset)=0$ but that $g(S)=F(1)=1$ for any set $S \neq \emptyset.$ Thus, its MLE is equal to:
$$G(x)=\sum_{S\neq \emptyset} \left( \prod_{i \in S} x_i \prod_{i\notin S}(1-x_i)\right)=\sum_{S} \left( \prod_{i \in S} x_i \prod_{i\notin S}(1-x_i)\right)-\prod_{i=1}^n (1-x_i)=1-\prod_{i=1}^n (1-x_i).$$
Taking second-order derivatives, we obtain: 
$$-\frac{\partial^2 G(x)}{\partial x_i \partial x_j}=\prod_{k=1,k\neq i,j}^n (1-x_k).$$
Note that for any polynomial $p$ over $\{0,1\}^n$, $p(x)$ is $t$-sos modulo $I_2[x]$ if and only if $p(1-x)$ is $t$-sos modulo $I_2[x].$ As $-\frac{\partial^2 G(1-x)}{\partial x_i \partial x_j}=\prod_{k \neq i,j} x_i$, the result follows from Lemma \ref{lem:example.n-2.sos}.
\Halmos
\endproof

\subsection{Proofs of Results in Section \ref{subsec:gap.sos.submod.submod}.}\label{appendix:gap.sos.submod.submod}

\proof{Proof of Proposition \ref{prop:low.degrees}.} Let $f:2^{\Omega} \rightarrow \mathbb{R}$ be a set function of degree $d$ and let $F$ be its MLE. When $d=0$, $F$ is a constant, and it is straightforward to check that, e.g., $(i')$ in Proposition \ref{prop:submod.mle} holds. Thus, $f$ is always submodular. Likewise, when $d=1$, $F$ is an affine function, that is $F(x)=a^Tx+b$ where $a \in \mathbb{R}^n, b\in \mathbb{R}.$ It is once again straightforward to check that $(i')$ in Proposition \ref{prop:submod.mle} holds. Indeed, we have
$$F(x)+F(y)-F(x+y-x\circ y)-F(x\circ y)=a^Tx+b+a^Ty+b-a^T(x+y-x \circ y)-b-a^T(x\circ y)-b=0.$$
Thus, $f$ is also submodular (in fact, it is modular). 

Now, let $d=2$ and suppose that $f$ is submodular. It follows from Proposition \ref{prop:submod.mle} $(iii')$ that $F_2^{ij}(x_{-i,j})\mathrel{\mathop{:}}=-\frac{\partial^2 F(x)}{\partial x_i \partial x_j} \geq 0$ for all $i,j \in \{1,\ldots,n\}$. As $\deg(f)=2,$ $F_2^{ij}$ is in fact equal to a nonnegative constant $C^{ij}$. Hence, we can write $$F_2^{ij}(x_{-i,j})\equiv \left(\sqrt{C^{ij}}\right)^2 \mod{I_2[x_{-i,j}]},$$
which shows that $f$ is $0$-sos-submodular.

Finally, let $d=3.$ \cite{billionnet1985maximizing} propose the following characterization of cubic submodular functions in Lemma 1: a cubic submodular set function $f$ is submodular if and only if its MLE $F$ can be written:
$$F(x)=\sum_{j \in J^-} c_j x_j-\sum_{j \in J^+} c_j x_j -\sum_{i,j \in IJ} c_{ij} x_ix_j -\sum_{(i,j,k) \in IJK^+} c_{ijk} x_ix_jx_+\sum_{(i,j,k) \in IJK^-} c_{ijk} x_ix_jx_k+K,$$
where $c_j \geq 0,\forall j \in J^-$, $c_j \geq 0, \forall j \in J^+,$ $c_{ij} \geq 0, \forall (i,j)\in IJ,$ $c_{ijk} \geq 0, \forall (i,j,k)\in IJK^+\cup IJK^-$, and $c_{ij} \geq \sum_{k~|~ (i,j,k)\in IJK^-} c_{ijk},\forall (i,j)\in \{1,\ldots,n\}^2, i<j$. Fix $i,j \in \{1,\ldots,n\}$, we have that:
\begin{align*}
-\frac{\partial^2 F(x)}{\partial x_i \partial x_j}&=c_{ij}+\sum_{k~|~(i,j,k)\in IJK^+ } c_{ijk}x_k-\sum_{k~|~(i,j,k)\in IJK^-} c_{ijk}x_k\\
&=c_{ij}-\sum_{k~|~ (i,j,k)\in IJK^-} c_{ijk}+\sum_{k~|~ (i,j,k)\in IJK^-} c_{ijk}+\sum_{k~|~(i,j,k)\in IJK^+ } c_{ijk}x_k-\sum_{k~|~(i,j,k)\in IJK^-} c_{ijk}x_k\\
&\equiv \left(\sqrt{c_{ij}-\sum_{k~|~ (i,j,k)\in IJK^-} c_{ijk}}\right)^2+\sum_{k~|~(i,j,k)\in IJK^+ } \left(\sqrt{c_{ijk}}x_k\right)^2+\sum_{k~|~(i,j,k)\in IJK^-} \left(\sqrt{c_{ijk}}(1-x_k)\right)^2,
\end{align*}
where the last equivalence is mod $I_2[x_{-i,j}]$ and uses the assumptions on the coefficients of $F$ given by \citep{billionnet1985maximizing}. This shows that $f$ is $1$-sos-submodular.
\Halmos
\endproof
The next lemma is implied by the the results in \citep{laurent2003lower,fawzi2015sparse,sakaue2016exact}. We restate it for clarity. It is needed for the proof of Proposition \ref{prop:CE.deg.4}.

\begin{lemma}[\cite{laurent2003lower,fawzi2015sparse,sakaue2016exact}] \label{lem:CE}
Let $$p(x)=\begin{cases} &2\sum_{1 \leq i<j \leq n}x_ix_j-(n-1)\sum_{i=1}^n x_i+\frac{n^2-1}{4}\text{ if $n$ odd}\\
&2\sum_{1 \leq i<j \leq n}x_ix_j-(n-2)\sum_{i=1}^n x_i+\frac{n(n-2)}{4}, \text{ if $n$ even}
\end{cases}, ~\text{ for } x\in \{0,1\}^n.$$
We have that $p$ is nonnegative (in fact $\left\lceil \frac{n}{2} \right\rceil$-sos when $n$ is odd and $\left\lceil \frac{n+1}{2} \right\rceil$-sos when $n$ is even) over $I_2[x]$ but not $\left(\left\lceil \frac{n}{2} \right\rceil -1\right)$-sos over $I_2[x].$
\end{lemma}

\proof{Proof.} Suppose first that $n$ is odd. Let $I[y]=\langle y_1^2-1,\ldots,y_n^2-1\rangle$ and for $y \in \{-1,1\}^n$, define $$q_{odd}(y)=\sum_{1 \leq i<j \leq n} y_iy_j+\left \lfloor \frac{n}{2} \right \rfloor.$$
\revision{Note that $\left \lfloor \frac{n}{2} \right \rfloor < \frac{n}{2}$} and thus $-\left \lfloor \frac{n}{2} \right \rfloor > -\frac{n}{2}$. From \cite[Section 3]{laurent2003lower}, it follows that $q(y)$ is not $\left \lceil \frac{n}{2} \right \rceil -1$-sos over $I[y].$ Further note that:
$$q_{odd}(y)\equiv \frac12 \left( \sum_{i=1}^n y_i \right)^2-\frac{n}{2}+\left\lfloor \frac{n}{2}\right\rfloor \mod I[y]. $$
As $n$ is odd, $q_{odd}(y)=\frac{1}{2} \left(\sum_{i=1}^n y_i^2 \right)-\frac{1}{2}.$ As $\left(\sum_{i=1}^n y_i^2 \right) \geq 1$ when $y \in \{-1,1\}^n, $ it follows that $q_{odd}(y)\geq 0$ for all $y \in \{-1,1\}^n.$ Thus, from \citep{fawzi2015sparse}, as $q_{odd}$ is a nonnegative form over $\{-1,1\}^n$ and quadratic, $q_{odd}(y)$ is $\left \lceil \frac{n}{2} \right \rceil$-sos.

Suppose now that $n$ is even and define
$$q_{even}(y)=\sum_{1 \leq i \leq j \leq n}y_iy_j+\sum_{i=1}^n y_i+\frac{n}{2}.$$
This polynomial is obtained from $q_{odd}$ by considering the contraction of the complete graph over $n+1$ nodes to the complete graph over $n$ nodes. Following \cite[Proposition 3 (i)]{laurent2003lower}, we obtain that $q(y)$ is not $\left \lceil \frac{n}{2} \right \rceil -1$-sos over $I[y].$
Further note that:
$$q_{even}(y)\equiv \frac12 \left( 1+\sum_{i=1}^n y_i \right)^2-1/2 \mod I[y]. $$
For $n$ even, as $\left(\sum_{i=1}^n 1+y_i^2 \right) \geq 1$ when $y \in \{-1,1\}^n, $ it follows that $q_{even}(y)\geq 0$ for all $y \in \{-1,1\}^n.$ Thus, from \cite{sakaue2016exact}, as $q_{even}$ is a quadratic function, $q_{even}(y)$ is $\left \lceil \frac{n+1}{2} \right \rceil$-sos.

To get the result, note that if $x \in \{0,1\}^n$, then $2x-1 \in \{-1,1\}^n$. Thus, as affine mappings maintain $t$-sos, the result holds.
\Halmos
\endproof

\proof{Proof of Proposition \ref{prop:CE.deg.4}.} We show the case when $n$ is odd. The case where $n$ is even is identical. We compute the second-order partial derivatives of $F$. For any $i,j \in \{1,\ldots,n-2\}, i\neq j$, we have
\begin{align*}
    &-\frac{\partial^2 F(x)}{\partial x_i \partial x_j} = 2x_{n-1}x_n \equiv 2x_{n-1}^2x_n^2 \mod I_2[x_{-i,j}].
    \end{align*}
    For any $i\in \{1,\ldots,n\},$ we have:
    \begin{align*}
    &-\frac{\partial^2 F(x)}{\partial x_i \partial x_{n-1}} = \left(-(n-3)+2\sum_{j=1,j \neq i}^{n-2}x_j\right)x_n+(n-3)\equiv (n-3)(1-x_n)^2+2\sum_{j=1,j \neq i}^{n-2}x_j^2x_n^2 \mod I_2[x_{-i,n-1}] \\
    &-\frac{\partial^2 F(x)}{\partial x_i \partial x_{n}} = \left(-(n-3)+2\sum_{j=1,j \neq i}^{n-2}x_j\right)x_{n-1}+(n-3)\equiv (n-3)(1-x_{n-1})^2+2\sum_{j=1,j \neq i}^{n-2}x_j^2x_{n-1}^2 \mod I_2[x_{-i,n}].
\end{align*}
Finally, we have:
\begin{align*}
    &-\frac{\partial^2 F(x)}{\partial x_{n-1} \partial x_{n}} = 2\sum_{1 \leq i<j \leq n-2} x_ix_j +\frac{n^2-4n+3}{4}-(n-3)\sum_{i=1}^{n-2} x_i.
\end{align*}
Note that $-\frac{\partial^2 F(x)}{\partial x_i \partial x_j}$ is $2$-sos for any $i,j \in \{1,\ldots,n-2\}$, as is $-\frac{\partial^2 F(x)}{\partial x_i \partial x_{n-1}}$ and $-\frac{\partial^2 F(x)}{\partial x_i \partial x_n}$ for any $i\in \{1,\ldots,n\}.$ However, following Lemma \ref{lem:CE}, $-\frac{\partial^2 F(x)}{\partial x_{n-1}\partial x_n}$ is $\lceil \frac{n-2}{2} \rceil$-sos over $I_2[x_{-n-1,n}]$ but not $\left(\left\lceil \frac{n-2}{2} \right\rceil -1\right)$-sos over $I_2[x_{-n-1,n}]$. As $n\geq 4$ (given that $d\geq 4$), it follows that $F$ is $\lceil \frac{n-2}{2} \rceil$-sos-submodular but not $\left(\left\lceil \frac{n-2}{2} \right\rceil -1\right)$-sos-submodular.
\Halmos
\endproof

\proof{Proof of Proposition \ref{prop:hypercut}.}
Note that the function
$$ F(x) = \sum_{e \in E} w_e \left[\left(1-\prod_{v\in e} x_v\right) - \prod_{v\in e}(1-x_v)\right]
$$
is the multilinear extension of the weighted cut function on $\mathcal{H}$ and that if $\max_{e \in E} |e| = t$, then $\mathrm{deg}(F) = t$. Let $E_{ij} = \{e \in E: i, j \in e\}$. We have $$-\frac{\partial^2 F(x)}{\partial x_i \partial x_j} = \sum_{e \in E_{ij}} w_e \left[\prod_{v\in e, v\neq i, j} x_v + \prod_{v\in e, v\neq i,j}(1-x_v)\right].$$

Let $F_e^{ij}(x) = \prod_{v\in e, v\neq i, j} x_v + \prod_{v\in e, v\neq i,j}(1-x_v)$. We show that $F_{e}^{ij}(x)$ is $2 \lfloor t/2-1 \rfloor$-sos. If $t$ is even, then it is easy to see that $\deg(F_e^{ij}) \leq t-2$ and when $t$ is odd, degree cancellations lead to $\deg(F_e^{ij}) \leq t-3$. Furthermore, as $F_e^{ij}$ takes on values $0$ or $1$ for all $x \in \{0,1\}^n$, we have that $F_e^{ij}(x)=\left(F_e^{ij} (x)\right)^2$ for all $x \in \{0,1\}^n$. As $w_e\geq 0, \forall e$, this shows that $F$ is $(t-2)$-sos-submodular when $t$ is even and $(t-3)$-sos-submodular when $t$ is odd, i.e., $2 \lfloor t/2-1 \rfloor$-sos-submodular.
\Halmos
\endproof

\proof{Proof of Proposition \ref{prop:coverage}.}
The multilinear extension of $f$ can be written as
    $$F(x) = \sum_{\ell = 1}^m\left(1 - \prod_{k: \ell \in A_k}(1-x_k)\right).$$ 
    Note that the degree of $F$ is equal to the maximum number of sets any element of $\Omega$ can belong to. Suppose $t = \max_{\ell = 1, \ldots, m} | \{ k :  \ell \in A_k \}|$. Then we have $$-\frac{\partial^2 F(x)}{\partial x_i \partial x_j} = \sum_{\ell=1}^m \prod_{k: \ell \in A_i \cap A_j \cap A_k} (1-x_k),$$ which is a multilinear function of degree $t-2$. Since the above is equal to either 0 or 1 for $x \in \{0,1\}^n$, using arguments identical to above, the coverage function is $t-2$-sos-submodular.
\Halmos
\endproof

\proof{Proof of Proposition \ref{prop:concave}.} Let $F$ be the MLE of $f$. We have that:
$$-\frac{\partial^2 F(x)}{\partial x_i \partial x_j}\equiv \phi(\sum_{k \neq i,j} x_k +1)+\phi(\sum_{k \neq i,j} x_k+1)-\phi(\sum_{k \neq i,j} x_{k}+2)-\phi(\sum_{k \neq i,j}x_k) \mod I_2[x_{-i,j}].$$
As $\phi$ is concave of degree $2d$ and univariate, this implies that $-\phi'' \geq 0$, of degree $2d-2$, and univariate. It then follows from e.g. \cite{blekherman2012semidefinite} that $-\phi''$ is sos, and so $g(s,t)=\phi(\frac12 s+\frac12 t)-\frac12 \phi(s)-\frac12 \phi(t)=\sum_{r} (q_r(t,s))^2$, where $q_r$ is a polynomial in $s,t \in \mathbb{R}^2$ of degree less than or equal to $d$ \citep{ahmadi2013complete}. Taking $s=\sum_{k \neq i,j} x_k+2$ and $t=\sum_{j \neq i,j} x_k$ in the previous expression, we obtain that 
$$-\frac{\partial^2 F(x)}{\partial x_i \partial x_j}\equiv \sum_r q_r^2\left(\sum_{k\neq i,j} x_k+2,\sum_{k \neq i,j} x_k\right) \mod I_2[x_{-i,j}]$$
and thus $f$ is $d$-sos-submodular.
\Halmos
\endproof

\section{Proof of Results in Section \ref{sec:applications}}
\label{appendix:apps}

\subsection{Proof of Results in Section \ref{subsec:approx.submod}} \label{appendix:proofs.approx.submod}

We start with the proof of Proposition \ref{prop:approx.submod}. To show this proof, we use the equivalent definition of the submodularity ratio $\gamma^*(f)$ as given in \citep[Definition 1]{bian2017guarantees}. As defined there, $\gamma^*(f)$ is the largest scalar $\gamma$ such that:
\begin{align} \label{eq:def.gamma.star}
\sum_{\omega \in L \backslash S} (f(S \cup \{\omega\})-f(S)) \geq \gamma (f(L \cup S)-f(S)), \forall L, S \subseteq \{1,\ldots,n\}. 
\end{align}

\proof{Proof of Proposition \ref{prop:approx.submod}.}
We show that $\gamma^*(f)$ is also the largest $\gamma$ such that:
\begin{align} \label{eq:def.gamma.star.mle}
\sum_{i=1}^n (y_i-x_i)\frac{\partial F(x)}{\partial x_i}+\gamma(F(x)-F(y))\geq 0, \forall x,y \in \{0,1\}^n, x\leq y.
\end{align}
Let $x,y \in \{0,1\}^n$ such that $x \leq y$ and suppose that \eqref{eq:def.gamma.star} holds. We set $S$ be the set such that $x=1_S$. We further set $L$ to be the set such that $y-x=1_L$ (note that $x\leq y$ and so $y-x \in \{0,1\}^n$). We have that $L \cap S=\emptyset$ and that $y=1_{L \cup S}$. Recalling that $F(1_S)=f(S)$ for any $S \in \{1,\ldots,n\}$, \eqref{eq:def.gamma.star} gives us:
$$\sum_{\omega \in L} (F(1_{S\cup \{\omega\}})-F(x)) \geq \gamma (F(y)-F(x)).$$
Now, as $\omega \in L$ if and only if there exists $i \in \{1,\ldots,n\}$ such that $y_i-x_i=1$, i.e., $y_i=1$ and $x_i=0$, we get:
$$\sum_{i=1}^n (y_i-x_i) (F(x+e_i)-F(x)) \geq \gamma (F(y)-F(x)),$$
which is equivalent from Lemma \ref{lem:deriv.multilin} to \eqref{eq:def.gamma.star.mle}.

Conversely, suppose now that \eqref{eq:def.gamma.star.mle} holds for any $x,y \in \{0,1\}^n$ and $x \leq y.$ We show that \eqref{eq:def.gamma.star.mle} holds. Let $S,L \subseteq \{1,\ldots,n\}.$ We let $y=1_{L \cup S}$ and $x=1_{S}.$ Note that $(y_i-x_i)=1$ for $i \in \{1,\ldots,n\}$ if and only if $i \in L \backslash S.$ Thus \eqref{eq:def.gamma.star.mle} gives us:
$$\sum_{i \in L \backslash S} (F(x+e_i)-F(x))+\gamma (f(S)-f(S\cup L)) \geq 0$$
which is equivalent to
$\sum_{i \in L \backslash S} (f(S\cup \{i\})-f(S)) \geq \gamma (f(S\cup L)-f(S)),$
which is none other than \eqref{eq:def.gamma.star}. Our conclusion follows.
\Halmos
\endproof

\proof{Proof of Proposition \ref{prop:truncation}.}
Note that we have: $F=F_k+T_k,$ thus, by definition of $m$, 
$$\frac{\partial F(x)}{\partial x_i}=\frac{\partial F_k(x)}{\partial x_i}+\frac{\partial T_k(x)}{\partial x_i} \geq \frac{\partial F_k(x)}{\partial x_i} +m, \forall x \in [0,1]^n.$$
Furthermore, by the triangle inequality,
\begin{align*}
|F(x)-F(y)|=|F_k(x)+T_k(x)-F_k(y)-T_k(y)| \leq |F_k(x)-F_k(y)|+|T_k(x)-T_k(y)|.
\end{align*}
Using the mean-value theorem and H\"older's theorem, it can be shown that $|T_k(x)-T_k(y)| \leq M ||x-y||_1, \forall x,y \in [0,1]^n.$
Thus $$|F(x)-F(y)| \leq |F_k(x)-F_k(y)| + M ||x-y||_1, \forall x,y \in [0,1]^n.$$
As $F(x)-F(y) \leq 0$ for $x\leq y$ and $\gamma \geq 0$, combining the bounds above, we have: 
\begin{align*}
&\sum_{i=1}^n (y_i-x_i) \frac{\partial F(x)}{\partial x_i}+\gamma (F(x)-F(y)) \\
&\geq \sum_{i=1}^n (y_i-x_i) \left(\frac{\partial F_k(x)}{\partial x_i} +m \right) +\gamma(F_k(x)-F_k(y))- \gamma M ||x-y||_1\\
&=\sum_{i=1}^n (y_i-x_i) \left(\frac{\partial F_k(x)}{\partial x_i} +m- \gamma M \right) +\gamma(F_k(x)-F_k(y)), \forall x,y \in [0,1]^n, x \leq y.
\end{align*}
Thus, by solving \eqref{eq:approx.submod.trunc}, we obtain a lower bound on $\gamma^*(f).$
\Halmos
\endproof

\subsection{Proof of Results in Section \ref{subsec:diff.submod}} \label{appendix:proofs.diff.submod}

\proof{Proof of Proposition \ref{prop:irreducible}.}
We show the first claim of the proof, that is, for any irreducible submodular decomposition $(g, h)$ of $f$, any other decomposition $(g', h')$ that reduces to $(g, h)$ results in a greater submodular approximation of $f$ when using the same permutation $\pi$ of $\Omega$:
$$(g-\underbar{h}_\pi)(S) \leq (g'-\underbar{h}_\pi')(S), \forall S \subseteq \Omega.$$
Since $(g',h')$ can be reduced to $(g, h)$, there exists a non-modular submodular $p$ such that $g' = g + p$ and $h' = h + p$. For any $S \subseteq \Omega$ and any $\pi$ such that $S^{\pi}_{|S|}=S:$
    \begin{align*}
  (g-\underbar{h}_\pi)(S)  = (g' - p' - (\underbar{h' - p'})_\pi)(S)  = (g' - \underbar{h}_\pi')(S) + (\underbar{p}_\pi-p)(S)\leq (g' - \underbar{h}_\pi')(S)
    \end{align*}
The second equality comes from linearity of the subgradient of any submodular function. The inequality stems from the fact that $\underbar{p}_\pi$ is a subgradient of $\pi$, and therefore $\underbar{p}_\pi \leq p$. 

We now show the second claim, which states that the above no longer holds when using two different permutations $\pi_1$ and $\pi_2$ to produce sub-gradients of $h$ and $h'$, via an example. Let $f = g-h$, where $g(S) = \sum_{i \in S} a_i$ and $h(S) = \sqrt{|S|}$, for all $S \subseteq \Omega.$ Notice that it must be that $(g,h)$ is irreducible, since $g$ is modular. Now consider a reducible decomposition: $g' = g + \alpha h$ and $h' = (1+\alpha)h$. It is easy to see that $\underbar{h}_\pi(S) = \sum_i w_i^\pi$ with $w_i^\pi = \sqrt{\pi^{-1}(i)} - \sqrt{\pi^{-1}(i)-1}$ and $\pi^{-1}(i)$ being equal to the position of element $i$ in permutation $\pi$. Then:
\begin{align*}
    (g-\underbar{h}_{\pi_1})(X) = \sum_{i\in S}(a_i - w_i^{\pi_1}), \text{ and }
    (g'-\underbar{h}'_{\pi_2})(X) = \sum_{i\in S}(a_i - (1+\alpha)w_i^{\pi_1})+\alpha h(X).
\end{align*}
Consider any permutation $\pi_1$ with element 1 in the last position, that is $\pi_1^{-1}(1) = n$ and $\pi_2$ with element 1 in its first position, $\pi_2^{-1}(1) = 1$. We have that 
\begin{align*}
    (g-\underbar{h}_{\pi_1})(\{1\}) = a_1 - \left(\sqrt{n}-\sqrt{n-1}\right), \text{ and }
    (g'-\underbar{h}'_{\pi_2})(\{1\}) = a_i - (1+\alpha)1+\alpha = a_i - 1,
\end{align*}
and so $(g-\underbar{h}_{\pi_1})(\{1\}) \geq (g'-\underbar{h}'_{\pi_2})(\{1\}) ~ \forall n \in \mathbb{N}$.
\Halmos
\endproof

\proof{Proof of Proposition \ref{prop:opt.irreducible}.}
Suppose $(G^*,H^*)$ is not an irreducible decomposition. Then there exists a non-modular submodular $P \in \mathbb{R}[x]$ such that $G \mathrel{\mathop{:}}= G^* - P$ and $H \mathrel{\mathop{:}}= H^* -P$, with $F=G-H$, and $G$ and $H$ submodular. By linearity of the second derivative, we have that: $$\frac{\partial^2 H(x)}{\partial x_i\partial x_j} = \frac{\partial^2 H^*(x)}{\partial x_i\partial x_j} - \frac{\partial^2 P(x)}{\partial x_i\partial x_j}.$$ Further, due to submodularity, $\frac{\partial P(x)}{\partial x_i\partial x_j} \leq 0 ~\forall i, j \in 1, \ldots, n.$ In fact, there exists some $\bar{x} \in \{0,1\}^n, \bar{i}, \bar{j}$ such that $\frac{\partial P(\bar{x})}{\partial x_{\bar{i}}\partial x_{\bar{j}}} < 0$, since it is assumed that $P$ is non-modular.
Therefore 
\begin{align*}
    \sum_{x\in\{0,1\}^n}\sum_{i< j}\frac{\partial H(x)}{\partial x_i\partial x_j} = \sum_{x\in\{0,1\}^n}\sum_{i< j} \left( \frac{\partial H^*(x)}{\partial x_i\partial x_j} - \frac{\partial P(x)}{\partial x_i\partial x_j}\right)> \sum_{x\in\{0,1\}^n}\sum_{i< j} \frac{\partial H^*(x)}{\partial x_i\partial x_j},
\end{align*}
which is a violation of our assumption that $G^*, H^*$ is an optimal solution to \eqref{model:curv.maximization}.
\Halmos
\endproof

\proof{Proof of Proposition \ref{prop:np.hard.irreducible}.} We provide a reduction from $SUBMOD_4$ as defined in Proposition~\ref{prop:np.hardness}. Let $q$ be a set function of degree $4$ with MLE $Q.$ We let $f=q$, or equivalently, $F=Q$ and take $k=0$. We show that $Q$ is submodular if and only if $F=G-H$ for submodular $G$ and $H$ and $\sum_{x \in \{0,1\}^n} \sum_{i<j} \frac{\partial^2 H(x)}{\partial x_i \partial x_j} \geq k=0.$

If $Q$ is submodular, then let $G=Q$ and $H=0$ and the implication follows. Now suppose that there exist submodular $G$ and $H$ such that $Q=G-H$ with $\sum_{x \in \{0,1\}^n} \sum_{i<j} \frac{\partial^2 H(x)}{\partial x_i \partial x_j} \geq 0.$ As $H$ is submodular, we have that $ \frac{\partial^2 H(x)}{\partial x_i \partial x_j} \leq 0, \forall x \in \{0,1\}^n, \forall i\neq j$, following Proposition \ref{prop:submod.mle} (iii'). Combining this with the assumption, we get that $$ \frac{\partial^2 H(x)}{\partial x_i \partial x_j} = 0, \forall x \in \{0,1\}^n, \forall i\neq j.$$
This implies that $H$ is modular. Thus, 
$$\frac{\partial^2 Q(x)}{\partial x_i \partial x_j}=\frac{\partial^2 G(x)}{\partial x_i \partial x_j} \geq 0, \forall x \in \{0,1\}^n, \forall i\neq j,$$
where the equality follows from the fact that $Q=G-H$ and the inequality follows from the fact that $G$ is submodular. We obtain in consequence of Proposition \ref{prop:submod.mle} (iii') that $Q$ is submodular and the converse follows.
\Halmos
\endproof

\section{Details for Numerical Results} \label{appendix:numeric}
In this section, we provide additional details on the experimental set-up and numerical results for the experiments presented in Section \ref{sec:applications}. All experiments were executed on the ISyE High-Performance Computing cluster at Georgia Tech. We use Mosek \citep{mosek2024} as our solver, JuMP \citep{dunning2017jump}, and the SumOfSquares.jl library \citep{legat2020sumofsquares}. All code can be found \href{https://github.com/SOSSubmodularity/SOSSubmodularity_Revision}{here}.

\subsection{Details for Section \ref{subsec:regression}} \label{appendix:t.sos.regression}

\subsubsection{Details on regression methods.}  Given training data $\{x_i, y_i\}_{i=1}^m$, we solve the following problems to obtain a polynomial regressor, a $t$-sos-submodular regressor, and a polynomial regressor as in \citep{stobbe2012learning}, respectively:
\begin{align}
F_{poly}^k = &\argmin_{F \in \mathbb{R}[x], \deg(F)\leq k} \sum_{i=1}^m (y_i-F(x_i))^2 + \lambda \sum_{T \subseteq \Omega, |T| \leq k} a(T)^2, \label{eq:poly}\\
F_{t-sos-submod}^k = &\argmin_{F \in \mathbb{R}[x], \deg(F)\leq k} \sum_{i=1}^m (y_i-F(x_i))^2 + \lambda \sum_{T \subseteq \Omega, |T| \leq k} a(T)^2 \label{eq:sos.regression}\\
&\text{s.t. } F \text{ is $t$-sos submodular} \nonumber\\
F_{necessary-submod}^k = &\argmin_{F \in \mathbb{R}[x], \deg(F)\leq k} \sum_{i=1}^m (y_i-F(x_i))^2 + \lambda \sum_{T \subseteq \Omega, |T| \leq k} a(T)^2, \label{eq:stobbe.regression}\\
&\text{s.t. } a(\{i, j\}) + |a(T)| \leq 0, ~ \forall \{i, j\} \subset T, \forall T \in \Omega, |T| \leq k,  \nonumber
\end{align}
The second term in the objective functions is equivalent to a ridge regression penalty on the coefficients of the polynomial fit to the data. We solve the training problem for $\lambda \in \{0.0, 10^{-4}, 10^{-3}, 10^{-2}\}$, and select the best performing $\lambda$ using the validation set. This regularization is needed to avoid computational errors for both \eqref{eq:poly} and \eqref{eq:stobbe.regression}. We add it to \eqref{eq:sos.regression} too to ensure an apples-to-apples comparison. As explained in Section \ref{subsec:regression}, the constraints appearing in \eqref{eq:stobbe.regression} are sufficient for submodularity when $k\in \{2,3\}$ but are no longer sufficient when $k \geq 4$, in line with Proposition~\ref{prop:np.hardness}.

FlexSubNet fits a monotone submodular function $F_{FlexSubNet} = F^{N}(S;\theta)$ with the following recursive form: 
$F^n(S; \theta) = \phi_\theta(\alpha F^{n-1}(S; \theta) + (1-\alpha) m_\theta^n(S)), n = 1, \ldots, N, \text{ and } F^0(S; \theta) = m_\theta(S),$
where $m_\theta^n$ and $\phi$ are modular and concave functions that are expressed by neural networks parameterized by $\theta$. It learns $\theta$ via stochastic gradient descent on the problem
$\min_\theta \sum_{i=1}^m (y_i-F^N(x_i;\theta))^2.$
We set the hyperparameters for this model as follows: batch size to three, number of hidden layers of the neural networks to four, learning rate to 0.005, and we train for 100 epochs. All remaining hyperparameters are set to their default as stated in the original paper and code \citep{de2022neural}. Since the resulting model is sensitive to random initialization, we retrain with five different random initializations and select the best one using the validation set.

\subsubsection{Additional experiments.}
We show submodular regression results for two additional synthetic set functions following \citep{de2022neural}. We generate a log-determinantal function
$F_{LD}(S) = \log\det(I + \sum_{s\in S} z_sz_s^\top),$ and a facility location function $F_{FL}(S) = \sum_{i=1}^n \max_{s\in S}\frac{z_i^\top z_s}{\|z_i\| \|z_s\|},$
with $z_i \sim \text{unif}[0,1]^{10},\ i \in 1, \ldots, n$. We set $n = 15$, and vary additive Gaussian noise as outlined in Section \ref{subsec:regression}. The test performance of each method, averaged  over five random problem instances, is presented in Figure~\ref{fig:regression.logdet.and.facloc} for the log-determinantal and the facility location functions.

\begin{figure}
    \centering    \includegraphics[width=0.7\linewidth]{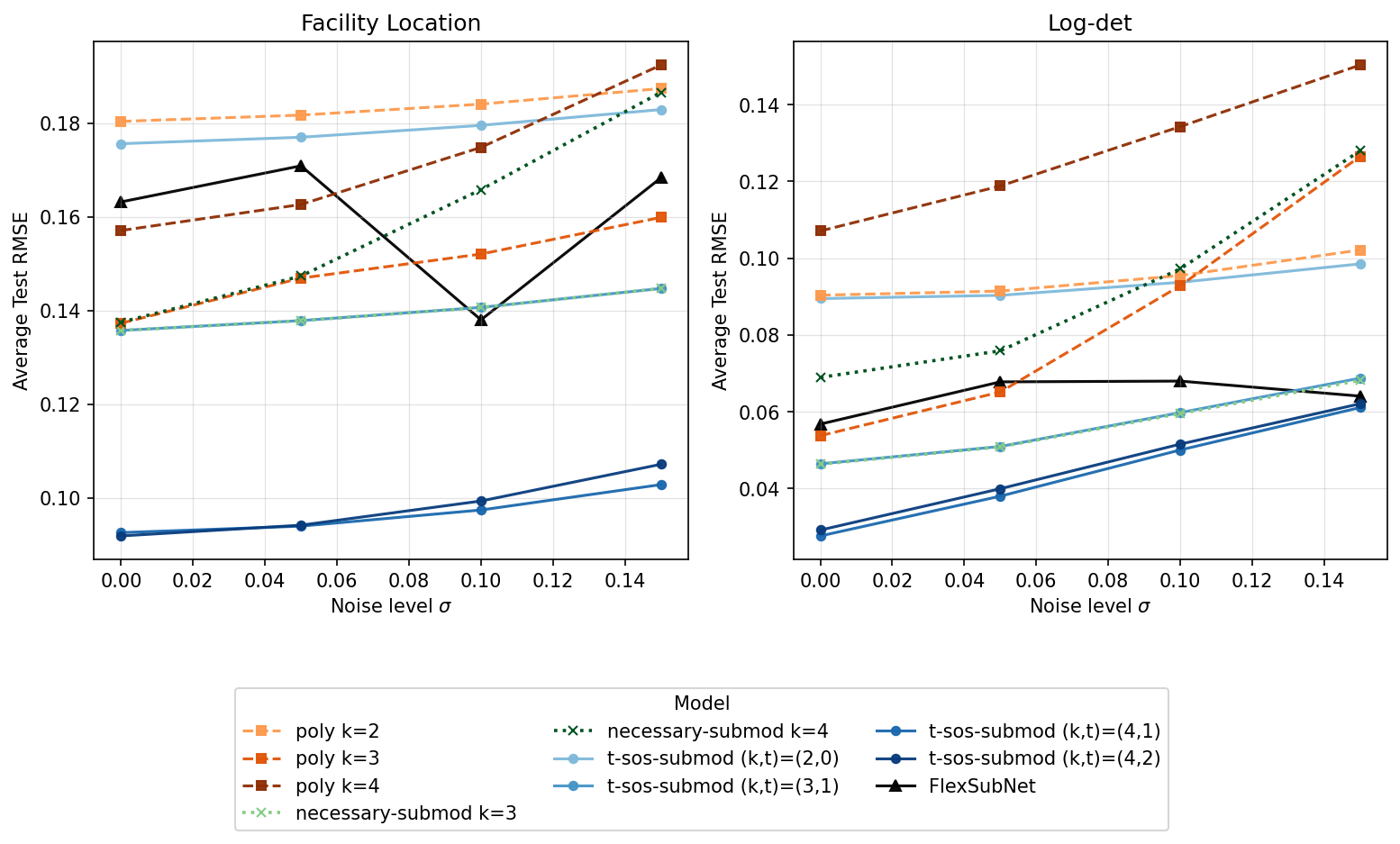}
    \caption{Comparison of root mean squared test error for various methods, fitted to noisy samples of facility location (right) and submodular log-determinant (left) functions.}
    \label{fig:regression.logdet.and.facloc}
\end{figure}

\subsubsection{Runtime.}
We report the average time to solve for $F^k_{poly}, F^k_{t-sos-submod}, F_{FlexSubNet}$ and $F_{necessary-submod}^k$ for $n=15$ while varying $k,t$ in Table \ref{tab:t.sos.regression.times}.

\begin{table}[h]
\centering
\resizebox{0.98\textwidth}{!}{
\begin{tabular}{|c|ccc|cccc|c|cc|}
\hline
Regression method & \multicolumn{3}{c|}{Polynomial ($k$)} & \multicolumn{4}{c|}{$t$-sos-submodular ($k,t$)} & FlexSubNet & \multicolumn{2}{c|}{ Nec.-submod. ($k$)} \\ \cline{1-11}
Parameters & $2$ & $3$ & $4$ & $(2,0)$ & $(3,1)$ & $(4,1)$ & $(4,2)$ & -- & 3 & 4\\ \hline
Average runtimes (s) & 5.56 & 5.90 & 7.30 & 18.26 & 20.07 & 32.68 & 5120.74 & 59.76 & 6.27 & 17.47 \\ \hline
\end{tabular}
}
\vspace{3mm}
\caption{Average training times for $F^k_{poly}$, $F^k_{t\text{-sos-submod}},F_{FlexSubNet}$ and $F_{necessary-submod}^k$ on synthetic data.}
\label{tab:t.sos.regression.times}
\end{table}

\subsection{Details for Section \ref{subsec:approx.submod}} \label{appendix:comp.submod}

We report the average time to solve for $\gamma^{t,k}_{trunc,sos}$ via \eqref{eq:approx.submod.trunc} for $n=10$ while varying $k,t$ in Table \ref{tab:approx.submod.times}.

\begin{table}[h]
    \centering
    \begin{tabular}{|c|c|c|c|c|}
    \hline
    $(k,t)$ & (1,1) & (2,2) & (3,2) & (4,3) \\
    \hline \hline
    Average runtime (s) &  1.22 & 12.5 & 13.72 & 7996\\
    \hline
    \end{tabular}
    \vspace{3mm}
    \caption{Average time to solve \eqref{eq:approx.submod.trunc} for $n=10$.
    }
    \label{tab:approx.submod.times}
\end{table}





\subsection{Details for Section \ref{subsec:diff.submod}} \label{appendix:diff.submod}
For ds decomposition, we seek to $\max \sum_{x\in\{0,1\}^n}\sum_{i<k}\frac{\partial^2H(x)}{\partial x_i\partial x_j}$, which we show how to write as a linear function in terms of the MLE coefficients $a(T):$
\begin{align*}
    \sum_{x\in\{0,1\}^n}\sum_{i<j}\frac{\partial^2H(x)}{\partial x_i\partial x_j} & = \sum_{x\in\{0,1\}^n}\sum_{i<j}\sum_{T \supseteq \{i, j\}} a(T) \prod_{k \in T\setminus\{i,j\}}x_k = \sum_{i<j}\sum_{T \supseteq \{i, j\}} a(T) \sum_{x\in\{0,1\}^n}\prod_{k \in T\setminus\{i,j\}}x_k\\
    & = \sum_{i<j}\sum_{T \supseteq \{i, j\}} a(T) 2^{n-(|T|-2)} = \sum_{T \subseteq \Omega: |T| \geq 2}a(T){|T|\choose 2} 2^{n-(|T|-2)}.
\end{align*}
Table \ref{tab:ds.decomp} shows the average optimality gap achieved over ten runs of the submodular-supermodular procedure for ten randomly generated set functions for each $n\in\{10,15,20\}$. Here we provide the disaggregated results, showing the average optimality gap attained for each individual randomly generated set function in Figure \ref{fig:ds.decomp.all}. These results show that not only does the $t$-sos-submodular-irreducible decomposition dominate on average, but also dominates each instance individually, showcasing the potential for drastic improvements. In Table~\ref{tab:ds.decomp.times} we report the average time it takes to solve \eqref{model:curv.maximization.sos} in order to obtain a $t$-sos-submodular irreducible decomposition.

\begin{figure}[h]
    \centering    \includegraphics[width=\linewidth]{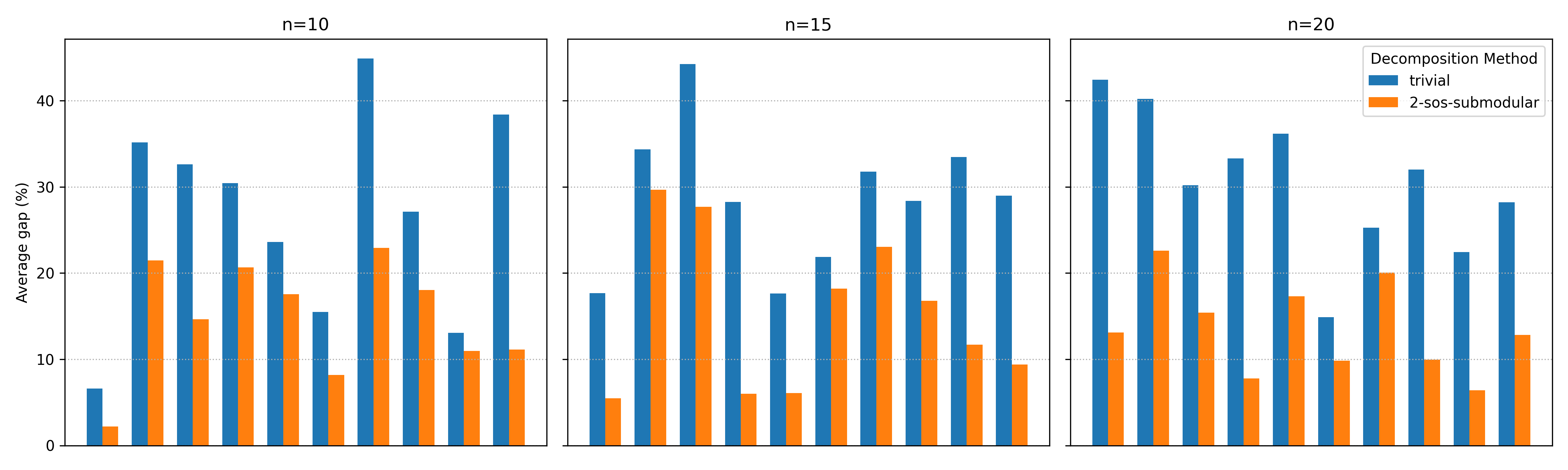}
    \caption{Comparison of average relative optimality gap obtained for ten runs of the submodular-supermodular procedure for the trivial decomposition and 2-sos-submodular-irreducible decomposition for each randomly generated set function in Section \ref{subsec:diff.submod}}
\label{fig:ds.decomp.all}
\end{figure}

\begin{table}[h]
    \centering
    \begin{tabular}{|c|c|c|c|}
    \hline
     & $n=10$ & $n=15$ & $n=20$ \\
    \hline \hline
    Average runtime (s) & 49.2 & 1,493 & 28,476 \\
    \hline
    \end{tabular}
    \vspace{3mm}
    \caption{Average time to solve \eqref{model:curv.maximization.sos} with $d=4$, $t=2$ for randomly generated set functions.
    }
    \label{tab:ds.decomp.times}
\end{table}

\vspace{-2mm}

%% file: Sos_submodularity.bbl
\begin{thebibliography}{62}
\providecommand{\natexlab}[1]{#1}
\providecommand{\url}[1]{\texttt{#1}}
\expandafter\ifx\csname urlstyle\endcsname\relax
  \providecommand{\doi}[1]{doi: #1}\else
  \providecommand{\doi}{doi: \begingroup \urlstyle{rm}\Url}\fi

\bibitem[Ahmadi and Hall(2018)]{ahmadi2018dc}
Amir~Ali Ahmadi and Georgina Hall.
\newblock {DC} decomposition of nonconvex polynomials with algebraic
  techniques.
\newblock \emph{Mathematical Programming}, 169:\penalty0 69--94, 2018.

\bibitem[Ahmadi and Parrilo(2013)]{ahmadi2013complete}
Amir~Ali Ahmadi and Pablo~A Parrilo.
\newblock A complete characterization of the gap between convexity and
  sos-convexity.
\newblock \emph{SIAM Journal on Optimization}, 23\penalty0 (2):\penalty0
  811--833, 2013.

\bibitem[Ahmadi et~al.(2013)Ahmadi, Olshevsky, Parrilo, and
  Tsitsiklis]{ahmadi2013np}
Amir~Ali Ahmadi, Alex Olshevsky, Pablo~A Parrilo, and John~N Tsitsiklis.
\newblock {NP}-hardness of deciding convexity of quartic polynomials and
  related problems.
\newblock \emph{Mathematical Programming}, 137:\penalty0 453--476, 2013.

\bibitem[Atamt{\"u}rk and G{\'o}mez(2020)]{atamturk2020submodularity}
Alper Atamt{\"u}rk and Andr{\'e}s G{\'o}mez.
\newblock Submodularity in conic quadratic mixed 0--1 optimization.
\newblock \emph{Operations Research}, 68\penalty0 (2):\penalty0 609--630, 2020.

\bibitem[Bach(2019)]{bach2019submodular}
Francis Bach.
\newblock Submodular functions: from discrete to continuous domains.
\newblock \emph{Mathematical Programming}, 175:\penalty0 419--459, 2019.

\bibitem[Bach et~al.(2013)]{bach2013learning}
Francis Bach et~al.
\newblock Learning with submodular functions: A convex optimization
  perspective.
\newblock \emph{Foundations and Trends{\textregistered} in Machine Learning},
  6\penalty0 (2-3):\penalty0 145--373, 2013.

\bibitem[Balcan and Harvey(2011)]{balcan2011learning}
Maria-Florina Balcan and Nicholas~JA Harvey.
\newblock Learning submodular functions.
\newblock In \emph{Proceedings of the forty-third annual ACM symposium on
  Theory of computing}, pages 793--802, 2011.

\bibitem[Balcan and Harvey(2018)]{balcan2018submodular}
Maria-Florina Balcan and Nicholas~JA Harvey.
\newblock Submodular functions: Learnability, structure, and optimization.
\newblock \emph{SIAM Journal on Computing}, 47\penalty0 (3):\penalty0 703--754,
  2018.

\bibitem[Bian et~al.(2017)Bian, Buhmann, Krause, and
  Tschiatschek]{bian2017guarantees}
Andrew~An Bian, Joachim~M Buhmann, Andreas Krause, and Sebastian Tschiatschek.
\newblock Guarantees for greedy maximization of non-submodular functions with
  applications.
\newblock In \emph{International Conference on Machine Learning}, pages
  498--507. PMLR, 2017.

\bibitem[Bian et~al.(2020)Bian, Buhmann, and Krause]{bian2020continuous}
Yatao Bian, Joachim~M Buhmann, and Andreas Krause.
\newblock Continuous submodular function maximization.
\newblock \emph{arXiv preprint arXiv:2006.13474}, 2020.

\bibitem[Billionnet and Minoux(1985)]{billionnet1985maximizing}
Alain Billionnet and Michel Minoux.
\newblock Maximizing a supermodular pseudoboolean function: A polynomial
  algorithm for supermodular cubic functions.
\newblock \emph{Discrete Applied Mathematics}, 12\penalty0 (1):\penalty0 1--11,
  1985.

\bibitem[Bilmes(2022)]{bilmes2022submodularity}
Jeff Bilmes.
\newblock Submodularity in machine learning and artificial intelligence.
\newblock \emph{arXiv preprint arXiv:2202.00132}, 2022.

\bibitem[Blekherman et~al.(2012)Blekherman, Parrilo, and
  Thomas]{blekherman2012semidefinite}
Grigoriy Blekherman, Pablo~A Parrilo, and Rekha~R Thomas.
\newblock \emph{Semidefinite {O}ptimization and {C}onvex {A}lgebraic
  {G}eometry}.
\newblock SIAM, 2012.

\bibitem[Bomze and Locatelli(2004)]{bomze2004undominated}
Immanuel~M Bomze and Marco Locatelli.
\newblock Undominated {DC} decompositions of quadratic functions and
  applications to branch-and-bound approaches.
\newblock \emph{Computational Optimization and Applications}, 28\penalty0
  (2):\penalty0 227--245, 2004.

\bibitem[Boyd and Vandenberghe(2004)]{boyd2004convex}
Stephen~P Boyd and Lieven Vandenberghe.
\newblock \emph{Convex {O}ptimization}.
\newblock Cambridge university press, 2004.

\bibitem[Brandenburg et~al.(2024)Brandenburg, Grillo, and
  Hertrich]{brandenburg2024decomposition}
Marie-Charlotte Brandenburg, Moritz Grillo, and Christoph Hertrich.
\newblock Decomposition polyhedra of piecewise linear functions.
\newblock \emph{arXiv preprint arXiv:2410.04907}, 2024.

\bibitem[Burer and Natarajan(2025)]{burer2025semidefinite}
Samuel Burer and Karthik Natarajan.
\newblock On the semidefinite representability of continuous quadratic
  submodular minimization with applications to moment problems.
\newblock \emph{arXiv preprint arXiv:2504.03996}, 2025.

\bibitem[Chatterjee et~al.(2015)Chatterjee, Guntuboyina, and
  Sen]{chatterjee2015risk}
Sabyasachi Chatterjee, Adityanand Guntuboyina, and Bodhisattva Sen.
\newblock On risk bounds in isotonic and other shape restricted regression
  problems.
\newblock \emph{Annals of Statistics}, 43\penalty0 (4):\penalty0 1774--1800,
  2015.

\bibitem[Crama(1989)]{crama1989recognition}
Yves Crama.
\newblock Recognition problems for special classes of polynomials in 0--1
  variables.
\newblock \emph{Mathematical programming}, 44:\penalty0 139--155, 1989.

\bibitem[Crama and Hammer(2011)]{crama2011boolean}
Yves Crama and Peter~L Hammer.
\newblock \emph{Boolean {F}unctions: {T}heory, {A}lgorithms, and
  {A}pplications}.
\newblock Cambridge University Press, 2011.

\bibitem[Curmei and Hall(2025)]{curmei2025shape}
Mihaela Curmei and Georgina Hall.
\newblock Shape-constrained regression using sum of squares polynomials.
\newblock \emph{Operations Research}, 73\penalty0 (1):\penalty0 543--559, 2025.

\bibitem[Das and Kempe(2018)]{das2018approximate}
Abhimanyu Das and David Kempe.
\newblock Approximate submodularity and its applications: Subset selection,
  sparse approximation and dictionary selection.
\newblock \emph{Journal of Machine Learning Research}, 19\penalty0
  (3):\penalty0 1--34, 2018.

\bibitem[De and Chakrabarti(2022)]{de2022neural}
Abir De and Soumen Chakrabarti.
\newblock Neural estimation of submodular functions with applications to
  differentiable subset selection.
\newblock \emph{Advances in Neural Information Processing Systems},
  35:\penalty0 19537--19552, 2022.

\bibitem[Dolhansky and Bilmes(2016)]{dolhansky2016deep}
Brian~W Dolhansky and Jeff~A Bilmes.
\newblock Deep submodular functions: Definitions and learning.
\newblock \emph{Advances in Neural Information Processing Systems}, 29, 2016.

\bibitem[Dunning et~al.(2017)Dunning, Huchette, and Lubin]{dunning2017jump}
Iain Dunning, Joey Huchette, and Miles Lubin.
\newblock Jump: A modeling language for mathematical optimization.
\newblock \emph{SIAM Review}, 59\penalty0 (2):\penalty0 295--320, 2017.

\bibitem[Edmonds(2003)]{edmonds2003submodular}
Jack Edmonds.
\newblock Submodular functions, matroids, and certain polyhedra.
\newblock In \emph{Combinatorial Optimization—Eureka, You Shrink! Papers
  Dedicated to Jack Edmonds 5th International Workshop Aussois, France, March
  5--9, 2001 Revised Papers}, pages 11--26. Springer, 2003.

\bibitem[El~Halabi et~al.(2023)El~Halabi, Orfanides, and
  Hoheisel]{el2023difference}
Marwa El~Halabi, George Orfanides, and Tim Hoheisel.
\newblock Difference of submodular minimization via {DC} programming.
\newblock In \emph{International Conference on Machine Learning}, pages
  9172--9201. PMLR, 2023.

\bibitem[Feldman and Vondr{\'a}k(2015)]{feldman2015tight}
Vitaly Feldman and Jan Vondr{\'a}k.
\newblock Tight bounds on low-degree spectral concentration of submodular and
  xos functions.
\newblock In \emph{2015 IEEE 56th Annual Symposium on Foundations of Computer
  Science}, pages 923--942. IEEE, 2015.

\bibitem[Feldman and Vondr{\'a}k(2016)]{feldman2016optimal}
Vitaly Feldman and Jan Vondr{\'a}k.
\newblock Optimal bounds on approximation of submodular and xos functions by
  juntas.
\newblock \emph{SIAM Journal on Computing}, 45\penalty0 (3):\penalty0
  1129--1170, 2016.

\bibitem[Feldman et~al.(2013)Feldman, Kothari, and
  Vondr{\'a}k]{feldman2013representation}
Vitaly Feldman, Pravesh Kothari, and Jan Vondr{\'a}k.
\newblock Representation, approximation and learning of submodular functions
  using low-rank decision trees.
\newblock In \emph{Conference on Learning Theory}, pages 711--740. PMLR, 2013.

\bibitem[Frank(1993)]{frank1993submodular}
Andr{\'a}s Frank.
\newblock Submodular functions in graph theory.
\newblock \emph{Discrete Mathematics}, 111\penalty0 (1-3):\penalty0 231--243,
  1993.

\bibitem[Fujishige(2005)]{fujishige2005submodular}
Satoru Fujishige.
\newblock \emph{Submodular {F}unctions and {O}ptimization}, volume~58.
\newblock Elsevier, 2005.

\bibitem[Gallo and Simeone(1989)]{gallo1989supermodular}
Giorgio Gallo and Bruno Simeone.
\newblock On the supermodular knapsack problem.
\newblock \emph{Mathematical Programming}, 45:\penalty0 295--309, 1989.

\bibitem[Goemans et~al.(2009)Goemans, Harvey, Iwata, and
  Mirrokni]{goemans2009approximating}
Michel~X Goemans, Nicholas~JA Harvey, Satoru Iwata, and Vahab Mirrokni.
\newblock Approximating submodular functions everywhere.
\newblock In \emph{Proceedings of the twentieth annual ACM-SIAM symposium on
  Discrete algorithms}, pages 535--544. SIAM, 2009.

\bibitem[Goujaud et~al.(2024)Goujaud, Moucer, Glineur, Hendrickx, Taylor, and
  Dieuleveut]{goujaud2024pepit}
Baptiste Goujaud, C{\'e}line Moucer, Fran{\c{c}}ois Glineur, Julien~M
  Hendrickx, Adrien~B Taylor, and Aymeric Dieuleveut.
\newblock {PEP}it: computer-assisted worst-case analyses of first-order
  optimization methods in python.
\newblock \emph{Mathematical Programming Computation}, 16\penalty0
  (3):\penalty0 337--367, 2024.

\bibitem[Grabisch et~al.(2000)Grabisch, Marichal, and
  Roubens]{grabisch2000equivalent}
Michel Grabisch, Jean-Luc Marichal, and Marc Roubens.
\newblock Equivalent representations of set functions.
\newblock \emph{Mathematics of Operations Research}, 25\penalty0 (2):\penalty0
  157--178, 2000.

\bibitem[Helton and Nie(2010)]{helton2010semidefinite}
J~William Helton and Jiawang Nie.
\newblock Semidefinite representation of convex sets.
\newblock \emph{Mathematical Programming}, 122:\penalty0 21--64, 2010.

\bibitem[Horel and Singer(2016)]{horel2016maximization}
Thibaut Horel and Yaron Singer.
\newblock Maximization of approximately submodular functions.
\newblock \emph{Advances in neural information processing systems}, 29, 2016.

\bibitem[Iyer and Bilmes(2012)]{iyer2012algorithms}
Rishabh Iyer and Jeff Bilmes.
\newblock Algorithms for approximate minimization of the difference between
  submodular functions, with applications.
\newblock \emph{arXiv preprint arXiv:1207.0560}, 2012.

\bibitem[Jegelka and Bilmes(2011)]{jegelka2011submodularity}
Stefanie Jegelka and Jeff Bilmes.
\newblock Submodularity beyond submodular energies: coupling edges in graph
  cuts.
\newblock In \emph{Proc. 2011 IEEE Conf. on Computer Vision and Pattern
  Recognition}, pages 1897--1904. IEEE, 2011.

\bibitem[Krause et~al.(2008)Krause, Singh, and Guestrin]{krause2008near}
Andreas Krause, Ajit Singh, and Carlos Guestrin.
\newblock Near-optimal sensor placements in gaussian processes: Theory,
  efficient algorithms and empirical studies.
\newblock \emph{Journal of Machine Learning Research}, 9\penalty0 (2), 2008.

\bibitem[Laurent and Slot(2024)]{laurent2024overview}
Monique Laurent and Lucas Slot.
\newblock An overview of convergence rates for sum of squares hierarchies in
  polynomial optimization.
\newblock \emph{arXiv preprint arXiv:2408.04417}, 2024.

\bibitem[Legat(2020)]{legat2020sumofsquares}
Beno{\^\i}t Legat.
\newblock Sumofsquares.jl: Sum-of-squares optimization in julia.
\newblock \url{https://github.com/jump-dev/SumOfSquares.jl}, 2020.

\bibitem[Lin and Bilmes(2012)]{lin2012learning}
Hui Lin and Jeff~A Bilmes.
\newblock Learning mixtures of submodular shells with application to document
  summarization.
\newblock \emph{arXiv preprint arXiv:1210.4871}, 2012.

\bibitem[Majumdar et~al.(2020)Majumdar, Hall, and Ahmadi]{majumdar2020recent}
Anirudha Majumdar, Georgina Hall, and Amir~Ali Ahmadi.
\newblock Recent scalability improvements for semidefinite programming with
  applications in machine learning, control, and robotics.
\newblock \emph{Annual Review of Control, Robotics, and Autonomous Systems},
  3\penalty0 (1):\penalty0 331--360, 2020.

\bibitem[Narasimhan and Bilmes(2005)]{narasimhan2005submodular}
Mukund Narasimhan and Jeff Bilmes.
\newblock A submodular-supermodular procedure with applications to
  discriminative structure learning.
\newblock In \emph{Proc. 21st Ann. Conf. on Uncertainty in Artificial
  Intelligence (UAI'05)}, pages 404--412. AUAI Press, 2005.

\bibitem[Narayanan(1997)]{narayanan1997submodular}
Hariharan Narayanan.
\newblock \emph{Submodular {F}unctions and {E}lectrical {N}etworks}, volume~54.
\newblock Elsevier, 1997.

\bibitem[Nemhauser et~al.(1978)Nemhauser, Wolsey, and
  Fisher]{nemhauser1978analysis}
George~L Nemhauser, Laurence~A Wolsey, and Marshall~L Fisher.
\newblock An analysis of approximations for maximizing submodular set
  functions—i.
\newblock \emph{Mathematical programming}, 14:\penalty0 265--294, 1978.

\bibitem[O'Donoghue et~al.(2016)O'Donoghue, Chu, Parikh, and
  Boyd]{odowd2021scs}
Brendan O'Donoghue, Eric Chu, Neal Parikh, and Stephen Boyd.
\newblock Conic optimization via operator splitting and homogeneous self-dual
  embedding.
\newblock \emph{Journal of Optimization Theory and Applications}, 169\penalty0
  (3):\penalty0 1042--1068, 2016.

\bibitem[Owen(1972)]{owen1972multilinear}
Guillermo Owen.
\newblock Multilinear extensions of games.
\newblock \emph{Management Science}, 18\penalty0 (5-part-2):\penalty0 64--79,
  1972.

\bibitem[Punnen(2022)]{punnen2022quadratic}
Abraham~P Punnen.
\newblock The quadratic unconstrained binary optimization problem.
\newblock \emph{Springer International Publishing}, 10:\penalty0 978--3, 2022.

\bibitem[Queyranne and Schulz(1995)]{queyranne1995scheduling}
Maurice Queyranne and Andreas~S Schulz.
\newblock Scheduling unit jobs with compatible release dates on parallel
  machines with nonstationary speeds.
\newblock In \emph{International Conference on Integer Programming and
  Combinatorial Optimization}, pages 307--320. Springer, 1995.

\bibitem[Rubinstein and Singla(2017)]{rubinstein2017combinatorial}
Aviad Rubinstein and Sahil Singla.
\newblock Combinatorial prophet inequalities.
\newblock In \emph{Proceedings of the Twenty-Eighth Annual ACM-SIAM Symposium
  on Discrete Algorithms}, pages 1671--1687. SIAM, 2017.

\bibitem[Schrijver(2000)]{schrijver2000combinatorial}
Alexander Schrijver.
\newblock A combinatorial algorithm minimizing submodular functions in strongly
  polynomial time.
\newblock \emph{Journal of Combinatorial Theory, Series B}, 80\penalty0
  (2):\penalty0 346--355, 2000.

\bibitem[Seshadhri and Vondr{\'a}k(2014)]{seshadhri2014submodularity}
Comandur Seshadhri and Jan Vondr{\'a}k.
\newblock Is submodularity testable?
\newblock \emph{Algorithmica}, 69\penalty0 (1):\penalty0 1--25, 2014.

\bibitem[Sipos et~al.(2012)Sipos, Shivaswamy, and Joachims]{sipos2012large}
Ruben Sipos, Pannaga Shivaswamy, and Thorsten Joachims.
\newblock Large-margin learning of submodular summarization models.
\newblock In \emph{Proceedings of the 13th Conference of the European Chapter
  of the Association for Computational Linguistics}, pages 224--233, 2012.

\bibitem[Sipser(1996)]{sipser1996introduction}
Michael Sipser.
\newblock Introduction to the {T}heory of {C}omputation.
\newblock \emph{ACM Sigact News}, 27\penalty0 (1):\penalty0 27--29, 1996.

\bibitem[Stobbe and Krause(2012)]{stobbe2012learning}
Peter Stobbe and Andreas Krause.
\newblock Learning {F}ourier sparse set functions.
\newblock In \emph{Artificial Intelligence and Statistics}, pages 1125--1133.
  PMLR, 2012.

\bibitem[Topkis(1978)]{topkis1978minimizing}
Donald~M Topkis.
\newblock Minimizing a submodular function on a lattice.
\newblock \emph{Operations research}, 26\penalty0 (2):\penalty0 305--321, 1978.

\bibitem[Topkis(1998)]{topkis1998supermodularity}
Donald~M Topkis.
\newblock \emph{Supermodularity and {C}omplementarity}.
\newblock Princeton university press, 1998.

\bibitem[Tschiatschek et~al.(2014)Tschiatschek, Iyer, Wei, and
  Bilmes]{tschiatschek2014learning}
Sebastian Tschiatschek, Rishabh~K Iyer, Haochen Wei, and Jeff~A Bilmes.
\newblock Learning mixtures of submodular functions for image collection
  summarization.
\newblock \emph{Advances in neural information processing systems}, 27, 2014.

\bibitem[Yuille and Rangarajan(2003)]{yuille2003concave}
Alan~L Yuille and Anand Rangarajan.
\newblock The concave-convex procedure.
\newblock \emph{Neural computation}, 15\penalty0 (4):\penalty0 915--936, 2003.

\end{thebibliography}

\begin{thebibliography}{15}
\providecommand{\natexlab}[1]{#1}
\providecommand{\url}[1]{\texttt{#1}}
\expandafter\ifx\csname urlstyle\endcsname\relax
  \providecommand{\doi}[1]{doi: #1}\else
  \providecommand{\doi}{doi: \begingroup \urlstyle{rm}\Url}\fi

\bibitem[Ahmadi and Parrilo(2013)]{ahmadi2013complete}
Amir~Ali Ahmadi and Pablo~A Parrilo.
\newblock A complete characterization of the gap between convexity and
  sos-convexity.
\newblock \emph{SIAM Journal on Optimization}, 23\penalty0 (2):\penalty0
  811--833, 2013.

\bibitem[Bian et~al.(2017)Bian, Buhmann, Krause, and
  Tschiatschek]{bian2017guarantees}
Andrew~An Bian, Joachim~M Buhmann, Andreas Krause, and Sebastian Tschiatschek.
\newblock Guarantees for greedy maximization of non-submodular functions with
  applications.
\newblock In \emph{International Conference on Machine Learning}, pages
  498--507. PMLR, 2017.

\bibitem[Billionnet and Minoux(1985)]{billionnet1985maximizing}
Alain Billionnet and Michel Minoux.
\newblock Maximizing a supermodular pseudoboolean function: A polynomial
  algorithm for supermodular cubic functions.
\newblock \emph{Discrete Applied Mathematics}, 12\penalty0 (1):\penalty0 1--11,
  1985.

\bibitem[Blekherman et~al.(2012)Blekherman, Parrilo, and
  Thomas]{blekherman2012semidefinite}
Grigoriy Blekherman, Pablo~A Parrilo, and Rekha~R Thomas.
\newblock \emph{Semidefinite {O}ptimization and {C}onvex {A}lgebraic
  {G}eometry}.
\newblock SIAM, 2012.

\bibitem[De and Chakrabarti(2022)]{de2022neural}
Abir De and Soumen Chakrabarti.
\newblock Neural estimation of submodular functions with applications to
  differentiable subset selection.
\newblock \emph{Advances in Neural Information Processing Systems},
  35:\penalty0 19537--19552, 2022.

\bibitem[Dunning et~al.(2017)Dunning, Huchette, and Lubin]{dunning2017jump}
Iain Dunning, Joey Huchette, and Miles Lubin.
\newblock Jump: A modeling language for mathematical optimization.
\newblock \emph{SIAM Review}, 59\penalty0 (2):\penalty0 295--320, 2017.

\bibitem[Fawzi et~al.(2015)Fawzi, Saunderson, and Parrilo]{fawzi2015sparse}
Hamza Fawzi, James Saunderson, and Pablo~A Parrilo.
\newblock Sparse sum-of-squares certificates on finite abelian groups.
\newblock In \emph{2015 54th IEEE Conference on Decision and Control (CDC)},
  pages 5909--5914. IEEE, 2015.

\bibitem[Grabisch et~al.(2000)Grabisch, Marichal, and
  Roubens]{grabisch2000equivalent}
Michel Grabisch, Jean-Luc Marichal, and Marc Roubens.
\newblock Equivalent representations of set functions.
\newblock \emph{Mathematics of Operations Research}, 25\penalty0 (2):\penalty0
  157--178, 2000.

\bibitem[Laurent(2003)]{laurent2003lower}
Monique Laurent.
\newblock Lower bound for the number of iterations in semidefinite hierarchies
  for the cut polytope.
\newblock \emph{Mathematics of operations research}, 28\penalty0 (4):\penalty0
  871--883, 2003.

\bibitem[Laurent(2009)]{laurent2009sums}
Monique Laurent.
\newblock Sums of squares, moment matrices and optimization over polynomials.
\newblock In \emph{Emerging applications of algebraic geometry}, pages
  157--270. Springer, 2009.

\bibitem[Legat(2020)]{legat2020sumofsquares}
Beno{\^\i}t Legat.
\newblock Sumofsquares.jl: Sum-of-squares optimization in julia.
\newblock \url{https://github.com/jump-dev/SumOfSquares.jl}, 2020.

\bibitem[{MOSEK ApS}(2024)]{mosek2024}
{MOSEK ApS}.
\newblock The mosek optimization toolbox for matlab manual. version 10.0.
\newblock \url{https://docs.mosek.com/}, 2024.

\bibitem[Permenter and Parrilo(2012)]{permenter2012selecting}
Frank Permenter and Pablo~A Parrilo.
\newblock Selecting a monomial basis for sums of squares programming over a
  quotient ring.
\newblock In \emph{2012 IEEE 51st IEEE Conference on Decision and Control
  (CDC)}, pages 1871--1876. IEEE, 2012.

\bibitem[Sakaue et~al.(2016)Sakaue, Takeda, Kim, and Ito]{sakaue2016exact}
Shinsaku Sakaue, Akiko Takeda, Sunyoung Kim, and Naoki Ito.
\newblock Exact {SDP} relaxations with truncated moment matrix for binary
  polynomial optimization problems.
\newblock In \emph{Proceedings of the 5th International Conference on
  Continuous Optimization ICCOPT}, 2016.

\bibitem[Stobbe and Krause(2012)]{stobbe2012learning}
Peter Stobbe and Andreas Krause.
\newblock Learning {F}ourier sparse set functions.
\newblock In \emph{Artificial Intelligence and Statistics}, pages 1125--1133.
  PMLR, 2012.

\end{thebibliography}
